\documentclass{amsart}
\calclayout
\usepackage{amssymb,amsmath,amsfonts,epsfig,latexsym,tikz}
\usepackage{tikz-cd}
\usepackage{hyperref}
\usepackage{enumerate}
\usepackage{mathtools}
\usepackage{verbatim}
\usepackage{cleveref}
\usepackage[shortlabels]{enumitem}
\usepackage{subcaption}
\usepackage{stmaryrd}

\usetikzlibrary{positioning}
\usetikzlibrary{matrix}
\usetikzlibrary{decorations}
\usetikzlibrary{decorations.pathreplacing, decorations.pathmorphing, angles,quotes}
 
 \newtheorem{theorem}{Theorem}[section]
\newtheorem{definition}[theorem]{Definition}
\newtheorem{proposition}[theorem]{Proposition}
\newtheorem{lemma}[theorem]{Lemma}
\newtheorem{corollary}[theorem]{Corollary}
\newtheorem{conjecture}[theorem]{Conjecture}

\theoremstyle{definition}
\newtheorem{remark}[theorem]{Remark}
\newtheorem{question}[theorem]{Question}
\newtheorem{example}[theorem]{Example}

\newtheorem{innercustomconj}{Conjecture}
\newenvironment{customconj}[1]
  {\renewcommand\theinnercustomconj{#1}\innercustomconj}
  {\endinnercustomconj}

\def\C{\mathbb{C}}

\def\N{\mathbb{N}}
\def\P{\mathbb{P}}
\def\Q{\mathbb{Q}} 
\def\R{\mathbb{R}}
\def\T{\mathcal{T}}
\def\Z{\mathbb{Z}}
\def\k{\Bbbk}
\def\m{\mathfrak{m}}
\def\tM{\widetilde{M}}
\def\tN{\widetilde{N}}

\def\1{\mathbf{1}}

\def\cS{\mathcal{S}}

\def\trho{\widetilde{\rho}}

\def\<{\langle}
\def\>{\rangle}

\DeclareMathOperator{\BBox}{Box}

\newcommand{\Conv}[1]{\operatorname{Conv}\left\{{#1}\right\}}
\DeclareMathOperator{\Ehr}{Ehr}
\DeclareMathOperator{\Hom}{Hom}
\DeclareMathOperator{\interior}{int}
\DeclareMathOperator{\lk}{lk}
\DeclareMathOperator{\mult}{mult}
\DeclareMathOperator{\ord}{ord}

\DeclareMathOperator{\reg}{reg}
\DeclareMathOperator{\Star}{Star}
\DeclareMathOperator{\supp}{supp}
\DeclareMathOperator{\Sym}{Sym}

\DeclareMathOperator{\Newt}{Newt}
\DeclareMathOperator{\pr}{pr}
\DeclareMathOperator{\triv}{triv}
\DeclareMathOperator{\Vol}{Vol}
\DeclareMathOperator{\Sec}{Sec}
\DeclareMathOperator{\Spec}{Spec}

\DeclareMathOperator{\GL}{GL}
\DeclareMathOperator{\Aff}{Aff}
\DeclareMathOperator{\id}{id}
\DeclareMathOperator{\EHR}{EHR}
\DeclareMathOperator{\Tor}{Tor}
\DeclareMathOperator{\ver}{Vert}
\DeclareMathOperator{\lcm}{lcm}
\DeclareMathOperator{\EE}{EE}
\DeclareMathOperator{\tr}{tr}
\DeclareMathOperator{\PA}{PA}
\DeclareMathOperator{\PL}{PL}
\DeclareMathOperator{\HH}{H}
\DeclareMathOperator{\Ind}{Ind}

\DeclareMathOperator{\ev}{ev}
\DeclareMathOperator{\Stab}{Stab}
\DeclareMathOperator{\Gr}{Gr}
\DeclareMathOperator{\Proj}{Proj}
\DeclareMathOperator{\init}{init}
\DeclareMathOperator{\Aut}{Aut}
\DeclareMathOperator{\HT}{pr}
\DeclareMathOperator{\Irr}{Irr}
\DeclareMathOperator{\Int}{Int}
\DeclareMathOperator{\Pyr}{Pyr}
\DeclareMathOperator{\codeg}{codeg}
\DeclareMathOperator{\Ceil}{Ceil}
\DeclareMathOperator{\UH}{UH}
\DeclareMathOperator{\Supp}{Supp}
\DeclareMathOperator{\SR}{SR}
\DeclareMathOperator{\Hilb}{Hilb}

\newcommand\nullset\varnothing

\newcommand{\alan}[1]{{ \sf $\diamondsuit\diamondsuit\diamondsuit$ {\textcolor{red}{Alan: [#1]}}}}

\begin{document}
	
\title[Equivariant Ehrhart theory]{Equivariant Ehrhart theory, commutative algebra and invariant triangulations of polytopes} 

    
\author{Alan Stapledon}

%

\address{Sydney Mathematics Research Institute, L4.42, Quadrangle A14, University of Sydney, NSW 2006, Australia}
\email{astapldn@gmail.com}


%
%

\begin{abstract}
	Ehrhart theory is the study of the enumeration of lattice points in lattice polytopes. Equivariant Ehrhart theory is a generalization of Ehrhart theory that takes into account the
	action of a finite group acting via affine transformations on the underlying lattice and  preserving the polytope. We further develop equivariant Ehrhart theory in part by establishing
	connections with commutative algebra as well as the question of when there  exists an invariant lattice triangulation of a lattice polytope. 
\end{abstract}

\maketitle

\vspace{-20 pt}

\setcounter{tocdepth}{1}
\tableofcontents

\section{Introduction}

The presence of a group action can lead to subtle questions. 
 For example,
given a lattice polytope $P$, there always exists a regular lattice triangulation of $P$. On the other hand, if a group $G$ acts on $P$ via an affine representation, then it is a subtle question to determine whether there exists a $G$-invariant regular lattice triangulation of $P$. 
As another example, in Ehrhart theory, one associates to $P$ an invariant  $h^*(P;t)$, which is a polynomial with nonnegative integer coefficients. If  a group $G$ acts on $P$  via an affine representation $\rho$, then there exists a natural invariant $h^*(P,\rho;t)$ whose coefficients are virtual representations of $G$ and which equals $h^*(P;t)$ when the group action is trivial. It is a subtle question to determine  whether the coefficients of 
$h^*(P,\rho;t)$ have nonnegative coefficients when expressed as sums of the classes of irreducible representations. The goal of this paper is to
 study 
and relate several such problems.

\emph{Ehrhart theory} is the study of the enumeration of lattice points in lattice polytopes. \emph{Equivariant Ehrhart theory} is a generalization of Ehrhart theory that takes into account the
 action of a finite group acting via affine transformations on the underlying lattice and  preserving the polytope.
It
was introduced in \cite{StapledonEquivariant} and studied in \cite{ASV20,ASV21,EKS22,CHK23}. 
Before stating our results, we first recall some background on  equivariant Ehrhart theory. Let $M \cong \Z^d$ be a lattice of rank $d$ 
and write $M_k := M \otimes_\Z k$ for a field $k$. Let $P \subset M_\R$ be a $d$-dimensional lattice polytope. Let $\Aff(M)$ be the group of affine transformations of $M$ and let $\rho: G \to \Aff(M)$ be an affine representation of $G$ (see Section~\ref{ss:affine} and Section~\ref{ss:reptheorybasics})). We assume that $P$ is $G$-invariant. Let $R(G)$ be the complex representation ring of $G$
(see Section~\ref{ss:reptheorybasics}). Then $R(G)$ is the free $\Z$-module with basis given by the classes $[V]$ in $R(G)$, as $V$ varies over  all irreducible finite-dimensional $\C G$-modules, and an element of $R(G)$ 
with nonnegative integer coefficients with respect to this basis is called \emph{effective}. We say that a power series 
in $R(G)[[t]]$ is effective if all its coefficients  are effective.

Given a positive integer $m$,  let $L(P,\rho;m)$ in  $R(G)$ be the class of the permutation representation of $G$ acting on $P \cap \frac{1}{m}M$. Then $L(P,\rho;m)$ is a quasi-polynomial of degree $d$ with coefficients in $R(G) \otimes_\Z \Q$ \cite[Theorem~5.7]{StapledonEquivariant} that precisely encodes the Ehrhart quasi-polynomials of the rational polytopes $P^g = \{ x \in P : g \cdot x = x \} \subset M_\R$ for $g$ in $G$ \cite[Lemma~5.2]{StapledonEquivariant}. This is alternatively encoded via a power series 
$h^*(P,\rho;t)$ in $R(G)[[t]]$ defined by the equality (see Definition~\ref{d:seriesdefs}):
\begin{equation*}
	\Ehr(P,\rho;t) := 1 + \sum_{m > 0} L(P,\rho;m) t^m  = \frac{h^*(P,\rho;t)}{\det(I - \tM_\C t)}. 
\end{equation*}
Above, 
we let $\tM = M \oplus \Z$ and consider $\Aff(M)$ as the subgroup of  $\GL(\tM)$ consisting of linear transformations that preserve projection onto the last coordinate $\HT : \tM \to \Z$. In particular, 
we consider $\tM_\C$ as a $\C[G]$-module.
If $V$ is a finite-dimensional $\C[G]$-module, we write
$\det(I - V t) := \sum_i (-1)^i [\bigwedge^i V] t^i \in R(G)[t]$ (see \eqref{e:detdef}).
A priori, $h^*(P,\rho;t)$  depends on both $\Ehr(P,\rho;t)$ and  the complex representation $\tM_\C$ of $G$, but, in fact, $\Ehr(P,\rho;t)$ precisely encodes both  $h^*(P,\rho;t)$ and  $\tM_\C$ (see Lemma~\ref{l:encoding}).

When $G$ acts trivially these invariants agree with the usual invariants in Ehrhart theory: $L(P,\rho;m)$ is the \emph{Ehrhart polynomial} $L(P;m)$ of $P$ and $h^*(P,\rho;t)$ is the \emph{Ehrhart $h^*$-polynomial} $h^*(P;t)$ of $P$. 
In the general case, and in contrast to the usual invariants of Ehrhart theory,
$h^*(P,\rho;t)$ is not necessarily effective, and is not necessarily a polynomial. 
Alternative notations for $h^*(P,\rho;t)$ in the literature include $\phi[t]$, $\HH^*[t]$ and 
$\HH^*(P;z)$.

Consider an element $f = \sum_{u \in M} \lambda_u x^u \in \C[M]$ for some $\lambda_u \in \C$. We have an induced action of $G$ on $\C[M]$ and $f$ is $G$-invariant if and only if $\lambda_u = \lambda_{g \cdot u}$ for all $g \in G$ and $u \in M$. 
The notion that $f$ is \emph{nondegenerate with respect to $P$} (see Definition~\ref{d:nondegeneratesmooth})  first appeared in a special case as a technical condition in \cite{DworkZetaI,DworkZetaII}, and was later introduced in \cite{Varchenko76} and then independently in \cite{GKZ94}. It implies that $\lambda_u = 0$ for $u \notin P$, and is a Zariski open condition on the coefficients $\{ \lambda_u : u \in P \cap M \}$.  Importantly, since $G$-invariance is a closed condition, the existence of some $f \in \C[M]$ that is both $G$-invariant and nondegenerate with respect to $P$ is not guaranteed.

The following result follows from \cite[Corollary~6.6]{StapledonRepresentations11} which interprets $h^*(P,\rho;t)$ in terms of representations on the graded pieces of the Hodge filtration of the cohomology of a hypersurface associated to $f$. 

\begin{theorem}\label{t:mainoriginal}(see \cite[Theorem~7.7]{StapledonEquivariant})
	Let $G$ be a finite group.
	Let $M$ be a lattice of rank $d$ and let $\rho: G \to  \Aff(M)$ be an affine representation.  
	Let $P \subset M_\R$ be a $G$-invariant $d$-dimensional lattice polytope. Consider the following conditions: 
	
	\begin{enumerate}
		\item\label{i:orignondeg} There exists $f \in \C[M]$ that is $G$-invariant and nondegenerate with respect to $P$.
		
		\item\label{i:origeffective} $h^*(P,\rho;t)$ is effective. 
		
		\item\label{i:origpolynomial} $h^*(P,\rho;t)$ is a  polynomial.
		
	\end{enumerate}
	Then the following implications hold:  
	\eqref{i:orignondeg} $\Rightarrow$ \eqref{i:origeffective} $\Rightarrow$ \eqref{i:origpolynomial}. 
\end{theorem}

Moreover, it was conjectured in \cite[Conjecture~12.1]{StapledonEquivariant} that 
conditions \eqref{i:orignondeg}-\eqref{i:origpolynomial} in Theorem~\ref{t:mainoriginal} are equivalent. For the 
permutahedron, the conjecture was verified for the action of the symmetric group 
\cite[Theorem~1.2]{ASV20} as well as the subgroup $\Z/p\Z$ acting by cycling coordinates when $d + 1 = p$ is prime \cite[Theorem~3.52]{EKS22} (see Example~\ref{e:permutahedronv1}, Example~\ref{e:permutahedronv2} and Example~\ref{e:permutahedronv3} below). The conjecture was verified for the action of the symmetric group on hypersimplices \cite[Theorem~3.60]{EKS22}.
However, a counterexample to the implication 
\eqref{i:origeffective} $\Rightarrow$ \eqref{i:orignondeg} 
was given by Santos and Stapledon and appeared in \cite[Theorem~1.2]{EKS22} (see Example~\ref{e:SantosStapledon} below). We will present additional counterexamples in 
Example~\ref{e:d+3counterexample} and Example~\ref{e:p=5examplev2}. 
On the other hand, the following weaker conjecture remains open.

 \begin{conjecture}(The effectiveness conjecture)\label{c:origmod} 
	The conditions \eqref{i:origeffective}-\eqref{i:origpolynomial} in Theorem~\ref{t:mainoriginal} are equivalent. That is, $h^*(P,\rho;t)$ is effective if and only $h^*(P,\rho;t)$ is a polynomial. 
\end{conjecture}

 The effectiveness conjecture
has been verified in the following additional cases: $d = 2$ \cite[Corollary~6.7]{StapledonEquivariant}, $P$ is a simplex \cite[Proposition~6.1]{StapledonEquivariant} or the hypercube \cite[Section~9]{StapledonEquivariant},  $G = \Z/2\Z$ acting on centrally symmetric lattice polytopes when the unique fixed point lies in $M$ \cite[Corollary~11.1]{StapledonEquivariant},  
the Weyl group acting on the `convex hull' of the primitive integer vectors of the associated Coxeter fan \cite[Corollary~8.4]{StapledonEquivariant}, $\Z/2\Z$ acting on the graphic zonotope of a path graph with an odd number of vertices \cite[Section~3]{EKS22}, certain subgroups of the dihedral group acting on the symmetric edge polytope of the cycle graph \cite[Theorem~3.4, Theorem~3.6]{CHK23}.

 We will additionally verify 
 the effectiveness conjecture
 when $P$ admits a $G$-invariant lattice triangulation (see Theorem~\ref{t:mainfull}), when $p$ is a prime and $G = \Z/p\Z$ acts with a unique fixed point (see Example~\ref{e:primenofixedpts}) and 
when $P$ has $d + 2$ vertices and the vertices of $P$ affinely generate $M$ (see Proposition~\ref{p:d+2}). See Question~\ref{q:extendconjecture} for a possible extension of the effectiveness conjecture. 

The following conjecture was also made. It was verified for the action of the symmetric group on the permutahedron in \cite[Proposition~5.9]{ASV20}. 
We provide a counterexample in Example~\ref{e:d+2v4}. See also 
Question~\ref{q:trivialtriangulation}. 

\begin{conjecture}\label{c:trivialrepalwaysappears}\cite[Conjecture~12.4]{StapledonEquivariant}
	Assume that $h^*(P,\rho;t)$ is a polynomial, and let $h^*(P,\rho;t)_m$ in $R(G)$ denote the coefficient of $t^m$ for any nonnegative integer $m$. If $h^*(P,\rho;t)_m$ is nonzero and effective, then the trivial representation occurs with nonzero multiplicity in $h^*(P,\rho;t)_m$. 
\end{conjecture}

We now present Theorem~\ref{t:mainfull}, a generalization of Theorem~\ref{t:mainoriginal} that is the main result of the paper. Let $C_P \subset \tM_\R$ be the cone generated by $P \times \{ 1 \}$, and consider the graded semigroup algebra $S_P := \C[C_P \cap \tM]$, 
with grading induced by $\HT: \tM \to \Z$. 
 For any positive integer $m$, there is a natural bijection between elements of 
$C_P \cap \HT^{-1}(m)$
and elements of $P \cap \frac{1}{m} M$, taking $(u,m) \mapsto \frac{1}{m}u$. 
 In particular, $G$ acts on $S_P$ as a graded $\C$-algebra and 	$\Ehr(P,\rho;t)$ is the equivariant Hilbert series of $S_P$ (see Section~\ref{ss:commutative}).  By a result of Hochster \cite{HochsterRingsInvariants}, $S_P$ is a Cohen-Macaulay ring. 
 A \emph{linear system of parameters} (l.s.o.p.) is a collection of $d + 1$ homogeneous elements
 $F_1,\ldots,F_{d + 1}$  of degree $1$ in $S_P$  such that 
 $S_P/(F_1,\ldots,F_{d + 1})$ is a finite dimensional $\C$-vector space (see Definition~\ref{d:hsop} and Theorem~\ref{t:hsop}).
%
%
If $F_1,\ldots,F_{d + 1}$ is a l.s.o.p. for $S_P$, then the ideal $(F_1,\ldots,F_{d + 1})$ is $G$-invariant if and only if the $\C$-vector space 
$\C F_1 + \cdots + \C F_{d + 1}$
is $G$-invariant if and only if  we may consider
$[\C F_1 + \cdots + \C F_{d + 1}]$ as an element of $R(G)$.  

\begin{theorem}\label{t:mainfull}
		Let $G$ be a finite group.
	Let $M$ be a lattice of rank $d$ and let $\rho: G \to  \Aff(M)$ be an affine representation.  
	Let $P \subset M_\R$ be a $G$-invariant $d$-dimensional lattice polytope. 
	Let $\tM = M \oplus \Z$ and let $C_P$ be the cone over $P \times \{ 1\}$ in $\tM_\R$ with corresponding graded semigroup algebra $S_P = \C[C_P \cap \tM]$.
	Consider the following conditions: 
	
	
	\begin{enumerate}
		
		\item\label{i:maintriangulation} There exists a $G$-invariant regular lattice triangulation of $P$.

		\item\label{i:mainnondegenerate} 
		There exists $f \in \C[M]$ that is $G$-invariant and nondegenerate with respect to $P$. 
		
		\item\label{i:mainhilbert} 
		There exists a l.s.o.p. $F_1,\ldots,F_{d + 1}$ of $S_P$ such that 
		$[\C F_1 + \cdots + \C F_{d + 1}] = [\tM_\C]$ in $R(G)$.
		
		\item\label{i:maineffective} $h^*(P , \rho ; t)$ is effective.
		
		\item\label{i:mainpolynomial} $h^*(P , \rho ; t)$ is a polynomial.
		
	\end{enumerate}
	
	Then the following implications hold: \eqref{i:maintriangulation} $\Rightarrow$
	\eqref{i:mainnondegenerate} $\Rightarrow$
	\eqref{i:mainhilbert} $\Rightarrow$
	\eqref{i:maineffective} $\Rightarrow$
	\eqref{i:mainpolynomial}.
	Moreover, if there exists a $G$-invariant (not necessarily regular) lattice triangulation of $P$, then  
	\eqref{i:maineffective} holds. 
\end{theorem}

When $G$ acts trivially, all the  conditions of  Theorem~\ref{t:mainfull} 
are well-known to hold. See, for example, \cite[Lemma~2.3.15]{DRSTriangulations10} for condition \eqref{i:maintriangulation} and
\cite{StanleyMagic,StanleyDecompositions} for conditions \eqref{i:mainhilbert} and \eqref{i:maineffective}. 
We observe that condition \eqref{i:maintriangulation} in Theorem~\ref{t:mainfull} is often easier to verify in practice than conditions \eqref{i:mainnondegenerate} or \eqref{i:mainhilbert}. In Proposition~\ref{p:existtriangulation} we give an explicit criterion to guarantee that \eqref{i:maintriangulation} holds i.e. a $G$-invariant regular lattice triangulation of $P$ exists. Using the notion of a \emph{translative} action studied in \cite{DRSemimatroids}, \cite{DDStanleyReisner}, \cite{BDSupersolvable} and \cite{DDEquivariantEhrhart} (see Definition~\ref{d:translative}) we deduce two corollaries for refining and gluing invariant lattice polyhedral subdivisions to construct invariant lattice triangulations (see Corollary~\ref{c:translativeK} and Corollary~\ref{c:translativeglue} respectively).
For example, we deduce that \eqref{i:maintriangulation} holds when $G = \Z/2\Z$ acts on a centrally symmetric lattice polytope (see Example~\ref{e:centrallysymmetrictriangulation}). 
On the other hand, we can view condition \eqref{i:mainpolynomial}  in Theorem~\ref{t:mainfull} as an obstruction that can be used to construct examples where conditions \eqref{i:maintriangulation}-\eqref{i:mainhilbert} fail.
 
For $G$-invariant lattice triangulations with some special additional structure, combinatorial formulas for the coefficients of $h^*(P,\rho;t)$ as explicit permutation representations were given in 
\cite[Theorem~3.30]{EKS22} (see, also, \cite[Theorem~3.32]{EKS22}).
When the action of $G$ is translative and there exists a $G$-invariant unimodular lattice triangulation of $P$, D'al\`{i} and Delucchi showed in independent work that the equivariant Ehrhart series can be expressed as a rational function with denominator $(1 - t)^{d + 1}$ and gave a criterion which guarantees that the corresponding numerator is effective \cite[Theorem~5.2]{DDEquivariantEhrhart}, with applications to order polytopes and Lipschitz polytopes of posets (\cite[Theorem~6.7]{DDEquivariantEhrhart} and \cite[Corollary~7.12]{DDEquivariantEhrhart} respectively).

  In the general case when there exists a $G$-invariant lattice triangulation of $P$,
we provide a formula for $h^*(P,\rho;t)$ involving both combinatorial terms and equivariant Hilbert polynomials of Stanley-Reisner rings in Proposition~\ref{p:formulatriangulation}. This generalizes known formulas when $P$ is a simplex \cite[Proposition~6.1]{StapledonEquivariant}, when the triangulation is unimodular \cite[Proposition~8.1]{StapledonEquivariant} or when $G$ acts trivially 
\cite[Theorem~1]{BMLatticePoints}. 

Let us make a few additional comments on the statement and proof of Theorem~\ref{t:mainfull}. Firstly, 
we deduce the implication \eqref{i:maintriangulation} $\Rightarrow$ 
\eqref{i:mainnondegenerate} as a corollary of a theorem of Gelfand, Kapranov and Zelevinsky \cite[Theorem 10.1.12']{GKZ94} (see Theorem~\ref{t:initdeg}). Secondly, in the implication \eqref{i:mainnondegenerate} $\Rightarrow$
\eqref{i:mainhilbert}, the ideal generated by the l.s.o.p. corresponding to $f$ is known as the \emph{Jacobian ideal} of $f$, and the quotient of $S_P$ by this ideal is known as the \emph{Jacobian ring} of $f$ \cite[Definition~4.7]{BatyrevVariations} (see Definition~\ref{d:Jacobiandef}). 
The original proof of Theorem~\ref{t:mainoriginal} can be recovered from the proof of Theorem~\ref{t:mainfull} using Batyrev's geometric interpretation of the Jacobian ring of $f$ \cite[Corollary~6.10]{BatyrevVariations}. 
Thirdly, 
either in the implication \eqref{i:mainhilbert} $\Rightarrow$
\eqref{i:maineffective} or in the case where there exists a (not necessarily regular) $G$-invariant lattice triangulation of $P$, $h^*(P,\rho;t)$ is realized as the equivariant Hilbert polynomial of the quotient by a l.s.o.p. of either $S_P$ or a deformed version of $S_P$ (see Definition~\ref{d:toricfacering}), known as a \emph{toric face ring} or \emph{deformed group ring} in the literature, respectively. 
In the latter case,  $h^*(P,\rho;t)$  is also interpreted geometrically as the equivariant Hilbert series  of an induced action of $G$ on the orbifold cohomology of 
 an associated  toric variety (see the discussion before Proposition~\ref{p:formulatriangulation}  and \cite[(1.1)]{KaruEhrhart}). 
 
 Finally, one may ask whether any of the converses of  the implications in Theorem~\ref{t:mainfull} hold.  
 When $d = 2$, we show that all conditions in Theorem~\ref{t:mainfull} are equivalent (see Proposition~\ref{p:dim2mainthm}). In the general case,
  the question of whether any of the conditions \eqref{i:mainhilbert}, \eqref{i:maineffective} and \eqref{i:mainpolynomial} are equivalent is left open as part of a possible extension of the effectiveness conjecture in Question~\ref{q:extendconjecture}. On the other hand, all other converse implications  in Theorem~\ref{t:mainfull} are false. See Example~\ref{e:d+2v4},
 Example~\ref{e:d+3counterexample},
 Example~\ref{e:p=5examplev2} and 
 Example~\ref{e:SantosStapledon}.

Our next goal is to extend our commutative algebra interpretation to the case when the conditions of Theorem~\ref{t:mainfull} don't necessarily apply. Our first result in this direction implies that we may apply Theorem~\ref{t:mainfull} after an appropriate 
scaling of the lattice $M$. Explicitly,  given a positive integer $N$, 
$\rho: G \to \Aff(M)$ induces an action 
$\rho_N: G \to \Aff(\frac{1}{N}M)$ that preserves $P$, and Theorem~\ref{t:triangulationOrderGroup} below implies that we may apply Theorem~\ref{t:mainfull} if we replace $\rho$ with $\rho_N$, provided that the order $|G|$ of $G$ divides $N$. 

%

\begin{theorem}\label{t:triangulationOrderGroup}
	Let $G$ be a finite group.
	Let $M$ be a lattice of rank $d$ and let $\rho: G \to  \Aff(M)$ be an affine representation.  
	Let $P \subset M_\R$ be a $G$-invariant $d$-dimensional lattice polytope. 
	Let $N$ be a positive integer 
	divisible by $|G|$. 
	Then there exists a $G$-invariant regular triangulation of $P$ with vertices in $\frac{1}{N} M$.
\end{theorem}
The content of this theorem is that the vertices of a $G$-invariant regular triangulation can be chosen to lie in $\frac{1}{|G|} M$. The existence of a $G$-invariant regular triangulation (with no restriction on the vertices) can be achieved easily, for example, by taking a 
barycentric subdivision using the averages of the vertices of all the faces of $P$ (cf. 
the proof of Corollary~\ref{c:translativeK}).

Secondly, 
for a positive integer $N$, define 
$h^*_{N}(P,\rho;t)$ 
in $ R(G)[[t^{\frac{1}{N}}]]$ by the equality:
\begin{equation*}
	\Ehr(P,\rho;t) = \frac{h^*_{N}(P,\rho;t^N)}{\det(I - \tM_\C t^N)} \in R(G)[[t]].
\end{equation*}
On the one hand, the definition is motivated by the non-equivariant case when $h^*_{N}(P,\rho;t)$ has an interpretation as the weighted Ehrhart polynomial of the $N$th dilate of the pyramid over $P$
(see Lemma~\ref{l:intandceil} and \cite{StapledonWeighted}).
%
On the other hand,
the definition 
is also motivated by the Ehrhart theory of rational polytopes where denominators with factors of the form $1 - t^N$ (rather than $1 - t$) naturally appear (see Remark~\ref{r:changedenominator}). In fact, we leave it as an open problem to extend the theory of this paper to the case when $P$ is a rational (rather than a lattice) polytope (cf. \cite{CHK23}). 
We recover the invariant $h^*(P,\rho_N;t)$ from $h_N^*(P,\rho;t)$ by extracting all terms with integer (rather than just rational) exponents in $t$ (see Lemma~\ref{l:intandceil}). 

The benefit of studying $h_N^*(P,\rho;t)$, as opposed to just $h^*(P,\rho;t)$ is that Theorem~\ref{t:triangulationOrderGroup}, together with Corollary~\ref{c:maincorollaryshort}) below, ensure that one may always choose $N$ sufficiently divisible such that $h_N^*(P,\rho;t)$ is effective i.e. its coefficients are effective. For example, 
in the case of the symmetric group acting on the permutahedron,  \cite[Proposition~5.3-5.4]{ASV20} showed that 
 $h^*(P,\rho;t)$ is a polynomial if and only if  $h^*(P,\rho;t)$ is effective if and only if $d \le 2$. We will show that  $h^*_N(P,\rho;t)$ is a polynomial if and only if  $h^*_N(P,\rho;t)$ is effective if and only if
 $N$ is even or $d \le 2$
(see Example~\ref{e:permutahedronv1} and Example~\ref{e:permutahedronv2}). 

We will deduce the following result as a corollary of 
Theorem~\ref{t:mainfull} and its proof. 
Below, we say $h^*_N(P , \rho ; t)$ is a polynomial if $h^*_N(P , \rho ; t) \in R(G)[t^{\frac{1}{N}}]$.

\begin{corollary}\label{c:maincorollaryshort}
	Let $G$ be a finite group.
	Let $M$ be a lattice of rank $d$ and let $\rho: G \to  \Aff(M)$ be an affine representation.  
	Let $P \subset M_\R$ be a $G$-invariant $d$-dimensional lattice polytope. 
	Let $N$ be a positive integer.
	Consider the following conditions:

	\begin{enumerate}
		
		\item\label{i:corrtriangulationshort} There exists a $G$-invariant triangulation of $P$ with vertices in $\frac{1}{N} M$.

		\item\label{i:correffectiveshort} $h^*_N(P , \rho ; t)$ is effective. 
		
		\item\label{i:corrpolynomialshort} $h^*_N(P , \rho ; t)$ is a polynomial.
	
	\end{enumerate}
	
	Then the following implications hold: \eqref{i:corrtriangulationshort} $\Rightarrow$
	\eqref{i:correffectiveshort} $\Rightarrow$
	\eqref{i:corrpolynomialshort}.

\end{corollary}

Finally, we anticipate that one can extend other results in Ehrhart theory that use the commutative algebra viewpoint to the equivariant setting. We mention one such result. When $G$ acts trivially, the theorem below reduces to a classic result of Stanley \cite[Theorem~3.3]{StanleyMonotonicity}. In fact, the proof of Theorem~\ref{t:monotonicity} below may be regarded as an equivariant analogue of Stanley's proof. Below, if $f(t) = \sum_m f_m t^m$ and $g(t) = \sum_m g_m t^m$ are elements in $R(G)[t]$, we write $f(t) \le g(t)$ is $g_m - f_m$ is effective for all $m$. If $Q$ is a $G$-invariant lattice polytope contained in $P$, let 
 $M_Q$ be  the intersection of $M$ and the affine span of $Q$. Then $M_Q$ is an affine lattice with an induced affine representation  $\rho_Q : G \to \Aff(M_Q)$, and one may consider the associated invariant $h^*(Q, \rho_Q;t)$ (see Remark~\ref{r:altformulation}).

\begin{theorem}\label{t:monotonicity}
		Let $G$ be a finite group.
	Let $M$ be a lattice of rank $d$ and let $\rho: G \to  \Aff(M)$ be an affine representation.  
	Let $P \subset M_\R$ be a $G$-invariant $d$-dimensional lattice polytope.
	Let $Q \subset P$ be a $G$-invariant (not necessarily full dimensional) lattice polytope. 
	Suppose there exists a $G$-invariant lattice triangulation of $P$ that restricts to a $G$-invariant lattice triangulation of $Q$. Then $h^*(Q, \rho_Q;t) \le h^*(P, \rho;t)$.
	
	
\end{theorem}

We briefly discuss the structure of the paper. In Section~\ref{s:equivariantEhrhart}, we recall and develop equivariant Ehrhart theory.
In Section~\ref{s:triangulations}, we develop criterion to guarantee the existence of a $G$-invariant lattice triangulation and prove Theorem~\ref{t:triangulationOrderGroup}.
In Section~\ref{s:commutativealgebra}, we recall background on commutative algebra and the principal $A$-determinant, before proving many of our results including Theorem~\ref{t:mainfull}, Corollary~\ref{c:maincorollaryshort} and  Theorem~\ref{t:monotonicity}. 
Below, we include a short erratum for the paper \cite{StapledonEquivariant}, before listing some notation to be used throughout the paper.

\emph{Erratum to \cite{StapledonEquivariant}}: We mention three errors in \cite{StapledonEquivariant}, ordered from most to least serious. 
Firstly, 
\cite[Remark~7.8]{StapledonEquivariant}, together with \cite[Remark~7.8]{StapledonEquivariant}, give a criterion that guarantees that condition
\eqref{i:mainnondegenerate} in Theorem~\ref{t:mainfull} holds. 
 However, Example~\ref{e:p=5example} shows that this is false. The fix is that in equation $(6)$ in \cite[Remark~7.8]{StapledonEquivariant} `$G_Q$-orbit' needs to be replaced by `$G$-orbit'. Alternatively, we provide Corollary~\ref{c:translativeK} and Corollary~\ref{c:translativeglue}, which together with Theorem~\ref{t:mainfull} may act as a replacement (with stronger hypotheses). As a consequence of this error, some of the proofs of subsequent statements are incomplete:  \cite[Corollary~7.10]{StapledonEquivariant} (we do not have a proof or a counterexample), \cite[Corollary~7.12]{StapledonEquivariant} (the statement is correct by Proposition~\ref{p:dim2mainthm}), \cite[Corollary~7.14]{StapledonEquivariant} (the statement is correct by Theorem~\ref{t:mainfull} and Theorem~\ref{t:triangulationOrderGroup}), 
 \cite[Example~7.16]{StapledonEquivariant} (we do not have a proof or a counterexample). 
 In the cases above where we do not have a counterexample, we do have reason to expect that the statements are true. 
 Unfortunately, the proof of \cite[Theorem~3.52]{EKS22} references \cite[Corollary~7.10]{StapledonEquivariant} and hence is incomplete (the version of Corollary~7.10 stated in \cite[Corollary~2.10]{EKS22} is false by Example~\ref{e:p=5example}). 
 However, the statement of \cite[Theorem~3.52]{EKS22} is true and Example~\ref{e:permutahedronv3} below provides a generalization (using Theorem~\ref{t:mainfull}).

 Secondly, in the statement of condition \eqref{i:mainnondegenerate} we consider a polynomial $f \in \C[M]$ that is $G$-invariant. This implies that the hypersurface $\{ f  = 0 \} \subset \Spec \C[M]$ is $G$-invariant. However, the converse is false. More precisely, there may be a nontrivial character $\chi: G \to \C$ such that $g \cdot f = \chi(g)f$ for all $g$ in $G$, and then $\{ f = 0 \}$ is $G$-invariant but $f$ is not. In particular, 
 references to $G$-invariant hypersurfaces in
 \cite{StapledonEquivariant} and \cite{StapledonRepresentations11}
  should be references to $G$-invariant polynomials. 
  We leave open the problem of extending the theory to deal with all $G$-invariant hypersurfaces. 

Finally, it what follows we will consider an affine representation $\rho \colon G \to \Aff(M)$. 
After possibly replacing $G$ with $G/\ker(\rho)$, we may always assume that $G$ acts faithfully. The statements of \cite[Corollary~5.9]{StapledonEquivariant} and \cite[Corollary~5.11]{StapledonEquivariant} should include the assumption that $G$ acts faithfully.

\emph{Notation and Conventions}:
Let $m,n$ be positive integers. We write $m | n$ if $m$ divides $n$, and 
write $\lcm(m,n)$ and $\gcd(m,n)$ for the lowest common multiple and greatest common divisor respectively of $m$ and $n$. 
If $G$ is a finite group and $g \in G$, we let $\ord(g)$ be the order of $g$, and 
let $\exp(G) := \lcm( \ord(g) : g \in G )$ denote the exponent of $G$. 
A \emph{lattice} $M$ is a   free abelian group of finite rank. 
We write $M_k := M \otimes_\Z k$ for a field $k$. If a group $G$ acts on a set $X$, we will often write $g \cdot x$ for the action of $g$ in $G$ on $x$ in $X$. 
We let  $\{ e_i : 1 \le i \le n \}$ denote the standard basis vectors in $\Z^n$  for a positive integer $n$.  We write $\Conv{A,B}$ to denote the convex hull of sets $A,B$. If $L$ is a lattice and $r$ is a ray in $L_\R$ such that $L \cap r$ is nonzero, then the  \emph{primitive integer vector} of $r$ is 
 the generator of the semigroup $r \cap L$. 

\emph{Acknowledgements}: 
We thank D'al\`{i} and Delucchi  for bringing their independent work \cite{DDEquivariantEhrhart} to the author's attention and for friendly discussions of their results. 
We also thank  Lev Borisov and Tim R\"omer for enlightening 
discussions.

\section{Equivariant Ehrhart theory}\label{s:equivariantEhrhart}

In this section, we recall and develop equivariant Ehrhart theory. The relevant background material is first introduced in Section~\ref{ss:rationalpolytopes}, Section~\ref{ss:affine} and
Section~\ref{ss:reptheorybasics}. The main invariants are then introduced and their properties are developed in 
Section~\ref{ss:equivariantEhrhart} and Section~\ref{ss:equivarianthNstar}


\subsection{Ehrhart theory of rational polytopes}\label{ss:rationalpolytopes}

We first recall some basic facts about Ehrhart theory of rational polytopes. We refer the reader to \cite[Section~4.6.2]{StanleyEnumerative} and \cite{BeckRobinsComputing}  for more details.
Let $M$ be a lattice and 
let 
$Q \subset M_\R$ be a \emph{rational polytope} i.e. the convex hull of finitely many points in $M_\Q$.
Then $Q$ is a \emph{lattice polytope} if its vertices lie in $M$. 
We say that $Q$ is a \emph{reflexive polytope} if it is a lattice polytope containing the origin in its relative interior, and every nonzero element of $M$ lies on the boundary of a positive integer dilate of $Q$. 
Let 
$L(Q;m) := |Q \cap \frac{1}{m} M|$ for a positive integer $m$. 
Then $L(Q;m)$ is a quasi-polynomial of degree $\dim Q$ with constant term $1$ called the \emph{Ehrhart quasi-polynomial} of $Q$. 
The associated generating series
\[
\Ehr(Q;t) := \sum_{m \ge 0} L(Q;m) t^m
\]
is called the \emph{Ehrhart series} of $Q$.
We will need the following theorem.

\begin{theorem}\cite[Theorem~4.6.8]{StanleyEnumerative}\label{t:orderpole}
Let $Q \subset M_\R$ be a rational polytope.
Then the Ehrhart series $\Ehr(Q;t)$ is a rational function in $t$ and the order of $t = 1$ as a pole is equal to $\dim Q + 1$. 
\end{theorem}

\begin{remark}\label{r:changedenominator}
In fact, a stronger result holds. 
The \emph{denominator} of $Q$ is the smallest positive integer $N$ such that $NQ$ is a lattice polytope. 
Then $\Ehr(Q;t)$ can be expressed as a rational function with denominator 
$(1 - t^N)^{\dim Q + 1}$ and numerator a polynomial of degree strictly less than $N (\dim Q + 1)$ with constant term equal to $1$ and nonnegative integer coefficients \cite{StanleyMagic,StanleyDecompositions}.
When $Q$ is a lattice polytope,  $N = 1$ and the numerator is the \emph{$h^*$-polynomial} $h^*(Q;t)$ of $Q$.
\end{remark}

Assume further than $Q$ is a simplex. In this case, there is a well-known formula for $\Ehr(Q;t)$ due to Ehrhart \cite[p.54]{EhrhartPolynomes} 
 that we now recall (cf. Proposition~\ref{p:formulatriangulation}). 
 Let $L$ be a lattice and let $\psi: L \to \Z$ be a surjective group homomorphism with corresponding linear map $\psi_\R : L_\R \to \R$. Assume that $Q \subset L_\R$ is a rational polytope with respect to $L$ that is contained in $\psi_\R^{-1}(1)$. Let $C_Q$ 
 be the cone generated by $Q$ in $L_\R$. Let 
 $u_1,\ldots,u_s$ be the primitive integer vectors of the rays of $C_Q$, and define
 \[
 \BBox(C_Q) := \{ u \in C_Q \cap L : u = \sum_{j = 1}^s \lambda_j u_{j} \textrm{ for some } 0 \le  \lambda_j < 1 \}.
 \]
 Then
 \begin{equation}\label{e:rationalsimplex}
 	\Ehr(Q;t) = \frac{\sum_{u \in \BBox(C_Q) } t^{\psi(u)}}{\prod_j (1 - t^{\psi(u_j)})}.
 \end{equation}
 

The following example will be useful for later computations.

\begin{example}\label{e:d+2v1}
	Assume that $d$ is an even positive integer, and let $r = \frac{d}{2} + 1$. 
	Let $a_1,\ldots,a_r$ be positive integers with $\gcd(a_1,\ldots,a_r) = 1$.
	Let $e_1,\ldots,e_r,f_1,\ldots,f_r$ be a basis of $\Z^{d + 2}$. 
	Let $\bar{e}_i$ and $\bar{f_i}$ denote the images of $e_i$ and $f_i$ respectively in the lattice
	$L := \Z^{d + 2}/\Z(\sum_i a_i (e_i - f_i))$. 
	Consider the group homomorphism  $\psi: L \to \Z$ satisfying $\psi(\bar{e}_i) = \psi(\bar{f}_i) = 1$ for all $1 \le i \le r$, 
	and the corresponding linear map $\psi_\R: L_\R \to \R$. 
	Let $I_1,\ldots,I_s$  be a partition of  $\{1,\ldots,r\}$. Assume that 
	$a_i = a_j$ if $i,j \in I_k$ for any $1 \le i,j \le r$ and $1 \le k \le s$. 
	Let $Q = \Conv{\frac{1}{2|I_k|} \sum_{i \in I_k} (\bar{e}_i + \bar{f}_i) :  1 \le k \le s} \subset L_\R$. Then $Q \subset \psi_\R^{-1}(1)$ is a rational simplex with $s$ vertices.
	We claim that 
\[
\Ehr(Q;t) = \frac{1 + t^{|\{ i : a_{i}\textrm{ odd } \}|}}{\prod_k (1 - t^{2|I_k|})}.
\]
To establish the claim, consider the cone $C_Q$ generated by $Q$ in $L_\R$. 
The primitive integer vectors of $C_Q$ are $\{ \sum_{i \in I_k} (\bar{e}_i + \bar{f}_i) : 1 \le k \le s \}$. 
The claim follows from \eqref{e:rationalsimplex} if we can establish that $\BBox(C_Q) = \{ 0,  \frac{1}{2}\sum_{i: a_{i} \textrm{ odd }} (\bar{e}_i + \bar{f}_i)  \}$.
Observe that $\BBox(C_Q)$ consists of the images $\bar{u}$ in $L$ of all elements $u$ in $\Z^{d + 2}$ that have the form:
	\[
u =  \sum_k \lambda_k \sum_{i \in I_k} (e_i + f_i) = 
\sum_i m_i  e_i +   \sum_i m_i' f_i + \beta \sum_i a_i (e_i - f_i)
\]
for some $0 \le \lambda_k < 1$, $m_i,m_i' \in \Z$ and $\beta \in \R$. 
Note that $\lambda_k$ is determined by $\lambda_k = \beta a_i = -\beta a_i \in \R/\Z$ for any $i \in I_k$. Since $\gcd(a_1,\ldots,a_r) = 1$, the condition $2 \beta a_i = 0 \in \R/\Z$ implies that either $\beta = 0 \in \R/\Z$ or $\beta = \frac{1}{2} \in \R/\Z$. If $\beta = 0 \in \R/\Z$, then $\lambda_k = 0$ for all $k$
and $\bar{u} = 0$. Assume that $\beta = \frac{1}{2} \in \R/\Z$. Then $\lambda_k = \frac{1}{2}$ if $a_{i}$ is odd for some/any $i \in I_k$, and  $\lambda_k = 0$ otherwise. 
The only solution is $u = \frac{1}{2}\sum_{i: a_{i} \textrm{ odd }} (e_i + f_i)$. 
\end{example}

\subsection{Affine transformations}\label{ss:affine}

We next recall some basic facts about affine spaces and affine transformations. 
We refer the reader to \cite{BergerGeometry} and  \cite[Section~5.3.D]{GKZ94} for details.

An \emph{affine lattice} $M'$ is a principal homogeneous space with respect to a lattice $M$. That is, $M$ is a free abelian group that acts transitively and freely on $M'$. Given $u_1,u_2$ in $M'$, there is a unique element $u_1 - u_2$ in $M$ such that $u_1 = (u_1 - u_2) + u_2$. 
If we fix an element $u'$ in $M'$, then we have a bijection 
$f_{u'}: M' \to M$ defined by
$f_{u'}(u) = u - u'$ for all $u \in M'$,
that we can use to give $M'$ the structure of a lattice. Conversely, $M$ naturally has the structure of an affine lattice with respect to the action of itself. Informally, we view
$M$ as being the data of $M'$ together with a choice of origin. 
Given $u_1,\ldots,u_r$ in $M'$ and $\lambda_1,\ldots,\lambda_r \in \Z$ such that 
$\sum_i \lambda_i = 1$, there is a well-defined element $\sum_i \lambda_i u_i$ in $M'$ such that   $f_{u'}(\sum_i \lambda_i u_i) = \sum_i \lambda_i (u_i - u') \in M$ for all choices of $u'$ in $M'$. 
We say that a finite set $A \subset M'$ \emph{affinely generates} $M'$ if $M' = \{ \sum_{a \in A} \lambda_a a : \lambda_a \in \Z, \sum_{a} \lambda_a = 1 \}$.

An \emph{affine transformation} of $M'$ is a map $\phi: M' \to  M'$ such that we have a well-defined group homomorphism
  $\bar{\phi}: M \to M$ satisfying $\bar{\phi}(u_1 - u_2) = \phi(u_1) - \phi(u_2)$ for all $u_1,u_2 \in M'$. The set of all 
bijective affine transformations forms a group $\Aff(M')$. On the other hand, an affine transformation of the lattice $M$ is a map  $\nu: M \to M$ of the form $\nu(u) = A(u) + b$  for some $A \in \GL(M)$, $b \in M$ and any $u \in M$. These form a group $\Aff(M)$ which may be identified with the semidirect product $M \rtimes \GL(M)$. 
If we fix an element $u'$ in $M'$, then we have an isomorphism of groups 
$\Aff(M') \to \Aff(M)$, $\phi \mapsto f_{u'} \circ \phi \circ f_{u'}^{-1}$. 

We will also need the following alternative interpretation of $\Aff(M)$. 
Let $\tM = M \oplus \Z$ with projection map  $\HT : \tM \to \Z$ onto the last coordinate. Then $M \times \{1\} = \HT^{-1}(1)$ is an affine lattice. Moreover, if $0_M$ denotes the origin in $M$, then we may use the distinguished point $(0_M, 1)$ in $M \times \{1\}$ to identify $\Aff(M)$ with $\Aff(M \times \{1\})$. On the other hand, $\Aff(M \times \{1\})$ is isomorphic (via restriction) to the  subgroup $\{ \widehat{\phi} \in \GL(\tM) : \HT \circ  \widehat{\phi} = \HT \} \subset \GL(\tM)$ i.e. the subgroup of 
 $\GL(\tM)$ that preserves the affine lattice $M \times \{1\}$. 
 In this way, we will often view $\Aff(M)$ as a subgroup of $\GL(\tM)$. 
 Explicitly, if $\nu(u) = A(u) + b$  for some $A \in \GL(M)$, $b \in M$ and any $u \in M$, then the corresponding linear transformation is $\widehat{\phi}(u,m) = (A(u) + mb, m)$ for any $u \in M$ and $m \in \Z$.

 Let $W$ be a real affine space i.e.  a principal homogeneous space with respect to a real vector space $V$. As in the case of affine lattices, given $w_1,w_2 \in W$, there is a unique element $w_1 - w_2$ in $V$ such that $w_1 = (w_1 - w_2) + w_2$, and 
 given $w_1,\ldots,w_r$ in $W$ and $\lambda_1,\ldots,\lambda_r \in \R$ such that 
 $\sum_i \lambda_i = 1$, there is a well-defined element $\sum_i \lambda_i w_i$ in $W$. In particular, a polytope in $W$ is defined to be the convex hull of finitely many points in $W$. 
 An affine transformation of $W$ is a map $\phi: W \to  W$ such that we have a well-defined 
 $\R$-linear map
 $\bar{\phi}: V \to V$ satisfying $\bar{\phi}(w_1 - w_2) = \phi(w_1) - \phi(w_2)$ for all $w_1,w_2 \in W$, and the set of all 
 bijective affine transformations forms a group $\Aff(W)$.
 
 Following \cite[Section~5.3.D]{GKZ94}, a subset $M' \subset W$ is an \emph{affine lattice in $W$} if $M'$ is an affine lattice with corresponding lattice $M \subset V$, where $M$ is a lattice of rank 
  $\dim V = \dim W$. In that case, one verifies that $\Aff(M')$ may be identified with the subgroup $\{ \phi \in \Aff(W) : \phi(M') = M' \}$ of $\Aff(W)$.  
 Also, the inclusion $M \subset V$ induces an isomorphism of vector spaces $M_\R \cong V$. 
 If we fix an element $u'$ in $M'$. Then $f_{u'}: M' \to M$ extends to a bijection from $W \to M_\R$ that takes $w \in W$ to $w - u' \in V \cong M_\R$. For any positive integer $m$, we may consider the affine lattice $\frac{1}{m} M' := \{ \sum_{i} \lambda_i u_i : u_i \in M', \lambda_i \in \frac{1}{m} \Z, \sum_i \lambda_i = 1 \}$ in $W$. Then $f_{u'}(\frac{1}{m} M') = \frac{1}{m}M \subset M_\R$ for all $u'$ in $M'$. If $P \subset W$ is a polytope, then 
 $P \cap (\cup_{m > 0} \frac{1}{m}M')$ has the structure of a semigroup with the origin removed. Indeed, given $u_1 \in P \cap \frac{1}{m_1} M'$ and $u_2 \in P \cap \frac{1}{m_2}M'$, then the sum of $u_1$ and $u_2$ is defined to be $\frac{m_1}{m_1 + m_2} u_1 + \frac{m_2}{m_1 + m_2} u_2 \in P \cap
\frac{1}{m_1 + m_2}M'$. Let $(P \cap (\cup_{m > 0} \frac{1}{m}M'))^\times$ denote the corresponding semigroup. 

\begin{example}\label{e:constructaffine}
	 Suppose that $L$ is a lattice, $\psi: L \to \Z$ is a group homomorphism and $m$ is an integer in the image of $\psi$. 
	Let $\psi_\R: L_\R \to \R$ be the corresponding map of $\R$-vector spaces. 
	Then $W = \psi_\R^{-1}(m)$ is a real affine space with corresponding vector space $V = \psi_\R^{-1}(0)$, and 
	$M' = \psi^{-1}(m)$ is an affine lattice in $W$ with corresponding lattice $\psi^{-1}(0)$. 
\end{example}
 
 Finally, if $Q$ is a lattice polytope with respect to an affine lattice $M'$ in $W$ then  the \emph{normalized volume} $\Vol(Q) \in \Z_{>0}$ 
 is $d!$ times the Euclidean volume of $\psi(Q)$ if one chooses an isomorphism 
 $\psi: W \to \R^d$ of affine spaces such that $\psi(M') = \Z^d$
 (see, for example, \cite[Definition~5.1.3]{GKZ94}).

%

%

\subsection{Representation theory of finite groups}\label{ss:reptheorybasics}

We now recall some basic facts about representation theory of finite groups. We refer the reader to 
\cite{SerreLinearRepresentations} for an excellent reference on this subject.


Let $G$ be a finite group with identity element $\id$ and let $M$ be a lattice. 
A (linear) \emph{representation} of $G$ acting on $M$ is a group homomorphism $\hat{\rho}: G \to \GL(M)$. 
An \emph{affine representation} of $G$ acting on $M$ is a group homomorphism $\rho: G \to \Aff(M)$.
The following lemma is well known and we provide a proof only for the convenience of the reader. In the statement below, note that $\hat{\rho}$ is uniquely defined, while $c$ may be any $G$-fixed point of $\frac{1}{|G|}M$. 
  
\begin{lemma}\label{l:affinereps}
Affine representations of $G$ acting on $M$ are precisely functions $G \times M \to M$ of the form $$\rho(g)(u) = 
\hat{\rho}(g)(u - c) + c \textrm{ for any } g \in G, u \in M,$$ where 
$\hat{\rho}: G \to \GL(M)$ is a representation of $G$
and $c \in \frac{1}{|G|}M$ 
satisfies 
 $c - \hat{\rho}(g)(c) \in M$ for all $g \in G$. 
\end{lemma}
\begin{proof}
Suppose we have such a $\hat{\rho}: G \to \GL(M)$ and $c \in \frac{1}{|G|}M$. For any $g \in G$ and $u \in M$, we define  $\rho(g)(u) = \hat{\rho}(g)(u) + b_g$, where $b_g := c - \hat{\rho}(g)(c) \in M$. 
Then $\rho(g) \in \Aff(M)$ and
a direct calculation shows that $\rho: G \to \Aff(M)$ is a group homomorphism. 

Conversely, let $\rho: G \to \Aff(M)$ be an affine representation. Composing with the natural projection from $\Aff(M) = M \rtimes \GL(M) \to \GL(M)$ gives a 
representation $\hat{\rho}: G \to \GL(M)$. 
Then $\rho(g)(u) = \hat{\rho}(g)(u) + b_g$ for all $g \in G$ and $u \in M$ and for some $b_g \in M$.
Since $G$ is finite, there exists a (not necessarily unique) $G$-fixed element $c $ in $\frac{1}{|G|}M$. 
Indeed, for any $u $ in $ M$, we can take $c = \frac{1}{|G|} \sum_{g \in G} g \cdot u \in \frac{1}{|G|}M$. Then $c = \rho(c) =  \hat{\rho}(g)(c) + b_g$ implies that $b_g = c - \hat{\rho}(g)(c) \in M$, and the result follows.

\end{proof}

In particular, Lemma~\ref{l:affinereps} implies that for any $g$ in $G$,  there exists a $g$-fixed point in $\frac{1}{\ord(g)}M$.
We will need the following simple lemma. 

\begin{lemma}\label{l:bounddenomiinatorfixedgeneral}
	Consider an affine representation $\rho: G \to \Aff(M)$. 
	Let $g \neq \id$ be an element of $G$. 
	Assume that 	 the eigenvalues of the corresponding linear action of $g$ on $M_\C$ are   roots of unity of order $\ord(g)$.
	Then there exists a unique $g$-fixed point $c \in M_\R$, and $c \in M$ unless $\ord(g)$ is a power of a prime $p$, in which case $c \in \frac{1}{p}M$.
	
\end{lemma}
\begin{proof}
	By Lemma~\ref{l:affinereps}, there is a $g$-fixed point $c$ in $M_\R$. 	Since $1$ is not an eigenvalue of $g$, $c$ is unique.  
%
	Suppose that a prime $p$ divides $\ord(g)$. Then $g' = g^{\frac{\ord(g)}{p}}$ has order $p$ and the eigenvalues of the corresponding linear action of $g'$ on $M_\C$ are primitive $p$th roots of unity. Let $c' \in \frac{1}{p}M$ be a $g'$-fixed point. 
		Since $1$ is not an eigenvalue of $g'$, $c'$ is unique and hence $c = c'$.
%
	If a prime $q \neq p$ also divides $\ord(g)$, then $c \in \frac{1}{p} M \cap \frac{1}{q}M = M$. 
	
\end{proof}


A \emph{complex representation} of $G$ is a group homomorphism $\phi: G \to GL(V)$ for some complex vector space $V$.
 Equivalently, a complex represention is a complex vector space $V$ with the structure of a $\C G$-module, where $\C G$ denotes the complex group algebra of $G$. 
In this paper, we will always assume that $V$ is finite-dimensional.
It is well known that we may also identify the complex representation $\phi$ with its associated \emph{character} $\chi: G \to \C$ defined by $\chi(g) = \tr(\phi(g))$ for all $g \in G$, where $\tr$ denotes the trace of a linear transformation.  
Given a subgroup $H$ of $G$ and a $\C H$-module $W$, the \emph{induced representation} is the $\C G$-module $\Ind_H^G W := \C G \otimes_{\C H} W$.
The \emph{trivial representation} $\1$ is the vector space $\C$ with $G$ acting trivially. 
A \emph{permutation representation} is a complex $\C G$-module isomorphic to a direct sum of $\C G$-modules of the form $\Ind_H^G \1$, for a subgroup $H$ of $G$. 
For example, when $H = \{ 1\}$, the induced representation $\chi_{\reg} := \Ind_H^G \1$ is the \emph{regular representation} of $G$. 

The \emph{complex representation ring} $R(G)$ of $G$ is the free abelian group generated by isomorphism classes of finite-dimensional $\C G$-modules, 
modulo the relation $[V \oplus W] = [V] + [W]$ for $\C G$-modules $V,W$. Multiplication is defined by $[V][W] := [V \otimes_\C W]$ and the identity element $1 \in R(G)$ is the class $[\1]$ of the trivial representation.
Let $\Irr(G)$ be the set of isomorphism classes of irreducible $\C G$-modules. 
Then $R(G)$ is a free $\Z$-module with basis $\{ [V] : V \in \Irr(G) \}$.  
An element $r$ of $R(G)$ is \emph{effective} if has the form $r = [V]$ for some $\C G$-module $V$. Equivalently, if we write $r = \sum_{V \in \Irr(G)} a_V [V]$ for some uniquely defined $a_V \in \Z$, then $r$ is effective if and only if $a_V \ge 0$ for all $V \in  \Irr(G)$. We say that a power series 
in $R(G)[[t]]$ is effective if all its coefficients  are effective. 
In computations, if $\chi$ is a complex representation of $G$, we will sometimes write $\chi \in R(G)$ instead of $[\chi] \in R(G)$ for convenience.

There is an isomorphism between $R(G) \otimes_\Z \C$ and the set of all class functions on $G$ i.e. functions $G \to \C$ that are constant on conjugacy classes. The isomorphism takes an effective element of $R(G)$ to its associated character. Given $r$ in $R(G)$ and $g$ in $G$, we let $r(g) \in \C$ denote the evaluation of the class function associated to $r$. 
Given a power series $r(t) = \sum_{m \ge 0} r_m t^m$ in $R(G)[[t]]$, we let $r(t)(g) = \sum_{m \ge 0} r_m(g) t^m \in \C[[t]]$.

If $V$ is a finite-dimensional $\C[G]$-module and 
$m$ is a nonnegative integer,
then 
the exterior power $\bigwedge^m V$ has the structure of a $\C G$-module, and
we write
\begin{equation}\label{e:detdef}
	\det(I - V t) := \sum_{m \ge 0} (-1)^m [\bigwedge^m V] t^m \in R(G)[t].
\end{equation}
%
Here $[\bigwedge^m V] = 1 \in R(G)$ when $m = 0$. 
\begin{remark}\label{r:deteigenvalues}
For any $g$ in $G$, if $\phi(g) \in \GL(V)$ has eigenvalues $(\lambda_1,\ldots,\lambda_r)$, then 
$$\det(I - V t)(g) = \det(I - \phi(g)t) = (1 - \lambda_1 t)\cdots(1 - \lambda_r t).$$ 
See, for example, \cite[Lemma~3.1]{StapledonEquivariant}. 
\end{remark}

\begin{example}\label{e:Masrep}
 Given the affine representation $\rho$ above, the corresponding linear representation $\hat{\rho}$ (see Lemma~\ref{l:affinereps}) gives $M_\C$ the structure of a $\C G$-module, and, under the identification of $\Aff(M)$ as a subgroup of $\GL(\tM)$, $\rho$ gives $\tM_\C$ the structure of a $\C G$-module. We have $\tM_\C \cong M_\C \oplus \1$, 
 $[\tM_\C] = [M_\C] + 1 \in R(G)$ and $\det(I - \tM_\C t) = (1 - t)\det(I - M_\C t) \in R(G)[t]$. 
\end{example}

\begin{remark}\label{r:cyclotomic}
For any positive integer $m$,
 recall that the $m$th cyclotomic polynomial is the unique monic irreducible polynomial in $\Z[t]$ whose roots over $\C$ are precisely the primitive $m$th roots of unity.
For each $g$ in $G$,
the characteristic polynomial of $\rho(g)  \in \GL(\tM)$ is defined over $\Z[t]$ and its roots over $\C$ are roots of unity of order $\ord(g)$.
Hence the irreducible factors of the characteristic polynomial are cyclotomic polynomials. 
%
\end{remark}

\subsection{Equivariant Ehrhart theory}\label{ss:equivariantEhrhart}

In this section, we 
 describe the setup of equivariant Ehrhart theory. This is stated slightly differently but is equivalent to the description in \cite[Section~4]{StapledonEquivariant}. See Remark~\ref{r:altformulation} below for an equivalent setup that will also be used. We also give examples and 
 develop some properties of the relevant invariants.

Let $M \cong \Z^d$ be a  lattice of rank $d$.
Let $G$ be a finite group  with identity element $\id$ and let $\rho: G \to \Aff(M)$ be an affine representation of $G$. 
Let $P \subset M_\R$ be a $G$-invariant $d$-dimensional lattice polytope. 
We mention that if $P$ is a $G$-invariant lattice polytope of dimension less than $d$, one can reduce to the case above by  
replacing $M$ with its intersection with the affine span of $P$ (see Remark~\ref{r:altformulation} below).
Let $\tM = M \oplus \Z$ and recall that we may view $\rho$ as a representation of $\tM$ by identifying $\Aff(M)$ as the subgroup of $\GL(\tM)$ consisting of transformations that preserve projection $\HT: \tM \to \Z$ onto the last coordinate. 
Let $C_P \subset \tM_\R$ denote the cone 
 generated by $P \times \{ 1 \}$ in $\tM_\R$. Then $C_P$ is $G$-invariant. 
 Recall from the introduction that for any positive integer $m$, 
 $L(P,\rho;m)$ in  $R(G)$ is the class of the permutation representation of $G$ acting on $P \cap \frac{1}{m}M$, or, equivalently, on $C_P \cap \tM \cap \HT^{-1}(m)$. 
%
%
%
 Then $L(P,\rho;m)$ is a quasi-polynomial of degree $d$ with coefficients in $R(G) \otimes_\Z \Q$ and constant term $1 \in R(G)$ \cite[Theorem~5.7]{StapledonEquivariant}. 
For any $g$ in $G$, recall that we consider the rational polytope
$P^g = \{ x \in P : g \cdot x = x \} \subset M_\R$. 
Then the quasi-polynomial $L(P,\rho;m)(g)$ obtained by evaluating the associated character at $g$ in $G$ is 
equal to $L(P^g;m)$ \cite[Lemma~5.2]{StapledonEquivariant}. That is, $L(P,\rho;m)$ precisely encodes the Ehrhart theory of the rational polytopes $\{ P^g : g \in G \}$. 

\begin{definition}\label{d:seriesdefs}
The \emph{equivariant Ehrhart series} is the power series
\[
\Ehr(P,\rho;t) := \sum_{m \ge 0} L(P,\rho;m) t^m \in R(G)[[t]].
\]
The \emph{equivariant $h^*$-series} is the power series
\[
h^*(P,\rho;t) := \det(I - \tM_\C t) \Ehr(P,\rho;t) \in R(G)[[t]].
\] 
When $h^*(P,\rho;t)$ is a polynomial we also call it the \emph{equivariant $h^*$-polynomial}. 
\end{definition}
Observe that 
$\Ehr(P,\rho;t)(g) = \Ehr(P^g;t)$ for any $g \in G$.
In particular,   $\Ehr(P,\rho;t)(g)$ is a rational function in $t$. 
Since the constant term in $\det(I - \tM_\C t)$  is $1 \in R(G)$,   
$\det(I - \tM_\C t)$ is invertible in $R(G)[[t]]$. Hence we may equivalently write:
\begin{equation}\label{e:hstardef}
\Ehr(P,\rho;t) = \frac{h^*(P,\rho;t)}{\det(I - \tM_\C t)}.
\end{equation}
For example, $h^*(P,\rho;t)(\id) = h^*(P;t)$ is the usual $h^*$-polynomial of $P$.

\begin{example}\label{e:linear}
The constant term of $h^*(P,\rho;t)$ is $1 \in R(G)$ and the coefficient of the linear term is 
$L(P,\rho;1) - [\tM_\C] \in R(G)$. The latter element is effective by 
\cite[Corollary~6.7]{StapledonEquivariant}.
\end{example}

\begin{remark}\label{r:altformulation}
	
	We will also use the following equivalent formulation of equivariant Ehrhart theory. 
	Let $W$ be a real affine space with corresponding real vector space $V$. 
	Let $M'$ be an affine lattice in $W$ with corresponding lattice $M \subset V$. 
	Let $G$ be a finite group with an affine representation $\rho': G \to \Aff(M')$. 
	Let $P' \subset W$ be a lattice polytope with respect to $M'$ that is $G$-invariant. 
	
	With this setup, we may define equivariant Ehrhart invariants as follows. 
	Recall that we may identify $\Aff(M')$  with  $\{ \phi \in \Aff(W) : \phi(M') = M' \}$, and $\phi \in \Aff(W)$ gives rise to an $\R$-linear map
	$\bar{\phi}: V \to V$. Hence $V \otimes_\R \C$ has the structure of a $\C G$-module, which we denote $M_\C'$. 
	We let $\tM_\C' := M_\C' \oplus \1$ be the direct sum of $M_\C'$ and the trivial $\C G$-module. 
	For any positive integer $m$, define $L(P',\rho';m)$ in $R(G)$ to be the permutation representation of $G$ acting on 
	$P' \cap \frac{1}{m} M'$.  Let $\Ehr(P',\rho';t)  = 1 + \sum_{m > 0} L(P',\rho';m) t^m \in R(G)[[t]]$ and 
	$h^*(P',\rho';t) = \Ehr(P',\rho';t) \det(I - \tM_\C' t) \in R(G)[[t]]$. 
	
	We claim that this formulation is equivalent to the formulation in the paper.
	Fix an element $u'$ in $M'$. Recall from Section~\ref{ss:affine} that we have a bijection  $f_{u'} : W \to M_\R$ such that $f_{u'}(\frac{1}{m}M') = \frac{1}{m}M$ for all positive integers $m$, and $f_{u'}$ induces an isomorphism $\Aff(M') \cong \Aff(M)$. Let $\rho: G \to \Aff(M)$ be the induced affine representation, and let $P = f_{u'}(P')$. Then $P$ is a lattice polytope with respect to $M$ that is $G$-invariant. By Example~\ref{e:linear}, $L(P,\rho;0) = 1 \in R(G)$ and, 
	by Example~\ref{e:Masrep}, $\det(I - \tM_\C t)  = (1 - t)\det(I - M_\C t)$. 
	It then follows that 
	$L(P,\rho;m) = L(P',\rho';m)$ for all positive integers $m$, 
	$\Ehr(P,\rho;t) = \Ehr(P',\rho';t)$ and $h^*(P,\rho;t) = h^*(P',\rho';t)$. 
	
	Conversely, given an affine representation $\rho: G \to \Aff(M)$ and a $G$-invariant lattice polytope $P \subset M_\R$, we may regard $W := M_\R$ as an real affine space with corresponding real vector space $V := M_\R$, and $M' :=  M \subset W$ as an affine lattice in $W$ with corresponding lattice $M \subset V$. 
	Let $\rho' = \rho$ and $P' = P$. Let $u'$ in $M'$ be the origin in $M$. Then $f_{u'}: W \to M_\R$ is the identity map, $P = f_{u'}(P')$ and $f_{u'}$ induces the equality $\Aff(M') = \Aff(M)$. In particular, as above, $L(P,\rho;m) = L(P',\rho';m)$ for all positive integers $m$, 
	$\Ehr(P,\rho;t) = \Ehr(P',\rho';t)$ and $h^*(P,\rho;t) = h^*(P',\rho';t)$. 
	
	Finally, in Section~\ref{s:triangulations} and Section~\ref{s:commutativealgebra} we will construct lattice triangulations and consider  commutative algebra invariants respectively associated to $P$, $\rho$ and $M$. Given an affine lattice $M'$ together with an affine representation $\rho': G \to \Aff(M')$ and $G$-invariant lattice polytope $P'$, then, as above, one may choose a point $u'$ in $M'$ and consider 	the corresponding lattice structure on $M'$,
	 and hence consider the corresponding lattice triangulations and commutative algebra invariants. The choice of $u'$ will not be important e.g. different choices of $u'$ give isomorphic commutative algebra invariants. In particular, 
	 conditions \eqref{i:maintriangulation}, \eqref{i:mainnondegenerate} and \eqref{i:mainhilbert} in Theorem~\ref{t:mainfull} are independent of the choice of $u'$. For example, the semigroup $C_P \cap \tM$ from the introduction is isomorphic to the semigroup $(P \cap (\cup_{m > 0} \frac{1}{m}M'))^\times$ 
	 considered in 
	 Section~\ref{ss:affine}, and hence the 
	 commutative ring $S_P$ from the introduction is isomorphic  to the semigroup algebra of $(P \cap (\cup_{m > 0} \frac{1}{m}M'))^\times$.
\end{remark}

\begin{remark}\label{r:effectiveimpliespolynomial}
	If $h^*(P,\rho;t)$ is effective, then $h^*(P,\rho;t)$ is a polynomial (see \cite[Section~7]{StapledonEquivariant}). Indeed, if $h^*(P,\rho;t)$ is effective, then the $h^*$-polynomial $h^*(P;t) = h^*(P,\rho;t)(\id)$ of $P$ encodes the dimensions of the representations of the coefficients of  $h^*(P,\rho;t)$.
\end{remark}

Apriori, $\Ehr(P,\rho;t)$ does not determine $h^*(P,\rho;t)$ or the complex representation $\tM_\C$. However, the following lemma shows that this is indeed true.

\begin{lemma}\label{l:encoding}
		Let $G$ be a finite group.
	Let $M$ be a lattice of rank $d$ and let $\rho: G \to  \Aff(M)$ be an affine representation.  
	Let $P \subset M_\R$ be a $G$-invariant $d$-dimensional lattice polytope.
%
Then the equivariant Ehrhart series $\Ehr(P,\rho;t)$ precisely encodes the equivariant $h^*$-series $h^*(P,\rho;t)$ together with the complex representation $\tM_\C$ of $G$. 
\end{lemma}
\begin{proof}
Given $h^*(P,\rho;t)$ and $\tM_\C$, we recover $\Ehr(P,\rho;t)$ via equation \eqref{e:hstardef}.
Conversely, suppose we are given $\Ehr(P,\rho;t)$. 
If we can show that we can recover the complex representation $\tM_\C$, then we also recover 
$h^*(P,\rho;t)$  via equation \eqref{e:hstardef}.
Hence it remains to determine $\tM_\C$.

 The complex representation $\tM_\C$ is determined by the eigenvalues of the action of the elements of $G$.
Consider an element $g$ in $G$ acting on $\tM_\C$ with eigenvalues $\{ \lambda_1,\ldots,\lambda_{d + 1} \}$, where $d = \dim M_\R$. 
For any positive integer $m$,
let $\mult(g,m) := |\{ \lambda_i : \lambda_i^m = 1 \}|$ and $\mult'(g,m) := |\{ \lambda_i : \lambda_i^m = 1, \lambda_i^{m'} \neq 1 \textrm{ for } m' < m \}|$. 
Then $\mult(g,m) = \sum_{m' | m} \mult'(g,m')$. 
By Remark~\ref{r:cyclotomic}, 
 the multiplicity of a primitive $m$th root of unity $\zeta_m$ as an eigenvalue of $\rho(g)  \in \GL(\tM)$ is independent of the choice of $\zeta_m$ in $\C$. 
It follows that $\tM_\C$ is determined by the numbers 
$\{ \mult'(g,m) : g \in G, m \in \Z_{> 0} \}$, and hence by the numbers $\{ \mult(g,m) : g \in G, m \in \Z_{> 0} \}$, for example, via M\"obius inversion . 
%
%
%
Since $g^m$ acts on $\tM_\C$ with eigenvalues $\{ \lambda_1^m,\ldots,\lambda_{d + 1}^m \}$, we have  
$\mult(g,m) = \mult(g^m,1)$. Hence $\tM_\C$ is determined by the numbers $\{ \mult(g,1) : g \in G \}$. By \cite[Lemma~5.3]{StapledonEquivariant}, $\mult(g,1)$ equals
$\dim P^g$. 
On the other hand, by Theorem~\ref{t:orderpole}, the order of $t = 1$ as a pole of $\Ehr(P,\rho;t)(g)  = \Ehr(P^g;t)$ is equal to $\dim P^g + 1$. 
\end{proof}


\begin{example}\label{e:Symprism1}
	Fix a positive integer $d > 1$. We consider the setup of Remark~\ref{r:altformulation}.
	As in Example~\ref{e:constructaffine}, 
	consider the map  $\psi: \Z^{d + 1} \to \Z$, $\psi(u_1,\ldots,u_d,u_{d + 1}) = \sum_{i = 1}^d u_i$ and corresponding map $\psi_\R: \R^{d + 1} \to \R$. 
	Then $M' = \psi^{-1}(1)$ is an affine lattice in the real affine space $W = \psi_\R^{-1}(1)$. The symmetric group  $G = \Sym_d$ acts on $M'$ by permuting the first $d$ coordinates and acting trivially on the last coordinate. 
	Let $\rho: G \to \Aff(M')$ denote the corresponding affine representation of $G$.
	Let $Q = \Conv{e_1,\ldots,e_d} \subset W$ and  $P = Q \times [0,e_{d+ 1}] \subset W$. If $\rho_Q$ denotes the corresponding affine representation of $G$ on   $M' \cap \Z^d$, then \cite[Proposition~6.1]{StapledonEquivariant} implies that $h^*(Q,\rho_Q;t) = 1$. 
    Let $g \in G$ be an element of cycle type $(\ell_1,\ldots,\ell_r)$.
Then $\Ehr(P,\rho;t)(g) = \sum_{m \ge 0} (m + 1)L(Q,\rho_Q;t)(g) = \partial_t (t\Ehr(Q,\rho_Q;t)(g))$ 
and
$\det(I  - M_\C' t)(g) = \prod_i (1 - t^{\lambda_i})$.
After a calculation we obtain:
\begin{equation}\label{e:Symhstar}
	h^*(P,\rho;t)(g) =     (1 - t)\left(1  + \sum_{i = 1}^r  \frac{\ell_i t^{\ell_i}}{1 - t^{\ell_i}} \right).
\end{equation}
Let $\zeta \in \C$ be a primitive $m$th root of unity for some positive integer $m > 1$. Then 
$$\lim_{t \to \zeta} (1 - t^m) h^*(P,\rho;t)(g) = (1 - \zeta)m|\{ \ell_i : m | \ell_i \}|.
$$ 
It follows that $h^*(P,\rho;t)(g)$ has a simple pole at $\zeta$ if $m | \ell_i$ for some $1 \le i \le r$. 
Let $H$ be any nontrivial subgroup of $G$ and let $\rho|_H$ denote the restriction of $\rho$ to $H$. We conclude that $h^*(H,\rho|_H;t)$ is not a polynomial.

\end{example}

\begin{example}\label{e:primenofixedpts}
	Suppose that $G = \Z/p\Z$ 
	for a prime $p$ and that 
	$G$ acts on $M$ with a unique $G$-fixed point $c$. For example, when $p = 2$, $P$ is called \emph{centrally symmetric}.  
	We will show that $h^*(P,\rho;t)$ is effective and, moreover, its coefficients are permutation representations. This generalizes the case when $p = 2$ and $c \in M$ \cite[Section~11]{StapledonEquivariant}. 
	
	Let $\deg(c)$ be the degree of $M$ as a sublattice of $M + \Z c \subset M_\Q$. 
	By Lemma~\ref{l:bounddenomiinatorfixedgeneral}, $c \in \frac{1}{p}M$, and hence
	\[
	\deg(c) =  \begin{cases}
		1 &\textrm{ if } c \in M, \\
		p
		&\textrm{ otherwise. }
	\end{cases}
	\]
	Fix $g \neq \id$ in $G$. Then
		$\Ehr(P,\rho;t)(g) = \frac{1}{1 - t^{\deg(c)}}.$	
	On the other hand, $g$ acts on $M_\C$ with eigenvalues equal to primitive $m$th roots of unity. By Remark~\ref{r:cyclotomic}, $\det(I - \tM_\C t)(g)$ equals 
	$(1 - t)(1 + t + \cdots + t^{p-1})^{\frac{d}{p - 1}}$. In particular, $p - 1$ divides $d$. Recall that $\chi_{\reg}$ denotes the regular representation of $G$, and both the regular representation and trivial representation are permutation representations.
    We deduce that	
	\begin{align*}
		h^*(P,\rho;t)
		&= \frac{h^*(P;t) - (1 + t + \cdots + t^{p-1})^{\frac{d + 1 - \deg(c)}{p - 1}}}{p} \chi_{\reg} + (1 + t + \cdots + t^{p-1})^{\frac{d + 1 - \deg(c)}{p - 1}}.
	\end{align*}
	It follows that the coefficients of $h^*(P;t) - (1 + t + \cdots + t^{p-1})^{\frac{d + 1 - \deg(c)}{p - 1}}$ are divisible by $p$, and $h^*(P,\rho;t)$ is effective if and only if $h^*(P;t) \ge (1 + t + \cdots + t^{p-1})^{\frac{d + 1 - \deg(c)}{p - 1}}$, where the inequality is taken coefficientwise. 
	
	It remains to deduce the inequality. 
	Fix a generator $\sigma$ of $G$.  
	Let $\{ u_1, \ldots, u_r \} \in M$ be a minimal set of vertices of $P$ with the property that the orbits 
	$G \cdot u_1,\ldots, G \cdot u_r$ affinely span $M_\Q$. Let $V = M_\Q - c$ considered as a $\Q G$-module. 	Let $V_i \subset V$ be the linear span of 
	$(G \cdot u_i) - c$. We have an isomorphism of $\Q G$-modules,  $\phi_i : V_i \to  \Q^p/\Q(1,\dots,1)$, $\phi_i(\sigma^k u_i - c) = e_k$, where $\sigma$ acts on the right hand side by cyclically shifting coordinates. The characteristic polynomial of $\sigma$ acting on $V_i$ is $1 + t + \cdots + t^{p - 1}$, an irreducible polynomial in $\Q[t]$, and hence $V_i$ is an irreducible $\Q G$-module.  
	Minimality then implies that  $V_i \cap V_j = \{ 0 \}$ for all $i \ne j$, and 
	$V = V_1 \oplus \cdots \oplus V_r$. 
	Let  $Q \subset P$ be the convex hull of  $G \cdot u_1,\ldots, G \cdot u_r$ in $M_\R$. Stanley's monotonicity theorem (the nonequivariant case of Theorem~\ref{t:monotonicity})  implies that $h^*(Q;t) \le h^*(P;t)$ \cite[Theorem~3.3]{StanleyMonotonicity}. 
	Let $M'$ be the affine sublattice of $M$ spanned by $G \cdot u_1,\ldots, G \cdot u_r$, together with $c$ in the case when $c \in M$. Explicitly, 
	$$
	\phi_i (M' \cap V_i) = \begin{cases}
		\Z^p/\Z(1,\dots,1) &\textrm{ if } c \in M, \\
		\{ (a_1,\ldots,a_p) \in \Z^p/\Z(1,\dots,1) : \sum_i a_i = 1 \mod p \}
		&\textrm{ otherwise. }
	\end{cases}
	$$
	Let $Q'$ be $Q$ considered as a lattice polytope with respect to $M' \subset M$. Then 
	$h^*(Q';t) \le h^*(Q;t)$ by lattice-monotonicity of lattice polytopes \cite[Corollary~2.21]{BKNThinPolytopes}. We may directly compute that 
	$h^*(Q';t) = (1 + t + \cdots + t^{p-1})^{\frac{d + 1 - \deg(c)}{p - 1}}$ using 
	 \cite[Corollary~5.9]{BJMLatticePointGenerating}, and the inequality follows.  
	
	We mention that when $c \in M$, the inequality 
	 $h^*(P;t) \ge (1 + t + \cdots + t^{p-1})^{\frac{d}{p - 1}}$ may also be deduced from \cite[Corollary~1.2]{KaruEhrhart} and \cite[Theorem~1.2]{ACampoGeneralizedHVectors}. Furthermore, when $c \in M$ and $p = 2$, this inequality is originally due to \cite[Remark~1.6]{BHWRotesEhrhart}.

\end{example}


\begin{example}\label{e:d+2v2}
As in Example~\ref{e:d+2v1}, assume that $d$ is an even positive integer, and let $r = \frac{d}{2} + 1$. 
Let $a_1,\ldots,a_r$ be positive integers with $\gcd(a_1,\ldots,a_r) = 1$.
Let $e_1,\ldots,e_r,f_1,\ldots,f_r$ be a basis of $\Z^{d + 2}$. 
Let $\bar{e}_i$ and $\bar{f_i}$ denote the images of $e_i$ and $f_i$ respectively in the lattice
$L := \Z^{d + 2}/\Z(\sum_i a_i (e_i - f_i))$. 
Consider the group homomorphism  $\psi: L \to \Z$ satisfying $\psi(\bar{e}_i) = \psi(\bar{f}_i) = 1$ for all $1 \le i \le r$, 
and the corresponding linear map $\psi_\R: \R^{d + 1} \to \R$.	
As in Example~\ref{e:constructaffine},  
 $M' = \psi^{-1}(1)$ is an affine lattice in the real affine space $W = \psi_\R^{-1}(1)$. 
 Let $P = \Conv{\bar{e}_1,\ldots,\bar{e}_r,\bar{f}_1,\ldots,\bar{f}_r} \subset W$. Then $P$ is a lattice polytope with respect to $M'$ with $2r = d + 2$ vertices and normalized volume $\Vol(P) = \sum_i a_i$.
 Consider an affine representation $\rho: G \to \Aff(M')$ such that $P$ is $G$-invariant. 
 Suppose there exists an element $g$ in $G$ such that $g \cdot \{ \bar{e}_1,\ldots,\bar{e}_r \} = \{ \bar{f}_1,\ldots,\bar{f}_r\}$. 
 After possibly relabelling coordinates, we may assume that 
  there exists a partition $I_1,\ldots,I_s$  of  $\{1,\ldots,r\}$ such that the $\langle g \rangle$-orbits of the vertices of $P$ are precisely $\{ \bar{e_i}, \bar{f_i} : i \in I_k \}$ for $1 \le k \le s$.
 By \cite[Lemma~5.4]{StapledonEquivariant}, the fixed locus $P^g$ is the rational simplex with vertices $\{ \frac{1}{2|I_k|} \sum_{i \in I_k} (\bar{e}_i + \bar{f}_i) :  1 \le k \le s \}$. By Example~\ref{e:d+2v1},
 \[
 \Ehr(P,\rho;t)(g) = \Ehr(P^g;t) = \frac{1 + t^{|\{ i : a_{i}\textrm{ odd } \}|}}{\prod_k (1 - t^{2|I_k|})}.
 \]
 On the other hand, 
 $\det(I - \tM_\C' t)(g) = \frac{1}{1 + t}\prod_k (1 - t^{2|I_k|})$. 
 Hence
 \[
 h^*(P,\rho;t)(g) = \frac{1 + t^{|\{ i : a_{i}\textrm{ odd } \}|}}{1 + t}.
 \]
 We deduce that $h^*(P,\rho;t)(g)$ is a polynomial if and only if $|\{ a_i : a_i \textrm{ odd } \}|$ is odd. The latter condition holds if and only if $\Vol(P) = \sum_i a_i$ is odd.
\end{example}

We have the following equivariant version of Ehrhart reciprocity (see \cite[Theorem~5.7]{StapledonEquivariant} for an alternative formulation). 
Let $P^\circ$ denote the relative interior of $P$.
Recall that the \emph{codegree} of 
$P$ is 
$\codeg(P) := \min \{ m \in \Z_{> 0} : P^\circ \cap \frac{1}{m}M \neq \emptyset \}$ and the \emph{degree}
of $P$ is $\deg(P) := d + 1 - \codeg(P)$. Ehrhart reciprocity (the non-equivariant version of Proposition~\ref{p:reciprocity} below) implies that $\deg(P) = \deg h^*(P;t)$.  
For a positive integer $m$, let 
$L(P^\circ,\rho;m)$ in  $R(G)$ be the class of the permutation representation of $G$ acting on $P^\circ \cap \frac{1}{m}M$. 

\begin{proposition}\label{p:reciprocity}\cite[Corollary~6.6]{StapledonEquivariant}
		Let $G$ be a finite group.
Let $M$ be a lattice of rank $d$ and let $\rho: G \to  \Aff(M)$ be an affine representation.  
Let $P \subset M_\R$ be a $G$-invariant $d$-dimensional lattice polytope.
For any $g$ in $G$, we have the following equality of rational functions:
\[
\sum_{m > 0} L(P^\circ,\rho;m)(g) t^m = \frac{t^{d + 1} h^*(P,\rho;t^{-1})(g)}{\det(I - \tM_\C t)(g)}.
\]
Here $h^*(P,\rho;t^{-1})(g)$ is the rational function obtained from $h^*(P,\rho;t)(g) \in \C(t)$ by replacing $t$ with $t^{-1}$. 

In particular, if $h^*(P,\rho;t)$ is a polynomial, then $\deg h^*(P,\rho;t) = \deg (P)$ and the leading coefficient is $L(P^\circ,\rho; \codeg(P)) \in R(G)$. 
\end{proposition}

\begin{remark}
In the statement of Proposition~\ref{p:reciprocity} above, the reason for evaluating all terms at an element $g$ in $G$ is that one needs to be careful with the definition of $h^*(P,\rho;t^{-1})$, especially when $h^*(P,\rho;t)$ is not a polynomial  (cf. Remark~\ref{r:warning}). 
More precisely, above we view $h^*(P,\rho;t^{-1})$  as a class function with values in $\C(t)$. 
 \end{remark}




\begin{remark}\label{r:interiorpointsexist}
	Later we will need the well-known fact that  $P^\circ \cap \frac{1}{m} M \neq \emptyset$ for all $m > \dim P$. This follows immediately from Ehrhart reciprocity and the discussion above. Indeed, consider the case when $G$ is trivial and write $L(P^\circ;m) = L(P^\circ,\rho;m)$ for any positive integer $m$.
	Then $L(P^\circ;m) \le L(P^\circ;m + 1)$ since if $v$ is a vertex of $P$ then $\{ \frac{v}{m + 1} + \frac{mu}{m + 1} : u \in P^\circ \cap \frac{1}{m} M\} \subset P^\circ \cap \frac{1}{m + 1} M$. In particular, $L(P^\circ;m) > 0$ for $m \ge \codeg(P) = \dim P + 1 - \deg(P)$. 
\end{remark}

%

\begin{remark}\label{r:warning}
We issue a warning that $R(G)$ is not necessarily a domain, so one can not discuss poles of $\Ehr(P,\rho;t)$, only poles of $\Ehr(P,\rho;t)(g) \in \C(t)$ for some $g \in G$. Here is a concrete example to show how problems can arise. Let $G = \Z/2\Z$ act on $M = \Z^2$ sending $(x,y)$ to $(y,x)$, and let $P = [0,1]^2$. Let $\chi$ denote the nontrivial $1$-dimensional representation of $G$. 
Then $R(G)$ is not a domain since $(1 + \chi)(1 - \chi) = 0 \in R(G)$. 
We compute $h^*(P,\rho;t) = 1 + t$ and $\det(I - \tM_\C t) = 1 - (2 + \chi) t + (1 + 2\chi) t^2 - \chi t^3$.
We have the following two ways of expressing $\Ehr(P,\rho;t)$ as an element of $R(G)[t]$ divided by an element of $R(G)[t]$ with invertible constant term. 
\[
\Ehr(P,\rho;t) = \frac{1 + t}{1 - (2 + \chi) t + (1 + 2\chi) t^2 - \chi t^3} 
= \frac{1 + \chi t}{(1 - t)^3}.
\]
\end{remark}

Finally, we show that both the equivariant Ehrhart series and the equivariant $h^*$-series are multiplicative with respect to the free join of polytopes. This generalizes the non-equivariant case (see \cite[Lemma~1.3]{HTLowerBounds}).
Let $P_1 \subset (M_1)_\R$ and
$P_2 \subset (M_2)_\R$ be full-dimensional lattice polytopes with respect to lattices  $M_1$ and $M_2$ respectively. For $i \in \{1,2\}$, let $\tM_i = M_i \oplus \Z$ and consider the corresponding projection map $\HT_i: \tM_i \to \Z$ onto the last coordinate.  Consider the 
map $\psi: \tM_1 \oplus \tM_2 \to \Z$ defined by $\psi(u_1,u_2) = \HT_1(u_1) + \HT_2(u_2)$ for all $(u_1,u_2) \in \tM_1 \oplus \tM_2$.
As in Example~\ref{e:constructaffine},  $M' = \psi^{-1}(1)$ is an affine lattice in the real affine space $W = \psi_\R^{-1}(1)$. If $0_{\tM_i}$ denotes the origin in $\tM_i$, then \emph{free join} $P_1 * P_2$ is the convex hull of
$P_1 \times \{ 1 \} \times \{ 0_{\tM_2}\}$ and $\{ 0_{\tM_1}\} \times P_2 \times \{ 1 \}$ in $W$. We have $\dim (P_1 * P_2) = \dim W$. 

Let $G_1$ and $G_2$ be finite groups. For $i \in \{1,2\}$,
given a complex representation of $G_i$, we obtain a complex representation of $G_1 \times G_2$ by composition with the natural projection $G_1 \times G_2 \to G_i$. This induces a map $j_i: R(G_i) \to R(G_1 \times G_2)$, which one may verify is an inclusion, for example, by considering the corresponding map of class functions. 

Let $\rho_1 : G_1 \to \Aff(M_1)$ and $\rho_2 : G_2 \to \Aff(M_2)$ be affine representations. Then we have an induced representation 
$\rho_1 \times \rho_2 : G_1 \times G_2 \to \GL(\tM_1) \times \GL(\tM_2) = \GL(\tM_1 \oplus \tM_2)$ that preserves $W$. Assume that $P_i$ is $G_i$-invariant for $i \in \{1,2\}$. Then $P_1 * P_2 \subset W$ is a $(G_1 \times G_2)$-invariant lattice polytope with respect to $M'$.  Below, for $i \in \{1,2\}$,  we regard $\Ehr(P_i,\rho_i;t)$ and $h^*(P_i,\rho_i;t)$ as invariants in $R(G_1 \times G_2)[[t]]$ via the inclusion $j_i: R(G_i) \to R(G_1 \times G_2)$.

\begin{lemma}\label{l:freejoin}
	For $i \in \{1,2\}$, let $G_i$ be a finite group and let $\rho_i: G_i \to \Aff(M_i)$ be an affine representation for some lattice $M_i$. Let $P_i \subset (M_i)_\R$ be a full-dimensional $G_i$-invariant lattice polytope. Then 
     we have the following equalities in $R(G_1 \times G_2)[[t]]$:
	\[
	\Ehr(P_1 * P_2,\rho_1 \times \rho_2;t) = \Ehr(P_1,\rho_1;t)\Ehr(P_2,\rho_2;t),
	\]
	and 
	\[
	h^*(P_1 * P_2,\rho_1 \times \rho_2;t) = h^*(P_1,\rho_1;t)h^*(P_2,\rho_2;t).
	\]
\end{lemma}
\begin{proof}
	If $C_i$ denotes the cone generated by $P_i \times \{ 1 \}$ in $(\tM_{i})_\R$ for $i \in \{1,2\}$, then the Minkowski sum of $C_{1} \times  \{ 0_{\tM_2}\}$ and $\{ 0_{\tM_1}\} \times C_{2}$ is the cone $C$ generated by   $P_1 * P_2$ in $(M_1)_\R \oplus (M_2)_\R$. Fix a positive integer $m$. Then $L(P_1 * P_2, \rho_1 \times \rho_2;m)$ is the class of the permutation representation of $G_1 \times G_2$ acting on $C \cap \psi^{-1}(m)$. 
	On the other hand, the latter expression equals $\sum_{0 \le \ell \le m} L(P_1,\rho_1;\ell)L(P_2,\rho_2;m - \ell)$. This implies the first equality after expanding the definitions.
	Since 
	$(\tM_1)_\C \oplus (\tM_2)_\C  \cong \tM_\C'$ as $\C G$-modules,
	the second equality follows if we can prove the identity
	\begin{equation}\label{e:detproduct}
		\det(I - ((\tM_1)_\C \oplus (\tM_2)_\C )t) = \det(I - (\tM_1)_\C t)\det(I - (\tM_2)_\C t).
	\end{equation} 
	This follows from Remark~\ref{r:deteigenvalues}. Explicitly, if $g_i$ in $G_i$ acts on $(\tM_i)_\C$ with eigenvalues $\{ \lambda_{i,1}, \ldots, \lambda_{i,d_i} \}$ for $i \in \{1,2\}$, then 
	$(g_1,g_2)$ acts on $(\tM_1)_\C \oplus (\tM_2)_\C$ with eigenvalues $\{ \lambda_{1,1}, \ldots, \lambda_{1,d_1},\lambda_{2,1}, \ldots, \lambda_{2,d_2} \}$.
	
\end{proof}

\begin{example}\label{e:pyramid}
With the notation above, let $G_1 = G$, $\rho_1 = \rho$ and $P_1 = P$. Let $G_2 = \{1\}$ act on  $P_2 = M_2 = \{ 0 \}$.
	 Then the free join $P_1 * P_2$ is the \emph{pyramid} $\Pyr(P)$ of $P$. The pyramid $\Pyr(P)$ may be identified with the the convex hull of $P \times \{ 1 \}$ and the origin in $\tM_\R$ with $G$ acting via $\rho: G \to \GL(\tM)$. 
In this case, Lemma~\ref{l:freejoin} implies that
$(1 - t)\Ehr(\Pyr(P),\rho;t) = \Ehr(P,\rho;t)$ and 
$h^*(\Pyr(P),\rho;t) = h^*(P,\rho;t)$ (see \cite[Remark~6.2]{StapledonEquivariant}). 
\end{example}

\subsection{The equivariant $h^*_N$-series}\label{ss:equivarianthNstar}

In this section, we fix a positive integer $N$.
We also continue with the notation from the previous section. That is, let $G$ be a finite group and let $\rho: G \to \Aff(M)$ be an affine representation of $G$ for some lattice $M \cong \Z^d$. Let $P \subset M_\R$ be a $G$-invariant $d$-dimensional lattice polytope.

\begin{definition}
The \emph{equivariant $h_N^*$-series} is the power series
\begin{equation}\label{e:hNstardef}
h_N^*(P,\rho;t) := \det(I - \tM_\C t) \Ehr(P,\rho;t^{\frac{1}{N}}) \in R(G)[[t^{\frac{1}{N}}]].
\end{equation}
\end{definition}
As in equation \eqref{e:hstardef}, we 
may alternatively write \eqref{e:hNstardef}
in either of the forms:
\begin{equation*}
\Ehr(P,\rho;t^{\frac{1}{N}}) = \frac{h_N^*(P,\rho;t)}{\det(I - \tM_\C t)}, \textrm{ or }
\Ehr(P,\rho;t) = \frac{h_N^*(P,\rho;t^N)}{\det(I - \tM_\C t^N)}.
\end{equation*} 
Alternatively, we can express the equivariant $h^*_N$-series in terms of the equivariant $h^*$-series:
\begin{equation}\label{e:hNstarintermsofh}
h_N^*(P,\rho;t) = h^*(P,\rho;t^{\frac{1}{N}}) \frac{\det(I - \tM_\C t)}{\det(I - \tM_\C t^{\frac{1}{N}})}.
\end{equation}

\begin{example}\label{e:Symprism2}
	We continue with Example~\ref{e:Symprism1}.
	Recall that $d > 1$, $\psi: \Z^{d + 1} \to \Z$, $\psi(u_1,\ldots,u_d,u_{d + 1}) = \sum_{i = 1}^d u_i$, and $M' = \psi^{-1}(1)$ is the affine lattice in the real affine space $W = \psi_\R^{-1}(1)$. Recall that $\rho: G \to \Aff(M')$ is the affine representation with  $G = \Sym_d$ acting by permuting coordinates in the first $d$ coordinates, $Q = \Conv{e_1,\ldots,e_d} \subset W$ and  $P = Q \times [0,e_{d+ 1}] \subset W$.
	 Let $g \in G$ be an element of cycle type $(\ell_1,\ldots,\ell_r)$.
	By \eqref{e:Symhstar} and \eqref{e:hNstarintermsofh}, 	 for any positive integer $N$:
	\[
	h^*_N(P,\rho;t)(g) =     (1 - t)\left(1  + \sum_{i = 1}^r  \frac{\ell_i t^{\frac{\ell_i}{N}}}{1 - t^{\frac{\ell_i}{N}}} \right) \prod_{i = 1}^r \frac{1 - t^{\ell_i}}{1 - t^{\frac{\ell_i}{N}}}.
	\]	 
	 	If 	 $\lcm(\ell_1,\ldots,\ell_r) | N$, then $h^*_N(P,\rho;t^N)(g) \in \Z[t]$.
	If 	 $\lcm(\ell_1,\ldots,\ell_r) \nmid N$, then there exists $1 \le k \le r$ such that $\ell_k \nmid N$. Let $m = \frac{\ell_k}{\gcd(\ell_k,N)} > 1$ and let 
	$\zeta \in \C$ be a primitive $m$th root of unity. 
	Since $\gcd(m,N) = 1$ and $m | \ell_k$, we compute	 	 
	$$\lim_{t \to \zeta} (1 - t^m) h^*_N(P,\rho;t^N)(g) = (1 - \zeta^N)m|\{ \ell_i : m | \ell_i \}| \lim_{t \to \zeta} \prod_{i = 1}^r \frac{1 - t^{N\ell_i}}{1 - t^{\ell_i}} \neq 0.
	$$  
	Hence $h^*_N(P,\rho;t^N)(g)$ has a simple pole at $\zeta$. 
	Since the order $\ord(g)$ of $g$ equals $\lcm(\ell_1,\ldots,\ell_r)$, we deduce that   
	$h^*_N(P,\rho;t)(g)$ is a polynomial in $\Z[t^{\frac{1}{N}}]$ if and only if
	$\ord(g) | N$. 
	Let $H$ be any subgroup of $G$ and let $\rho|_H$ denote the restriction of $\rho$ to $H$. Recall that $\exp(H) = \lcm( \ord(h) : h \in H )$ is the exponent of $H$.
	We conclude that $h^*_N(P,\rho|_H;t)$ is a polynomial if and only if 
	$\exp(H) | N$.
\end{example}


\begin{example}\label{e:permutahedronv1}
	We consider the setup of Remark~\ref{r:altformulation}.
	As in Example~\ref{e:constructaffine}, 
	consider the map  $\psi: \Z^{d + 1} \to \Z$, $\psi(u_1,\ldots,u_{d + 1}) = \sum_{i = 1}^{d + 1} u_i$ and corresponding map $\psi_\R: \R^{d + 1} \to \R$. 
	Then $M' = \psi^{-1}(\binom{d + 2}{2} )$ is an affine lattice in the real affine space $W = \psi_\R^{-1}(\binom{d + 2}{2} )$. The symmetric group  $G = \Sym_{d + 1}$ acts on $M'$ by permuting coordinates.  
	Let $\rho: G \to \Aff(M')$ denote the corresponding affine representation of $G$.
	Let $P$ be the convex hull of the orbit $G \cdot (1,2,\ldots,d + 1) \subset W$.
	Then $P$ is the \emph{$(d + 1)$-permutahedron}. The equivariant Ehrhart theory of $P$ was studied in \cite{ASV20} and \cite{ASV21}. In particular, explicit combinatorial formulas for $L(P,\rho;m)$ and $h^*(P,\rho;t)(g)$ were given in \cite[Theorem~1.1]{ASV20} and  \cite[Proposition~5.1]{ASV20} respectively. In particular, they showed that if $g$ has cycle type $(\ell_1,\ldots,\ell_r)$, then 
	$h^*(P,\rho;t)(g)$ has a pole at $t = -1$ if and only if $h^*(P,\rho;t)(g)$ is not a polynomial if and only if the number $e$ of even parts of the cycle type of $g$ satisfies $0 < e < r - 1$
	\cite[Lemma~5.2]{ASV20}. 
	They deduced that $h^*(P,\rho;t)$ is a polynomial if and only if $h^*(P,\rho;t)$ is effective if and only if $d \le 2$ \cite[Proposition~5.3-5.4]{ASV20}. 
	Since $\det(I  - \tM_\C' t)(g) = \prod_i (1 - t^{\lambda_i})$,
	we compute:
	$$
	h^*_N(P,\rho;t^N)(g) = h^*(P,\rho;t)(g) \prod_i \frac{ 1 - t^{N\lambda_i}}{ 1 - t^{\lambda_i}}
	$$
	It follows that if $h^*(P,\rho;t)(g)$ is not a polynomial and $N$ is odd, then  $h^*_N(P,\rho;t)(g)$ is not a polynomial. 
	When $N$ is even, we will see in Example~\ref{e:permutahedronv2} that $P$ admits a $G$-invariant triangulation with vertices in $\frac{1}{N}M'$ and hence Corollary~\ref{c:maincorollaryshort} implies that $h^*_N(P,\rho;t)$ is a polynomial. 
	We conclude that $h^*_N(P,\rho;t)(g)$ is not a polynomial if and only if $N$ is odd and  the number $e$ of even parts of the cycle type of $g$ satisfies $0 < e < r - 1$. 
	
\end{example}

\begin{example}\label{e:d+2v3}
	In Example~\ref{e:d+2v2}, one verifies that $h^*_N(P,\rho;t)(g)$ is a polynomial if and only if either $\Vol(P) = \sum_i a_i$ is odd or $N$ is even. 
\end{example}

Below we will establish some properties of the equivariant $h^*_N$-series.
The following result is equivalent  to equivariant Ehrhart reciprocity. Indeed, we deduce Proposition~\ref{p:Nreciprocity} as a consequence of Proposition~\ref{p:reciprocity}, 
while Proposition~\ref{p:reciprocity} is the special case of Proposition~\ref{p:Nreciprocity} when $N = 1$. 

\begin{proposition}\label{p:Nreciprocity}\cite[Corollary~6.6]{StapledonEquivariant}
			Let $G$ be a finite group.
	Let $M$ be a lattice of rank $d$ and let $\rho: G \to  \Aff(M)$ be an affine representation.  
	Let $P \subset M_\R$ be a $G$-invariant $d$-dimensional lattice polytope.
	Fix a positive integer $N$.
	For any $g$ in $G$, we have the following equality of rational functions:
 
\begin{equation}\label{e:Nreciprocity}
\sum_{m > 0} L(P^\circ,\rho;m)(g) t^{\frac{m}{N}} = \frac{t^{d + 1} h^*_N(P,\rho;t^{-1})(g)}{\det(I - \tM_\C t)(g)}.
\end{equation}
Here $h^*_N(P,\rho;t^{-1})(g)$ is the rational function obtained from $h^*_N(P,\rho;t)(g) \in \C(t^{\frac{1}{N}})$ by replacing $t$ with $t^{-1}$. 

In particular, if $h_N^*(P,\rho;t)$ is a polynomial, then 
$\deg h_N^*(P,\rho;t) = 
d + 1 - \frac{\codeg(P)}{N}
$ 
and the leading coefficient is $L(P^\circ,\rho; \codeg(P)) \in R(G)$. 
\end{proposition}
\begin{proof}
The second statement follows by expanding both sides of \eqref{e:Nreciprocity}. It remains to establish 
\eqref{e:Nreciprocity}. We will use the following equality
\begin{equation}\label{e:detequality}
t^{d + 1} \det(I - \tM_\C t^{-1})(g) = \det(\rho(g)) \det(I - \tM_\C t)(g). 
\end{equation}
This was established in the proof of \cite[Corollary~6.6]{StapledonEquivariant}, but we will repeat the argument for the benefit of the reader. Suppose that $g$ acts on $\tM_\C$ with eigenvalues $\{ \lambda_1,\ldots,\lambda_{d + 1} \}$. 
Then the above equality is equivalent to
\begin{equation}\label{e:dettemp}
t^{d + 1}(1 - \lambda_1 t^{-1})\cdots(1 - \lambda_{d + 1} t^{-1}) = \lambda_1\cdots\lambda_{d + 1}(1 - \lambda_1 t)\cdots(1 - \lambda_{d + 1} t).
\end{equation}
Since $\rho$ is real-valued, the set of eigenvalues with multiplicity is invariant under complex conjugation. 
Also, since 
%
each eigenvalue $\lambda_i$ is a root of unity, $\bar{\lambda}_i = \lambda_i^{-1}$. 
We deduce that 
%
$(1 - \lambda_1 t)\cdots(1 - \lambda_{d + 1} t) = (1 - \lambda_1^{-1} t)\cdots(1 - \lambda_{d + 1}^{-1} t)$. Substituting into the right hand side of \eqref{e:dettemp} and rearranging establishes \eqref{e:detequality}.

By Proposition~\ref{p:reciprocity} and \eqref{e:detequality}: 
\begin{align*}
\sum_{m > 0} L(P^\circ,\rho;m)(g) t^{\frac{m}{N}} = \frac{t^{\frac{d + 1}{N}} h^*(P,\rho;t^{-\frac{1}{N}})(g)}{\det(I - \tM_\C t^{\frac{1}{N}})(g)} = \frac{\det(\rho(g)) h^*(P,\rho;t^{-\frac{1}{N}})(g)}{\det(I - \tM_\C t^{-\frac{1}{N}})(g)}.
\end{align*}
On the other hand, by \eqref{e:hNstarintermsofh} and \eqref{e:detequality}:
\[
\frac{t^{d + 1} h_N^*(P,\rho;t^{-1})(g)}{\det(I - \tM_\C t)(g)} = \frac{t^{d + 1} \det(I - \tM_\C t^{-1})(g)}{\det(I - \tM_\C t)(g)}  \frac{h^*(P,\rho;t^{-\frac{1}{N}})(g)}{\det(I - \tM_\C t^{-\frac{1}{N}})(g)} =  \frac{  \det(\rho(g)) h^*(P,\rho;t^{-\frac{1}{N}})(g)}{\det(I - \tM_\C t^{-\frac{1}{N}})(g)}.
\]
The result follows after comparing the previous two equations.
%

\end{proof}

\begin{example}
We have the following generalization of Example~\ref{e:linear} that describes some of the lowest and highest degree terms of $h^*_N(P,\rho;t)$.
Write
$h^*_N(P,\rho;t) = \sum_{j \in \frac{1}{N}\Z_{\ge 0}} h^*_N(P,\rho)_{j} t^{j}$ for some $h^*_N(P,\rho)_{j} \in R(G)$. Then 
$$h^*_N(P,\rho)_{\frac{m}{N}} = L(P,\rho;m) \textrm{ for } 0 \le m < N,
\textrm{ and  }
h^*_N(P,\rho)_1 = L(P,\rho;N) - [\tM_\C].$$
This follows by expanding the right hand side of \eqref{e:hNstardef}. Here $h^*_N(P,\rho)_1$ is effective by \cite[Corollary~6.7]{StapledonEquivariant} (applied to the corresponding representation $\rho_N: G \to \Aff(\frac{1}{N}M)$). Further assume that $h_N^*(P,\rho;t)$ is a polynomial. Then expanding \eqref{e:Nreciprocity} in Proposition~\ref{p:Nreciprocity} implies that if $s = \deg h^*_N(P,\rho;t) =  d + 1 - \frac{\codeg(P)}{N}$, then 
$$h^*_N(P,\rho)_{s - \frac{m}{N}} = L(P^\circ,\rho;\codeg(P) + m) \textrm{ for } 0 \le m < N.$$ 

\end{example}

\begin{remark}
It follows from Lemma~\ref{l:encoding} and \eqref{e:hNstardef} that 
 the equivariant Ehrhart series $\Ehr(P,\rho;t)$ precisely encodes the equivariant $h^*$-series $h^*_N(P,\rho;t)$ together with the complex representation $\tM_\C$ of $G$, for any fixed choice of $N$.
\end{remark}

\begin{remark}
It follows from Lemma~\ref{l:freejoin} together with \eqref{e:detproduct} and  \eqref{e:hNstardef} that 
 the equivariant $h^*_N$-series is multiplicative with respect to the free join of polytopes. That is, with the notation of Lemma~\ref{l:freejoin}, 
 	\[
 h^*_N(P_1 * P_2,\rho_1 \times \rho_2;t) = h^*_N(P_1,\rho_1;t)h^*_N(P_2,\rho_2;t).
 \]
\end{remark}

\begin{example}\label{e:pyramidN}
As in Example~\ref{e:pyramid}, one may consider the action of $G$ on the pyramid $\Pyr(P)$ of $P$. It follows from Example~\ref{e:pyramid} and \eqref{e:hNstardef} that 
\begin{equation*}
(1 - t^{\frac{1}{N}})h^*_N(\Pyr(P);\rho;t) = (1 - t)h^*_N(P;\rho;t).
\end{equation*}
\end{example}

Consider the $R(G)$-module homomorphisms defined as follows:
\[
\Psi_{\Int}, \Psi_{\Ceil} : \cup_{m > 0} R(G)[[t^{\frac{1}{m}}]] \to R(G)[[t]],
\]
\[
\Psi_{\Int}(t^{j}) = \begin{cases}
t^j &\textrm{ if } j \in \Z, \\
0 &\textrm{ otherwise}
\end{cases},
\textrm{  and }
\Psi_{\Ceil}(t^{j}) = t^{\lceil j \rceil}.
\]
These are related by the identity:
\begin{equation}\label{e:PsiPhi}
\Psi_{\Ceil}(f(t)) = \Psi_{\Int} \left( \frac{1 - t}{1 - t^{\frac{1}{N}}}f(t)\right) \textrm{ for all } f(t) \in R(G)[[t^{\frac{1}{N}}]].
\end{equation}

Recall that  $\rho_N: G \to \Aff(\frac{1}{N} M) \subset \GL(\frac{1}{N} \tM)$ denotes the corresponding affine representation of $G$ acting on $\frac{1}{N} M$. 
Recall from Example~\ref{e:pyramid} that we may consider the induced action of $G$ on the pyramid $\Pyr(P)$ of $P$. 

\begin{lemma}\label{l:intandceil}
For any positive integer $N$, 
$$\Psi_{\Int}(h^*_N(P,\rho;t)) = h^*(P,\rho_N;t)  \textrm{  and } \Psi_{\Ceil}(h^*_N(P,\rho;t)) = h^*(\Pyr(P),\rho_N;t).$$ 
\end{lemma}
\begin{proof}
For any nonnegative integer $m$, $L(P,\rho_N;m) = L(P,\rho;Nm)$. Hence
$$\Ehr(P, \rho_N; t) = \sum_{m \ge 0} L(P,\rho;Nm) (t^{\frac{1}{N}})^{Nm} = \Psi_{\Int}(\Ehr(P,\rho; t^{\frac{1}{N}})).$$
The first equation follows after multiplying both sides by 
$\det(I - \tM_\C t)$.  By \eqref{e:PsiPhi} and Example~\ref{e:pyramidN}, 
\begin{align*}
\Psi_{\Ceil}(h^*_N(P,\rho;t)) = \Psi_{\Int} \left( \frac{1 - t}{1 - t^{\frac{1}{N}}} h^*_N(P,\rho;t) \right) = \Psi_{\Int}(h^*_N(\Pyr(P);\rho;t)). 
\end{align*}
The second equation now follows by applying the first equation to $\Pyr(P)$. 
\end{proof}

\begin{remark}
In fact, $h^*_N(P,\rho;t)$ precisely encodes 
$\{ h^*(\Pyr^r(P),\rho_N;t) : 0 \le r < N \}$, where $\Pyr^r(P) = \Pyr(\Pyr^{r - 1}(P))$ for $r > 0$ is the $r$th iterated pyramid of $P$. 

Indeed, by Example~\ref{e:pyramidN} and Lemma~\ref{l:intandceil}, 
\begin{equation}\label{e:iterpyramid}
h^*(\Pyr^r(P),\rho_N;t) = \Psi_{\Int}(h^*_N(\Pyr^r(P),\rho;t)) = \Psi_{\Int}\left( \left(\frac{1 - t}{1 - t^{\frac{1}{N}}}\right)^r h^*_N(P,\rho;t) \right).
\end{equation}
Conversely, write $h^*(\Pyr^r(P),\rho_N;t) = \sum_{m \in \Z} \alpha_{m,r} t^m$ for some $\alpha_{m,r} \in R(G)$, and $h^*_N(P,\rho;t) = \sum_{j \in \Q} h^*_j t^j$ for some $h^*_j \in R(G)$. By expanding \eqref{e:iterpyramid},  one may verify that for a nonnegative integer $m$ and $0 \le k < N$:
$$
h^*_{m - \frac{k}{N}} = \sum_{i = 0}^k (-1)^{i}  \binom{k}{i} \alpha_{m,k - i} + \sum_{j < m - \frac{k}{N}} c_j h^*_{j},
$$
for some (known) $c_j \in R(G)$. Hence we may recover 
$h^*_N(P,\rho;t)$ from 
$\{ h^*(\Pyr^r(P),\rho_N;t) : 0 \le r < N \}$.

\end{remark}



\section{Invariant triangulations}\label{s:triangulations}

In this section, we consider the question of when there exists an invariant lattice triangulation of a lattice polytope. The relevant background material is introduced in Section~\ref{ss:subdivisonspolytopes}. In Section~\ref{ss:constructions}, we give criterion to guarantee the existence of invariant triangulations and prove Theorem~\ref{t:triangulationOrderGroup}. Using these results, in  
 Section~\ref{ss:translative}, we deduce criterion  for producing invariant lattice triangulation involving  the concept of translative actions.
In Section~\ref{ss:dim2}, we consider the case of two dimensional polytopes. 
As in the previous section,
let $M \cong \Z^d$ be a rank $d$ lattice.
Let $G$ be a finite group  
and let $\rho: G \to \Aff(M_\R)$ be an affine representation of $G$. 
Let $P \subset M_\R$ be a $G$-invariant $d$-dimensional polytope. 
We are mainly interested in the case when $\rho$ preserves $M$ and $P$ is a lattice polytope, but will not assume this unless otherwise stated. 
Let $\tM = M \oplus \Z$ and let $\HT_\R: \tM_\R \to \R$ be projection onto the last coordinate. 

\subsection{Polyhedral subdivisions of polytopes}\label{ss:subdivisonspolytopes}

We first recall some background on polyhedral subdivisions. We refer the reader to 
\cite{DRSTriangulations10} and \cite[Chapter~7]{GKZ94} for details.
A \emph{polytopal complex} $\cS  = \{ F_i \}_{i \in I}$ in $M_\R$ is a collection of 
polytopes $F_i$ in $M_\R$ indexed by a finite subset $I$ such that: 
\begin{enumerate}
		
	\item  Every (possibly empty) face of $F_i$ lies in $\cS$ for all $i \in I$.
	
	\item The intersection $F_i \cap F_j$ is a (possibly empty) face of both $F_i$ and $F_j$ for all $i,j \in I$.
	
\end{enumerate}

We say that $\cS$ is $G$-invariant if $g \cdot F \in \cS$ for all $g$ in $G$ and $F \in \cS$.
Elements of $\cS$ are called \emph{faces} of $\cS$, and the $0$-dimensional and maximal faces of $\cS$ are called vertices and facets respectively.  
We say that $\cS$ is a \emph{lattice} or \emph{rational} polytopal complex if the vertices of $\cS$ lie in $M$ or $M_\Q$ respectively. 
The \emph{support} $\Supp(\cS)$ of $\cS$ is the union of the faces $\{ F_i : i \in I \}$ in $M_\R$. 
We say that $\cS$ is a \emph{polyhedral subdivision} of its support. In particular, a polyhedral subdivision of $P$ is a polytopal complex in $M_\R$ with support $P$. 
A \emph{triangulation} of $P$ is a polyhedral subdivision of $P$ such that each 
facet is a simplex i.e. has precisely $d + 1$ vertices. 

%
%
%
%
%
%

Let $A \subset P$ be a finite set 
containing 
the vertices of $P$ 
and consider a 
function $\omega: A \to \R$.  Let $\UH(\omega)$ be the convex hull of $\{ (u,\lambda) : u \in A, \lambda \ge \omega(u) \} \subset \tM_\R$. 
Let $\cS(\omega)$ be the polyhedral subdivision of $P$ consisting of 
the images under  projection $\tM_\R \to M_\R$ onto all but the last coordinate
of the bounded faces of $\UH(\omega)$.
A polyhedral subdivision $\cS$ of $P$ is \emph{regular} if $\cS = \cS(\omega)$ for some such finite subset $A$ and function $\omega: A \to \R$. 

\begin{remark}\label{r:invariantheight}
	If $\cS$ is $G$-invariant and regular then 
	$\cS = \cS(\omega)$ for some $\omega$  that is $G$-invariant in the sense that $A$ is $G$-invariant and $\omega(g \cdot u) = \omega(u)$ for all $g$ in $G$ and $u$ in $A$. 
	Indeed, we may always achieve this as follows: firstly, we may replace $A$ by any 
	finite set $A'$ containing $A$ by defining $\omega(u) = \min_{(u,\lambda) \in \UH(\omega)} \lambda$ for all $u \in A'$. In particular, after setting $A' = G \cdot A$, we may assume that $A$ is $G$-invariant. 
	Secondly, we may replace
	 $\omega$ with $\omega'$ where $\omega'(u) := \frac{1}{|G|}\sum_{g \in G} \omega(g \cdot u)$ for all $g$ in $G$ and $u$ in $A$ (see \cite[Corollary~2.11]{Reiner02}). Conversely, if $\omega$ is $G$-invariant, then $\cS(\omega)$ is $G$-invariant and regular. 
\end{remark}

The above construction has a natural generalization that we recall in the lemma below.
If $\cS$ and $\cS'$ are polytopal complexes, then $\cS'$  \emph{refines} $\cS$ if every face of $\cS'$ is contained in a face of $\cS$, and $\cS$ and $\cS'$ have the same support.  In that case, if $F$ is a face of $\cS$, then the restriction $\cS'|_F = \{ F' \in \cS' : F' \subset F \}$ is a polyhedral subdivision of $F$.
Let $\ver(\cS)$ denote the set of vertices of $\cS$. 



\begin{lemma}\cite[Lemma~2.3.16]{DRSTriangulations10}\label{l:regrefinement}
Let $\cS$ be a polyhedral subdivision of $P$. 
Let $A \subset M_\R$ be a finite set containing 
$\ver(\cS)$
and 
%
consider a 
function $\omega: A \to \R$. 
Then there exists a unique polyhedral subdivision $\cS'$ refining $\cS$ such that $\cS'|_F = \cS(\omega|_{A \cap F})$ for all nonempty faces $F$ of $\cS$. Moreover, if $\cS$ is regular then $\cS'$ is regular. Explicitly, if $\cS = \cS(\omega_0)$ for some $\omega_0: A \to \R$, then
$\cS' = \cS(\omega_0 + \epsilon \omega)$ for sufficiently small positive $\epsilon$. 
\end{lemma}

The refinement in Lemma~\ref{l:regrefinement} is called the \emph{regular refinement} of $\cS$ by $\omega$. For example, $\cS(\omega)$ is the regular refinement of the trivial subdivision of $P$ by $\omega$.
If $\cS$ 
and $\omega$ are $G$-invariant, then the regular refinement of $\cS$ by $\omega$ is $G$-invariant. The following definition extends a special case of \cite[Definition~4.3.11]{DRSTriangulations10} (which considers the case below when $J$ is a point).

\begin{definition}\cite[Definition~4.3.11]{DRSTriangulations10}\label{d:pulling}
Let $\cS$ be a polyhedral subdivision of $P$. Let $J \subset P$ be a finite subset. 
Let $A \subset M_\R$ be a finite set containing 
$\ver(\cS) \cup J$.
 The \emph{pulling refinement} $\cS_J$ of $\cS$ by $J$ is the regular 
 refinement of $\cS$ by the function $\omega: A \to \R$ defined by 
 \begin{equation*}
 \omega(u) = \begin{cases}
-1 &\textrm{ if } u \in J \\
0  &\textrm{ otherwise. }
\end{cases}
 \end{equation*}

\end{definition}

In Definition~\ref{d:pulling}, the pulling refinement $\cS_J$
is independent of the choice of $A$. Indeed, let $A' = \ver(\cS) \cup J$ with corresponding function $\omega': A' \to \R$ defined by $\omega'(u) = -1$ if $u \in J$ and $\omega'(u) = 0$ otherwise. Then for any face $F$ in $\cS$, we have an equality of upper convex hulls $\UH(\omega|_{A \cap F}) = \UH(\omega|_{A' \cap F})$ and
Lemma~\ref{l:regrefinement} implies that the regular refinements of $\cS$ by $\omega$ and $\omega'$ agree. 

%

\begin{remark}\label{r:simpleobs}
We will need the following observations.
The vertices of $\cS_J$ are contained in the union of the vertices of $\cS$ and $J$. In particular,
if $\cS$ is a lattice polyhedral subdivision and $J \subset P \cap M$, then 
$\cS_J$
is a lattice polyhedral subdivision.
 Also, if $\cS$ is $G$-invariant and $J$ is $G$-invariant, then 
 $\cS_J$
is $G$-invariant. By Lemma~\ref{l:regrefinement}, if $\cS$ is regular then
 $\cS_J$ is regular.
\end{remark}

A case of particular interest is when  $J = \{ u \}$ for a point $u$ in $P$. In this case, we write $\cS_u : = \cS_J$. 
The following lemma is a direct consequence of Lemma~\ref{l:regrefinement} and \cite[Lemma~4.3.10]{DRSTriangulations10}.

\begin{lemma}\label{l:pullingpoints}
Let $\cS$ be a polyhedral subdivision of $P$ and let
$F$ be a face of $\cS$. Let 
 $u$ be a point in $P$. If $u \notin F$, then $F \in \cS_u$. If $u \in F$, then 
the facets of  $\cS_u|_F$ are 
$\{ \Conv{F',u}  : u \notin F' \subset F, \dim F' = \dim F - 1 \}$.
 
%
\end{lemma}

\subsection{Constructing invariant 
	triangulations}\label{ss:constructions}

The goal of this section is to develop some methods to construct invariant triangulations and to prove Theorem~\ref{t:triangulationOrderGroup}. 
We continue with the notation above. 

Our main tool is the following lemma.  Below we consider the empty set to be a simplex of dimension $-1$. 

\begin{lemma}\label{l:constructtriangulation}
	Let $G$ be a finite group.
	Let $M$ be a lattice of rank $d$ and let $\rho: G \to  \Aff(M_\R)$ be an affine representation.  
	Let $P \subset M_\R$ be a $G$-invariant $d$-dimensional polytope.
	Let $u_1,\ldots,u_r$ be a collection of points in $P$ and let $\cS$ be a $G$-invariant regular polyhedral subdivision of $P$. Consider the subcomplex 
	$\mathcal{T}(i) := \{ F  \in \cS : F \cap (\cup_{1 \le j \le i} G \cdot u_j) = \emptyset \}$ of $\cS$ for all $0 \le i \le r$. 
	Assume that the following properties are satisfied:
	\begin{enumerate}
		\item Every face in $\mathcal{T}(r)$ is a simplex.
		
		\item For all $1 \le i \le r$, $u_i \in \Supp(\mathcal{T}(i - 1))$  and $|F \cap (G \cdot u_i)| \le 1$ for all 
		$F \in \mathcal{T}(i - 1)$.

		\end{enumerate}
		
		Then the polyhedral subdivision $\cS_{G \cdot u_1, \ldots, G \cdot u_r}$ 
		is a $G$-invariant regular triangulation. 
	\end{lemma}
\begin{proof}
	Let $\cS(0) = \cS$ and let $\cS(i) = \cS(i - 1)_{G \cdot u_i}$ for $1 \le i \le r$. By Remark~\ref{r:simpleobs}, $\cS(r)$ is a $G$-invariant regular polyhedral subdivision of $P$. It remains to show that 
	$\cS(r)$ is a triangulation.
	%
	%
	
	Recall that given polytopes $Q,Q',Q''$ in $M_\R$ with $Q = \Conv{Q',Q''}$ we say that $Q$ is the \emph{join} of $Q'$ and $Q''$ and write $Q = Q' * Q''$ if $\dim Q = \dim Q' + \dim Q'' + 1$. 
	We claim that every facet $F$ of
	$\cS(i)$ has the form $F = F' * F''$
	where  $F' \in \mathcal{T}(i)$
	and $F''$ is a simplex 
	with vertices contained in $\cup_{1 \le j \le i} G \cdot u_j$.
	Since all nonempty faces of $\mathcal{T}(r)$ are simplices by assumption, it follows that $\cS(r)$ is a triangulation. 
	
	It remains to prove the claim. The claim holds when $i = 0$ since $\cS(0) = \mathcal{T}(0) = \cS$ and we may write $F = F' * F''$ with $F' = F$ and $F'' = \emptyset$ for any facet $F$ of $\cS(0)$.  
	Assume that  $1 \le i \le r$. By assumption, $u_i \in  \Supp(\mathcal{T}(i - 1))$ and since $\mathcal{T}(i - 1)$ is $G$-invariant, the orbit $G \cdot u_i$ lies in $\Supp(\mathcal{T}(i - 1))$. Let $F$ be a facet of $\cS(i - 1)$. 
	It is enough to establish the claim for all facets of 
	$\cS(i)|_F = \cS(i - 1)_{G \cdot u_i}|_F$. 
	By induction, we can write 
	$F = F' * F''$
	where  $F' \in \mathcal{T}(i-1)$
	and $F''$ is a simplex 
	with vertices contained in $\cup_{1 \le j \le i-1} G \cdot u_j$.
	Then $F'$ is the intersection of $F$ with $\Supp(\mathcal{T}(i - 1))$. 
	From our assumptions, $|F \cap (G \cdot u_i)| = |F' \cap (G \cdot u_i)| \le 1$. 
	If $F \cap (G \cdot u_i) = \emptyset$, then $F' \in \mathcal{T}(i)$ and
	$F \in \cS(i)$ by Lemma~\ref{l:pullingpoints}, and the claim follows.
	
	It remains to consider the case when $F \cap (G \cdot u_i) = \{ w \}$ is a point. Let $H$ be a facet of $F$ that doesn't contain $w$. Then $H \cap (G \cdot u_i) = \emptyset$. Since $F = F' * F''$ and $w \in F'$, we have $H = H' * F''$ for some facet $H'$ of $F'$. Since $H' \subset F' \in \mathcal{T}(i - 1)$, we have 
	$H' \in \mathcal{T}(i)$.     
	By Lemma~\ref{l:pullingpoints}, the polytopes of the form $\Conv{H,w}$ are precisely the
	facets of 
	$\cS(i)|_F$. Then $\Conv{H,w} = H' * \Conv{F'',w}$, where $H' \in \mathcal{T}(i)$ and $\Conv{F'',w}$ is a simplex 
	with vertices contained in $\cup_{1 \le j \le i} G \cdot u_j$. This completes the proof of the claim.
	\end{proof}


\begin{example}\label{e:permutahedronv2}
	We continue with the setup of Example~\ref{e:permutahedronv1}. That is, 
	 $\psi: \Z^{d + 1} \to \Z$, $\psi(u_1,\ldots,u_{d + 1}) = \sum_{i = 1}^{d + 1} u_i$, and $M' = \psi^{-1}(\binom{d + 2}{2} )$ is an affine lattice in the real affine space $W = \psi_\R^{-1}(\binom{d + 2}{2} )$.
	  The symmetric group $G = \Sym_{d + 1}$ acts on  the $(d + 1)$-permutahedron $P \subset W$.

	 
	 The nonempty faces of $P$ are in order-preserving bijection with the set of ordered partitions of $\{ 1,\ldots, d+ 1\}$ (see, for example, \cite[Example~0.10]{ZieglerLectures}). Explicitly, consider an ordered partition 
	 $I = (I_1,\ldots,I_s)$ of $\{ 1,\ldots, d+ 1\}$ and the corresponding face 
	 $F_I$ of $P$. Then  $\dim F_I = d + 1 - s$ and the  vertices of $F_I$ are  precisely all $v = (v_1,\ldots,v_{d + 1})$ such that we have an equality of unordered sets $\{ v_i : i \in I_1 \cup \cdots \cup I_j \} = \{ 1,2,\ldots, \sum_{k \le j} |I_k| \}$ for all $1 \le j \le s$. Then the average of the vertices of $F_I$ is the element
	 $u_I = (u_{I,1},\ldots,u_{I,d+1})$ defined by 
	 \[
	 u_{I,i} = \frac{|I_j| + 1}{2} + \sum_{k < j} |I_k|  \textrm{ if } i \in I_j.
	 \]
	 In particular, $u_{I} \in \frac{1}{2}M'$ and $u_I$ lies in the relative interior of $F_I$ for all $I$. Note that $G$ naturally acts on the set of all ordered partitions of $\{ 1,\ldots, d+ 1\}$ and $g \cdot u_I = u_{g \cdot I}$ for any $g$ in $G$. Let $I(1),\ldots,I(t)$ be representatives of the $G$-orbits of 
     the set of all ordered partitions of $\{ 1,\ldots, d+ 1\}$, ordered such that $\dim F_{I(k)} \ge \dim F_{I(k')}$ for $k \le k'$. Let $\cS$ be the trivial subdivision of $P$. 
     Then Lemma~\ref{l:constructtriangulation}, implies that  the barycentric subdivision 
      $\cS_{G \cdot u_{I(1)}, \ldots, G \cdot u_{I(t)}}$ 
     is a $G$-invariant regular triangulation with vertices in $\frac{1}{2}M'$. 

	Note that $P$ is a simplex when $d = 1$. When $d = 2$, $(2,2,2)$ is a $G$-fixed interior lattice point of $P$ and hence the pulling refinement of the trivial subdivision of $P$ by $(2,2,2)$ is a $G$-invariant lattice triangulation (see also Proposition~\ref{p:dim2mainthm}). 
	Using Example~\ref{e:permutahedronv1} and Corollary~\ref{c:maincorollaryshort}, we deduce that the following are equivalent: 
	\begin{enumerate}
		\item There exists a $G$-invariant triangulation of $P$ with vertices in $\frac{1}{N}M'$.
		\item $h^*_N(P,\rho;t)$ is effective.
		\item $h^*_N(P,\rho;t)$ is a polynomial.
		\item $N$ is even or $d \le 2$. 
	\end{enumerate}
	In particular, let $P'$ be the convex hull of the $G$-orbit $(-d,-d + 2,\ldots, d - 2, d)$ contained in $\R^{d + 1}$. Then $P'$ may be identified up to a lattice translation with $2P$, and we deduce that $P'$ admits a $G$-invariant lattice triangulation. 	It may be an interesting problem to determine a combinatorial formula for the (nonnegative) multiplicities of all irreducible representations of $\Sym_{d + 1}$ in the coefficients of the equivariant $h^*$-polynomial of $P'$. More generally, one could ask the same question for $h^*_N(P,\rho;t)$ when $N = 2$ (cf. Lemma~\ref{l:intandceil}). 
\end{example}

We are now ready to prove Theorem~\ref{t:triangulationOrderGroup}. 
In order to help the reader understand the proof,
we provide a concrete example below (see Example~\ref{e:Z/3Zprism4}). Recall that Theorem~\ref{t:triangulationOrderGroup} states that if $P$ is a $G$-invariant lattice polytope and $N$ is a positive integer  
divisible by $|G|$, then there exists a $G$-invariant regular triangulation of $P$ with vertices in $\frac{1}{N} M$.

\begin{proof}[Proof of Theorem~\ref{t:triangulationOrderGroup}]
	Let $\cS$ be the trivial subdivision of $P$.
Our goal is to inductively construct a sequence $u_1,\ldots,u_r$ in $P \cap \frac{1}{N}M$ for some positive integer $r$ and apply Lemma~\ref{l:constructtriangulation}. As in Lemma~\ref{l:constructtriangulation}, we will consider the subcomplex 
$\mathcal{T}(i) := \{ F  \subset P : F \cap (\cup_{1 \le j \le i} G \cdot u_j) = \emptyset \}$ of $P$ for all $0 \le i \le r$.
	
Let $u_1$ to be a $G$-fixed point in $P \cap \frac{1}{N}M$. This exists since for any $u \in P \cap M$, the average $\frac{1}{|G|} \sum_{g \in G} g \cdot u \in P \cap \frac{1}{N}M$ is $G$-fixed. Suppose that we have constructed $u_1,\ldots,u_{i - 1} \in P \cap \frac{1}{N}M$ for some positive integer $i > 1$. If $\T(i - 1)$ is empty, then set $r = i - 1$. Otherwise,  let $u$ be a lattice point in $\Supp(\T(i - 1))$. Let $m = \max_{Q \in \T(i - 1)} | Q \cap (G \cdot u)| \in \Z_{> 0}$. Fix $Q$ in $\T(i - 1)$ such that  
$| Q \cap (G \cdot u)| = m$. After possibly replacing $Q$ by a face of $Q$
we may also assume that $Q$ is the smallest face of $P$ containing $ Q \cap (G \cdot u)$. Let $G_Q$ denote the stabilizer of $Q$ and let $u_{i}$ be a $G_Q$-fixed point in the relative interior $Q^\circ$ of $Q$.

 We claim that $|F \cap (G \cdot u_i)| \le 1$ for all  $F$ in $\T(i - 1)$.
To see this, suppose there exists a face $F$ in $\T(i - 1)$ such that $|F \cap (G \cdot u_i)| > 1$ for some positive integer $i$.
After possibly replacing $F$ by $g \cdot F$ for some $g$ in $G$, we may assume that $u_i \in F$.  By assumption there exists $g$ in $G$ such that $u_i \neq g \cdot u_i \in F$. Since $u_i$ and $g \cdot u_i$ lie in $Q^\circ$ and $(g \cdot Q)^\circ$ respectively, we have $Q, g \cdot Q \subset F$. On the other hand, since $| F \cap (G \cdot u)| \le m$ and $| Q \cap (G \cdot u)| = | g \cdot Q \cap (G \cdot u)| = m$, we have $Q \cap (G \cdot u) = g \cdot Q \cap (G \cdot u) = F \cap (G \cdot u)$.
Since $Q$ is the smallest face of $P$ containing $Q \cap (G \cdot u)$, we deduce that 
$Q \subset g \cdot Q$ and hence $Q = g \cdot Q$ and $g \in G_Q$. Since $u_i$ is $G_Q$-fixed by assumption, we have $g \cdot u_i = u_i$, a contradiction.

Observe that  $\T(i) \neq \T(i - 1)$ and hence the above construction terminates. We conclude that we have  produced a finite sequence $u_1,\ldots,u_r$  with $u_i \in \Supp(T(i - 1))$ for $1 \le i \le r$ and  $\T(r) = \emptyset$. Moreover, the fact that $u_1$ is $G$-fixed and the claim above imply that the hypothesis of Lemma~\ref{l:constructtriangulation} hold.
Then Lemma~\ref{l:constructtriangulation}  implies that the polyhedral subdivision $\cS_{G \cdot u_1, \ldots, G \cdot u_r}$ 
is a $G$-invariant regular triangulation. 


It remains to show that we may choose $u_i \in \frac{1}{N}M$ for $1 < i \le r$. We continue with the notation above and recall that $u_i$ is any choice of $G_Q$-fixed point in $Q^\circ$.
Consider the lattice polytope $S = \Conv{Q \cap (G \cdot u)}$. For any $g$ in $G_Q$, $g \cdot (Q \cap (G \cdot u)) = Q \cap (G \cdot u)$, and hence $g \cdot S = S$. Let $S'$ be the $G_Q$-fixed locus of $S$. 
Since $(S')^\circ \subset Q^\circ$, it is enough to show that $(S')^\circ \cap \frac{1}{|G|}M \neq \emptyset$. 

Let $G_Q \cdot y_1, \ldots, G_Q \cdot y_\ell$ be the $G_Q$-orbits of $Q \cap (G \cdot u)$ for some positive integer $\ell$. 
Then the vertices of the rational polytope $S'$ are contained in $\{ \frac{1}{|G_Q|} \sum_{g \in G_Q} g \cdot y_i : 1 \le i \le \ell \}$ by \cite[Lemma~5.4]{StapledonEquivariant}. 
In particular, $S'$ is a lattice polytope with respect to the lattice $\frac{1}{|G_Q|} M$. By Remark~\ref{r:interiorpointsexist}, $(S')^\circ \cap \frac{1}{s|G_Q|} M \neq \emptyset$ for all $s > \dim S'$. Since $S'$ has at most $\ell$ vertices, $\dim S' < \ell$ and hence to finish the proof it is enough to establish that 
$\ell \le \frac{|G|}{|G_Q|}$. 
%
Consider an element $y \in Q \cap (G \cdot u)$ such that $|G_Q \cdot y|$ is minimal. 
Write $\Stab(G,y)$ and $\Stab(G_Q,y)$ for the stablizers of $y$ in $G$ and $G_Q$ respectively. 
Since $G \cdot u = G \cdot y$, we compute
\[
\ell \le \frac{|Q \cap (G \cdot y)|}{|G_Q \cdot y|} \le \frac{|G \cdot y|}{|G_Q \cdot y|} = \frac{|G|}{|G_Q|} \frac{|\Stab(G_Q,y)|}{|\Stab(G,y)|} \le  \frac{|G|}{|G_Q|}.
\]
%

\end{proof}

\begin{example}\label{e:Z/3Zprism4}
	We demonstrate the proof of Theorem~\ref{t:triangulationOrderGroup} in a special case of Example~\ref{e:Symprism1} and Example~\ref{e:Symprism2}. 
   Let $d = 3$ and let $\psi: \Z^{4} \to \Z$, $\psi(u_1,u_2,u_3,u_4) = u_1 + u_2 + u_3$. Let $M' = \psi^{-1}(1)$ be the affine lattice in the real affine space $W = \psi_\R^{-1}(1)$. Let $H = \Z/3\Z$ act by cycling the first three coordinates in $M'$, and let
   $P = \Conv{e_1,e_2,e_3} \times [0,e_4] \subset W$.
   

%
	We now consider the notation of the proof of Theorem~\ref{t:triangulationOrderGroup}, and show how to construct a sequence $u_1,\ldots,u_r \in P$.
	%
	Let $u_1 = (\frac{1}{3},\frac{1}{3},\frac{1}{3},1)$. Then $\T(1)$ is all faces of $P$ except for $ \Conv{e_1,e_2,e_3} \times \{ 1\}$ and $P$ itself. To construct $u_2$, suppose that we choose $u = (1,0,0,0) \in \Supp(\T(1))$. Then $m = 3$ and we may choose $Q = \Conv{e_1,e_2,e_3} \times \{ 0 \}$. Then $G_Q = G$ and we may set $u_2 =  (\frac{1}{3},\frac{1}{3},\frac{1}{3},0)$.
	Now $\T(2)$ consists of all faces of $\T(1)$ except for $ \Conv{e_1,e_2,e_3} \times \{ 0\}$.  To construct $u_3$, suppose that we choose $u = (1,0,0,0) \in \Supp(\T(2))$. Then $m = 2$ and we may choose  
	$Q = \Conv{e_1,e_2} \times \{ 0 \}$. 
	Then $G_Q = \{1\}$ and we may set $u_3 = (\frac{2}{3},\frac{1}{3},0,0)$.
	By the proof of  Theorem~\ref{t:triangulationOrderGroup},
	the polyhedral subdivision obtained from $P$ by applying successive pulling refinements to the orbits $G \cdot u_1 = \{u_1\}, G \cdot u_2 = \{u_2\}, G \cdot u_3$ is a $G$-invariant regular triangulation with vertices in $\frac{1}{|G|}M$. 
	
\end{example}

\begin{remark}\label{r:triangulationOrdercor}
	It follows from Theorem~\ref{t:triangulationOrderGroup} and Corollary~\ref{c:maincorollaryshort}  that $h^*_N(P,\rho;t)$ is effective (and a polynomial) if
	$|G|$ divides $N$.
	 Since $h^*_N(P,\rho;t)$ is a polynomial if and only if $h^*_N(P,\rho;t)(g)$ is a polynomial for all $g$ in $G$, it follows from 
	 that $h^*_N(P,\rho;t)$ is a polynomial if $\exp(G) | N$ (cf. \cite[Remark~7.2]{StapledonEquivariant}).  
\end{remark}

\begin{question}\label{q:exponent}
	Does the following strengthening of Theorem~\ref{t:triangulationOrderGroup} hold: 
	 if $\exp(G) | N$ then there exists a $G$-invariant regular triangulation of $P$ with vertices in $\frac{1}{N} M$? 	Does the following strengthening of Remark~\ref{r:triangulationOrdercor} hold: 
	 if $\exp(G) | N$ then $h^*_N(P,\rho;t)$ is effective? 
	 Note that the first statement implies the second statement by Corollary~\ref{c:maincorollaryshort}.
	 By Remark~\ref{r:triangulationOrdercor}, the second statement would follow if the effectiveness conjecture is true (see Conjecture~\ref{c:origmod}).
	 Can these statements be established for all choices of subgroup $H$ in Example~\ref{e:Symprism4} below?
\end{question}


\begin{example}\label{e:Symprism4}
	We continue with Example~\ref{e:Symprism1} and Example~\ref{e:Symprism2}.
	Recall that $d > 1$, $\psi: \Z^{d + 1} \to \Z$, $\psi(u_1,\ldots,u_d,u_{d + 1}) = \sum_{i = 1}^d u_i$, and $M' = \psi^{-1}(1)$ is the affine lattice in the real affine space $W = \psi_\R^{-1}(1)$. Recall that $\rho: G \to \Aff(M')$ is the affine representation with $G = \Sym_d$ acting by permuting coordinates in the first $d$ coordinates, $Q = \Conv{e_1,\ldots,e_d} \subset W$ and  $P = Q \times [0,e_{d+ 1}] \subset W$.
	Let $H$ be any subgroup of $G$ and let $N$ be a positive integer. By Corollary~\ref{c:maincorollaryshort} and Example~\ref{e:Symprism2}, 
if $\exp(H) \nmid N$ then there does not exist an $H$-invariant triangulation of $P$ with vertices in $\frac{1}{N}M'$. On the other hand, Theorem~\ref{t:triangulationOrderGroup} implies that there exists a regular $H$-invariant triangulation of $P$ with vertices in $\frac{1}{N}M'$ if $|H|$ divides $N$. 
	
	In Question~\ref{q:exponent} above, we ask if there exists a  regular $H$-invariant triangulation of $P$ with vertices in $\frac{1}{N}M'$ if $\exp(H) | N$. 
	We claim that this does exist when $H = G$. In this case, $\exp(G) = \lcm(m: 1 \le m \le d)$. To prove the claim it is enough to consider the case when $N = \exp(G)$. Starting with the trivial subdivision, apply successive pulling refinements by the $G$-orbits 
$$
G \cdot (\frac{1}{d},\ldots,\frac{1}{d},0) 
, G \cdot (\frac{1}{d-1},\ldots,\frac{1}{d-1},0,0),\ldots,G \cdot (\frac{1}{2},\frac{1}{2},0,\ldots,0,0).
$$
By Lemma~\ref{l:constructtriangulation}, this produces a $G$-invariant regular triangulation of $P$ with vertices in $\frac{1}{\exp(G)}M'$.

%
	
	Finally, we mention that the (necessarily regular) lattice triangulations of $P$ are in bijection with vertices of the secondary polytope of $P \cap M'$ \cite[Definition~5.1.6]{DRSTriangulations10}, which is
%
	combinatorially equivalent to the $d$-permutahedron \cite[Example~7.3C]{GKZ94}. 
	 Here the $d$-permutahedron is the convex hull of the orbit $G \cdot (1,2,\ldots,d) \subset \R^d$, where $G$ acts on $\R^d$ by permuting coordinates (cf. Example~\ref{e:permutahedronv1}). Moreover, the combinatorial equivalence preserves the action of $G$. 
	In particular, $G$ acts with trivial stabilizers on the $d!$ vertices  of the $d$-permutahedron, and hence we have an alternative proof that there are no 
	$H$-invariant lattice triangulations of $P$ provided $H$ is nontrivial.
	 
	 We next develop a criterion to guarantee the existence of a $G$-invariant lattice triangulation. Applications of this proposition will be explored in Section~\ref{ss:translative}.
	 
\end{example}


\begin{proposition}\label{p:existtriangulation}
				Let $G$ be a finite group.
	Let $M$ be a lattice of rank $d$ and let $\rho: G \to  \Aff(M)$ be an affine representation.  
	Let $P \subset M_\R$ be a $G$-invariant $d$-dimensional lattice polytope.
	Let $\cS$ be a $G$-invariant regular lattice polyhedral subdivision of $P$. Assume that for every face $F$ of $\cS$ that is not a simplex there exists $u \in F \cap M$ such that $G \cdot u$ intersects every facet of $\cS$ at most once. Then there exists a
	$G$-invariant regular lattice triangulation $\cS'$ of $P$ that refines $\cS$. 
	Moreover, if we can choose each $u$ above to be a vertex of $\cS$, 
	then we may choose $\cS'$ 
	such that $\cS$ and $\cS'$ have the same vertices and every simplex in $\cS$ is a simplex in $\cS'$. 
\end{proposition}
\begin{proof}
		We want to apply Lemma~\ref{l:constructtriangulation}. We will construct a sequence $u_1,\ldots,u_r \in P \cap M$ inductively for some nonnegative integer $r$.
	Recall from Lemma~\ref{l:constructtriangulation} that we consider the subcomplex 
	$\mathcal{T}(i) := \{ F  \in \cS : F \cap (\cup_{1 \le j \le i} G \cdot u_j) = \emptyset \}$ of $\cS$ for all $0 \le i \le r$. In particular, $\T(0) = \cS$.
	Suppose that we have constructed $u_1,\ldots,u_{i - 1} \in P \cap M$ for some positive integer $i$. If all faces of  $\mathcal{T}(i - 1)$ are simplices then we set $r = i - 1$.  If not, then there exists a face $F$ in  $\mathcal{T}(i-1)$ that is not a simplex. By assumption, there exists $u \in F \cap M$ such that $G \cdot u$ intersects every facet of $\cS$ at most once. We set 
	$u_i = u$. Since $\T(i) \neq \T(i - 1)$, this process will terminate. 
	Then  $u_1,\ldots,u_r \in P \cap M$ satisfy the hypotheses of Lemma~\ref{l:constructtriangulation}. 
	Lemma~\ref{l:constructtriangulation} then
	implies that
	$\cS' = \cS_{G \cdot u_1, \ldots, G \cdot u_r}$  
	is a $G$-invariant regular triangulation. By Remark~\ref{r:simpleobs}, $\cS'$ is a lattice triangulation.
	By Remark~\ref{r:simpleobs}, the vertices of $\cS'$ are contained in the vertices of $\cS$ and the union of the orbits $G \cdot u_1, \ldots, G \cdot u_r$. 
	Assume that each $u_i$ is a vertex of $\cS$.
	Then the elements of each orbit $G \cdot u_i$ are vertices of $\cS$, and $\cS$ and $\cS'$ have the same vertices.
	Finally, since applying a pulling refinement to a set of vertices of a simplex induces the trivial subdivision, every simplex in $\cS$ is a simplex in $\cS'$ by Lemma~\ref{l:regrefinement}.  
\end{proof}

\subsection{Translative actions}\label{ss:translative}

We now describe some applications of Proposition~\ref{p:existtriangulation} for constructing triangulations using the concept of translative actions. Although the translative property is quite strong, we will see that translative actions on subcomplexes of a polytope can be used to construct (not necessarily translative) invariant triangulations.  We continue with the notation above.


A \emph{translative} action of a group on a poset was defined by 
 D'al\`{i} and Delucchi in \cite{DDStanleyReisner} after being introduced in  a special case in \cite{DRSemimatroids}. These actions were further studied in \cite{BDSupersolvable} and \cite{DDEquivariantEhrhart}. For our purposes, we will be interested in the following special case.


%



\begin{definition}\label{d:translative}\cite[Definition~2.8]{DDStanleyReisner}
					Let $G$ be a finite group.
	Let $M$ be a lattice of rank $d$ and let $\rho: G \to  \Aff(M_\R)$ be an affine representation.  Let $\cS$ be a $G$-invariant polytopal complex in $M_\R$. The action of $G$ on $\cS$ is \emph{translative} if for every vertex $u$ of $\cS$, the orbit $G \cdot u$ intersects every facet of $\cS$ at most once.  
\end{definition}


We have the following equivalent formulations of the translative property. The lemma below follows from the discussion of \cite[Section~2.2]{DDEquivariantEhrhart}, but we provide a proof for the benefit of the reader. In particular, condition \eqref{i:trans2} below directly matches the definition of translative that appears in \cite[Definition~2.2]{DDEquivariantEhrhart} (for the poset of faces of $\cS$). 

\begin{lemma}\label{l:transequiv}
		Let $G$ be a finite group.
	Let $M$ be a lattice of rank $d$ and let $\rho: G \to  \Aff(M_\R)$ be an affine representation.  Let $\cS$ be a $G$-invariant polytopal complex in $M_\R$. Then the action of $G$ on $\cS$ is translative if and only if one of the following equivalent conditions holds:
	\begin{enumerate}
		\item\label{i:trans1} For every face $F$ of $\cS$ and $g$ in $G$, if $F$ and $g \cdot F$ lie in a common facet of $\cS$ then $g$ fixes $F$ pointwise.
		
		\item\label{i:trans2} For every face $F$ of $\cS$ and $g$ in $G$, if $F$ and $g \cdot F$ lie in a common facet of $\cS$ then $g \cdot F = F$.
		
		\item\label{i:trans3} For every vertex $u$ of $\cS$ and $g$ in $G$, if $u$ and $g \cdot u$ lie in a common facet of $\cS$ then $g \cdot u = u$.
		
		\item\label{i:trans4} For every vertex $u$ of $\cS$, $G \cdot u$ intersects every facet of $\cS$ at most once. 
		
	\end{enumerate}
\end{lemma}
\begin{proof}
	Above, condition \eqref{i:trans4} is the definition of a translative action given in Definition~\ref{d:translative}, so it suffices to show that all the conditions are equivalent. The implications \eqref{i:trans1} $\Rightarrow$  \eqref{i:trans2} $\Rightarrow$  \eqref{i:trans3} are clear. Suppose that \eqref{i:trans3} holds. If $u$ is a vertex of $\cS$ and $G \cdot u$ intersects a facet of $\cS$ at two points $w$ and $w'$, then there exists $g$ in $G$ such that $w' = g \cdot w$. Since elements of $G \cdot u$ are vertices of $\cS$, \eqref{i:trans3} implies that $w' = g \cdot w = w$, and hence \eqref{i:trans4} holds. Assume that \eqref{i:trans4} holds. Consider a  face $F$ of $\cS$ and $g$ in $G$ such that 
	 $F$ and $g \cdot F$ lie in a common facet $L$ of $\cS$. Let $u$ be a vertex of $F$. Then $u$ and $g \cdot u$ lie in $L$ and \eqref{i:trans4} implies that $g \cdot u = u$. Since $g$ acts on $M_\R$ via an affine transformation and fixes all the vertices of $F$, we conclude that $g$ fixes $F$ pointwise. 
\end{proof}

\begin{example}\label{e:translativeboundaryZmod2Z}
	Let $G = \Z/2\Z = \langle \sigma \rangle$, 
	let $M$ be a lattice of rank $d$ and let $\rho: G \to  \Aff(M)$ be an affine representation. 
	Let $P \subset M_\R$ be a $G$-invariant $d$-dimensional lattice polytope. 
	Recall that $P^\sigma$ is the fixed locus of $\sigma$. Let $K$ be the subcomplex of $P$ consisting of the faces $F$ of $P$ such that $F \cap P^\sigma = \emptyset$.
	We claim that the induced action of $G$ on $K$ is translative. 
	This follows since if
	$u$ is an vertex in $K$, then a face of $P$ containing $u$ and $\sigma \cdot u$ must also contain $\frac{u + \sigma \cdot u}{2}$  in $P^\sigma$ and hence does not lie in $K$. 
\end{example}

\begin{remark}\label{r:translative}
	Consider an action of $G$ on a polytopal complex $\cS$. Suppose that 
	$\cS'$ is a $G$-invariant polytopal complex that refines $\cS$ and has the same vertices as $\cS$. 
%
	For every vertex $u$ of $\cS$, the intersection of $G \cdot u$ with a face of $\cS'$ is contained in a face of $\cS$.  
%
	 In particular, if the action of $G$ on $\cS$ is translative, then the action of $G$ on $\cS'$ is translative.
Similarly, if $\hat{\cS}$ is a $G$-invariant subcomplex of $\cS$ and the action of $G$ on $\cS$ is translative, then the action of $G$ on $\hat{\cS}$ is translative.
\end{remark}

 
%


We next present two corollaries to Proposition~\ref{p:existtriangulation}. 
Roughly speaking, the first corollary says that $G$-invariant translative regular lattice polyhedral subdivisions of a subcomplex of $P$ can be refined to $G$-invariant translative regular lattice triangulations, while the second corollary says that $G$-invariant (not necessarily translative) regular lattice triangulations of polytopal complexes supported on appropriate disjoint subsets of $P$ can be `glued'  to obtain a $G$-invariant (not necessarily translative) regular lattice triangulation of $P$ if at least one of the two actions of $G$ on a triangulation is translative. 

Let $K$ be a subcomplex of $P$. Below we will slightly abuse notation by sometimes identifying $K$ with its support. Let $\cS$ be a polyhedral subdivision of $K$.
%
Then $\cS$ is regular if  there exists a finite set $A \subset K$ containing the vertices of  $K$ and a function $\omega : A \to \R$ such that $\cS|_F = \cS(\omega|_{A \cap F})$ for every face $F$ of $K$. 
Combined with Theorem~\ref{t:mainfull}, Corollary~\ref{c:translativeK} below (with $K = P$ in the statement)  implies that if $P$ admits a  $G$-invariant translative regular lattice polyhedral subdivision then $h^*(P,\rho;t)$ is effective. 

\begin{corollary}\label{c:translativeK}
		Let $G$ be a finite group.
	Let $M$ be a lattice of rank $d$ and let $\rho: G \to  \Aff(M)$ be an affine representation.  
	Let $P \subset M_\R$ be a $G$-invariant $d$-dimensional lattice polytope. 
	Let $K$ be a subcomplex of $P$ and let  $\cS$ be a
$G$-invariant translative regular lattice polyhedral subdivision of $K$. 
	Then there exists a $G$-invariant translative regular triangulation $\cS'$ of $K$ that refines $\cS$ 	such that $\cS$ and $\cS'$ have the same vertices and every simplex in $\cS$ is a simplex in $\cS'$. 
%
\end{corollary}
\begin{proof}
	First consider the case $K = P$. 	Let $F$ be a face of $\cS$ and let $u$ be any vertex of $F$. By Definition~\ref{d:translative}, $G \cdot u$ intersects every facet of $\cS$ at most once. Proposition~\ref{p:existtriangulation} implies that there exists a
	$G$-invariant regular triangulation $\cS'$ of $P$ that refines $\cS$ 
	such that $\cS$ and $\cS'$ have the same vertices and every simplex in $\cS$ is a simplex in $\cS'$. 
	The action of $G$ on $\cS'$ is translative by Remark~\ref{r:translative}.

	Now consider the case when $K \neq P$. We claim that this reduces to the case above.  
	As in Remark~\ref{r:invariantheight}, there exists a $G$-invariant function $\omega_K: K \cap M \to \R$ such that $\cS|_F = \cS(\omega_K|_{A \cap F})$ for every face $F$ of $P$ contained in $K$. Let $\omega$ be the function extending $\omega_K$ with value zero (or any fixed constant) at all vertices of $P$ not in $K$, and set $\hat{\cS} = \cS(\omega)$. Then $\hat{\cS}$ is a $G$-invariant regular lattice polyhedral subdivision of $P$ that restricts to $\cS$ on $K$.
	
	Let $F_1,\ldots,F_r$ be representatives of the $G$-orbits of the faces of $\hat{\cS}$ not contained in $K$. After reordering, we may assume that $\dim F_i \ge \dim F_{i + 1}$ for $1 \le i < r$. Let $x_i$ be an element of $M_\Q$  in the relative interior of $F_i$ that is fixed by the stabilizer of $F_i$ e.g. the average of the vertices of $F_i$. We apply successive pulling refinements to obtain
	the $G$-invariant regular polyhedral subdivision $\mathcal{T} := \hat{\cS}_{G \cdot x_1, \ldots, G \cdot x_r}$. Then $\mathcal{T}$ is a barycentric subdivision of  $\hat{\cS}$ except that we exclude the faces of $\hat{\cS}$ contained in $K$.
	Explicitly, if we choose any ordering of the elements of $G \cdot x_i$ and write $G \cdot x_i = \{ y_{i,1},\ldots,y_{u,s_i}\}$ for some $s_i \in \Z_{>0}$, then $\mathcal{T}$ is obtained from $\hat{\cS}$ by applying successive pulling triangulations at the points $y_{1,1},\ldots,y_{1,s_1},\ldots,y_{r,1},\ldots,y_{r,s_r}$. Moreover, by Lemma~\ref{l:pullingpoints}, every facet of $\mathcal{T}$ is the convex hull of 
	a (possibly empty) face of $\cS$ and some elements $y_{i_1,j_1},\ldots,y_{i_t,j_t}$ for some $i_1 < \ldots < i_t$. 
	In particular, if two distinct vertices of a facet of $\mathcal{T}$ lie in the same $G$-orbit, then the vertices must be contained in a face of $\cS$. 
	 Since the action of $G$ on $\cS$ is translative, we deduce that the action of $G$ on $\mathcal{T}$ is translative. 
	 Observe that the restriction of $\mathcal{T}$ to $K$ is $\cS$. 
	 Let $M' \subset M_\R$ be the $d$-dimensional $G$-invariant lattice generated by $M$ and all the (rational) vertices of  $\mathcal{T}$. By the $K = P$ case applied to $\mathcal{T}$ and the lattice $M'$, there exists a $G$-invariant translative regular triangulation
	 $\mathcal{T}'$ of $P$ that refines $\mathcal{T}$ such that $\mathcal{T}$ and $\mathcal{T}'$ have the same vertices and every simplex in $\mathcal{T}$ is a simplex in $\mathcal{T}'$. 
	 Then the restriction of $\mathcal{T}'$ to $K$ is our desired triangulation refining  $\cS$.

\end{proof}


Below we state and prove the secondary corollary. Note that the translative assumption below is necessary. As a simple example, let $G = \Z/2\Z$ act on $\R^2$ via reflection about the line $\{ (\frac{1}{2}, y): y \in \R \}$ and let $P = [0,1]^2$ (cf. Proposition~\ref{p:dim2mainthm}). Then $Q = [0,1] \times \{0\}$ and $K = [0,1] \times \{1\}$ are $G$-invariant lattice simplices but $P$ does not admit a $G$-invariant lattice triangulation. See also Example~\ref{e:Symprism4}. 

\begin{corollary}\label{c:translativeglue}
 	Let $G$ be a finite group.
 Let $M$ be a lattice of rank $d$ and let $\rho: G \to  \Aff(M)$ be an affine representation.  
 Let $P \subset M_\R$ be a $G$-invariant $d$-dimensional lattice polytope.
 Let $Q \subset P$ be a lattice polytope and let 
 $K$ be the subcomplex of $P$ consisting of all faces  $F$ of $P$ such that $F \cap Q = \emptyset$.
Let $\cS_Q$ and $\cS_K$ be $G$-invariant regular lattice triangulations of $Q$ and $K$ respectively. Assume either the action of $G$ on $\cS_Q$ or $\cS_K$ is translative. Then 	 there exists a
 $G$-invariant regular lattice triangulation $\cS$ of $P$ such that $\cS|_Q = \cS_Q$ and $\cS|_K = \cS_K$. 
 Moreover, the action of $G$ on both $\cS_Q$ and $\cS_K$ is translative if and only if we may choose $\cS$ such that the action of $G$ on $\cS$ is translative.
\end{corollary}
\begin{proof}
	 Let $\mathcal{T}$ be the pulling refinement of the trivial subdivision of $P$ by the vertices of $Q$  (see Definition~\ref{d:pulling}). That is, 
  let $B$ be the union of the vertices of $P$ and $Q$ and define $\omega: B \to \R$ by 
   \begin{equation*}
  	\omega(u) = \begin{cases}
  		-1 &\textrm{ if } u \in Q \\
  		0  &\textrm{ otherwise. }
  	\end{cases}
  \end{equation*}
  Then $\mathcal{T} = \cS(\omega)$. 
  If $F$ is a face of $\mathcal{T}$, then we claim that $F \cap Q$ and $F \cap K$ are also faces of  $\mathcal{T}$. Indeed, recall that $\tM = M \oplus \Z$ and $\pr_\R: \tM_\R \to \R$ denotes projection onto the last coordinate.   Let $\pi_\R: \tM_\R \to M_\R$ denote projection onto all but the last coordinate. Recall that $\UH(\omega)$ is the convex hull of
  $\{ (u,\lambda) : u \in B, \lambda \ge \omega(u) \} \subset \tM_\R$ and 
  $\mathcal{T}$ consists of the images of the bounded faces of $\UH(\omega)$ via $\pi_\R$. 
  Let $\hat{F}$ be the bounded face of $\UH(\omega)$ such that 
  $\pi_\R(\hat{F}) = F$. Then $\HT_R(\hat{F}) \subset [-1,0]$ and we may consider the (possibly empty) faces $\hat{F}_Q := \hat{F} \cap \HT_R^{-1}(-1)$ and $\hat{F}_K = \hat{F} \cap \HT_R^{-1}(0)$ of $\hat{F}$. The claim follows since  $F \cap Q = \pi_\R(\hat{F}_Q)$ and $F \cap K  = \pi_\R(\hat{F}_K)$.


  	Let $A_Q = Q \cap M$ and $A_K = K \cap M$. As in Remark~\ref{r:invariantheight},
  there exists $G$-invariant functions $\omega_Q: A_Q \to \R$ and $\omega_K: A_K \to \R$ such that $\cS_Q = \cS(\omega_Q)$ and $\cS_K|_F = \cS(\omega_K|_{A \cap F})$ for every face $F$ of $K$. 
  Let $A$ be the disjoint union of $A_Q$ and $A_K$. Note that  $B \subset A$. 
  Let $\hat{\omega}_Q$ and $\hat{\omega}_Q$ be the extensions of $\omega_Q$ and $\omega_K$ to $A$ defined by setting all undefined values equal to zero (or any fixed constant). Let $\cS$ be the successive regular refinement of $\mathcal{T}$ first by $\hat{\omega}_Q$ and then by $\hat{\omega}_K$. By Lemma~\ref{l:regrefinement}, $\cS$ is a $G$-invariant regular polyhedral subdivision such that $\cS|_Q = \cS_Q$ and $\cS|_K = \cS_K$. Moreover, the vertices of $\cS$ are the union of the vertices of $\cS_Q$ and $\cS_K$. 
  
  Let $u_Q$ and $u_K$ be vertices of $\cS_Q$ and $\cS_K$ respectively. Let $F$ be a face of $\cS$. Let $L$ be a face of $\mathcal{T}$ containing $F$, so that we may view $F$ as a face of the polyhedral subdivision $\cS|_L$. Then $F \cap Q = F \cap (L \cap Q)$ is the intersection of $F$ with a face of $L$, and hence is a face of $F$. We deduce that $F \cap Q$ is a face of $\cS_Q$. Similarly, 
  $F \cap K = F \cap (L \cap K)$ is a face of $\cS_K$. In particular, $(G \cdot u_Q) \cap F$ is contained in a face of $\cS_Q$, and  $(G \cdot u_K) \cap F$ is contained in a face of $\cS_K$. Hence, if the action of $G$ on either  $\cS_Q$ or $\cS_K$ is translative, then $|(G \cdot u_Q) \cap F| \le 1$ or $|(G \cdot u_K) \cap F| \le 1$ respectively. 
  If particular, if
   the action of $G$ on both $\cS_Q$ and $\cS_K$ is translative, then the action of $G$ on $\cS$ is translative.

  Let $H$ be a face of $\cS$ that is not a simplex. Since $H \cap Q$ and $H \cap K$ are faces of $\cS_Q$ and $\cS_K$ respectively, they are the convex hulls of the vertices of $H$ that lie in $Q$ and $K$ respectively. 
   Since $\cS_Q$ and $\cS_K$ are both triangulations and $H$ is not a simplex, we deduce that $H$ must contain a vertex $u_Q$ in $\cS_Q$ and a vertex $u_K$ in $\cS_K$. 
%
   Since either the action of $G$ on  $\cS_Q$ or $\cS_K$ is translative, we deduce that either $|(G \cdot u_Q) \cap F| \le 1$ for all faces $F$ of $\cS$, or $|(G \cdot u_Q) \cap F| \le 1$ for all faces $F$ of $\cS$. Hence we can apply Proposition~\ref{p:existtriangulation} to deduce that there exists a $G$-invariant regular lattice triangulation $\cS'$ of $\cS$
   such that $\cS$ and $\cS'$ have the same vertices and every simplex in $\cS$ is a simplex in $\cS'$.  
   The latter property implies that $\cS'|_Q = \cS|_Q = \cS_Q$ and $\cS'|_K = \cS|_K = \cS_K$. 
  We conclude that $\cS'$ is our desired triangulation.  
  
  Finally,  if the action of $G$ on both $\cS_Q$ and $\cS_K$ is translative, then we have already shown that the action of $G$ on $\cS$ is translative.    By Remark~\ref{r:translative}, we deduce that the action of $G$ on $\cS'$ is translative. Conversely, if  the action of $G$ on $\cS'$ is translative, then the 
  action of $G$ on both $\cS'|_Q  = \cS_Q$ and $\cS'|_K = \cS_K$ is translative by Remark~\ref{r:translative}. 
  
\end{proof}

We now present some examples.

\begin{example}
	Let $G$ be a finite group.
	Let $M$ be a lattice of rank $d$ and let $\rho: G \to  \Aff(M)$ be an affine representation.  
	Let $P \subset M_\R$ be a $G$-invariant $d$-dimensional lattice polytope.
	Then $P \times [0,1] \subset M_\R \oplus \R$ is a $G$-invariant lattice polytope, where $G$ acts trivially on the last coordinate of $M_\R \oplus \R$. Let $Q = P \times \{ 0 \}$ and $K =  P \times \{ 1 \}$.
	Then Remark~\ref{r:translative} and Corollary~\ref{c:translativeglue} imply that the property of admitting a $G$-invariant translative regular lattice triangulation holds for $P$ if and only it holds for $P \times [0,1]$ (cf. Example~\ref{e:Symprism4}). 
\end{example}

\begin{example}
	Let $G = \Z/2\Z = \langle \sigma \rangle$, 
	let $M$ be a lattice of rank $d$, and let $\rho: G \to  \Aff(M)$ be an affine representation. 	Let $P \subset M_\R$ be a $G$-invariant $d$-dimensional lattice polytope. Assume that the fixed locus $P^\sigma$ is a lattice polytope. 
	Let $K$ be the subcomplex of $P$ consisting of the faces $F$ of $P$ such that $F \cap P^\sigma = \emptyset$.
	By Example~\ref{e:translativeboundaryZmod2Z}, the induced action of $G$ on $K$ is translative. By Corollary~\ref{c:translativeK} and Corollary~\ref{c:translativeglue},  there exists a $G$-invariant regular lattice triangulation of $P$.
In particular, it follows from Theorem~\ref{t:mainfull} that
$h^*(P,\rho;t)$ is effective.
\end{example}

%
%
%



\begin{example}\label{e:centrallysymmetrictriangulation}
						Let $G = \Z/2\Z$, 
	let $M$ be a lattice of rank $d$, and let $\rho: G \to  \Aff(M)$ be an affine representation such that $M_\R$ has a unique $G$-fixed point.
	Let $P \subset M_\R$ be a $G$-invariant $d$-dimensional lattice polytope. That is, $P$ is a centrally symmetric polytope. 		 We claim there exists a $G$-invariant regular lattice triangulation of $P$. Combined with Theorem~\ref{t:mainfull}, this generalizes Example~\ref{e:primenofixedpts} in this case.
	
	To prove the claim, let $c$ be the unique $G$-fixed point in $M_\R$ and let $\sigma$ be the generator of $G$. Fix 
	any lattice point $u$ in $P$. The convex hull $Q$ of $u$ and $\sigma \cdot u$ is a $G$-invariant lattice simplex. Let 
	$K$ be the subcomplex of $P$ consisting of all faces  $F$ of $P$ such that $F \cap Q = \emptyset$.
	Since $c = \frac{u + \sigma \cdot u}{2}$ lies in $Q$,
  the action of $G$ on $K$ is translative by Example~\ref{e:translativeboundaryZmod2Z}. The result now follows from Corollary~\ref{c:translativeK} and Corollary~\ref{c:translativeglue}. 
	
%
\end{example}

The question below asks whether Example~\ref{e:centrallysymmetrictriangulation} generalizes to all primes $p$. Combined with Theorem~\ref{t:mainfull}, a positive answer to the question would give an alternative proof of the effectiveness assertion in  Example~\ref{e:primenofixedpts}.

\begin{question}\label{q:primetriangulations}
	Let $p$ be a prime. 					Let $G = \Z/p\Z$, 
	let $M$ be a lattice of rank $d$ and let $\rho: G \to  \Aff(M)$ be an affine representation such that $M_\R$ has a unique $G$-fixed point.
	Does there exist a $G$-invariant lattice triangulation of $P$? 
\end{question}

The following example shows that the answer to Question~\ref{q:primetriangulations} is no if one insists that the  $G$-invariant lattice triangulation is regular. In particular, the action of $G$ on the boundary of $P$ 
need not be translative. 


\begin{example}\label{e:p=5example}
	Let $G = \Z/5\Z = \langle \sigma \rangle$, and 
	 let $M = \Z^5/\Z(e_1 + \cdots + e_5)$.
	  Let $\bar{e}_i$ be the image of $e_i$ in $M$ for $1 \le i \le 5$. 
		 Assume that $G$ acts on   $M$ by cyclically shifting coordinates i.e. $\sigma \cdot \bar{e}_i = \bar{e}_{i + 1}$ for $1 \le i \le 5$, where $\bar{e}_{6} = \bar{e}_{1}$ .
 Let $\bar{f}_i := \sigma^{i - 1} \cdot (\bar{e}_1 + \bar{e}_2)$ for $1 \le i \le 5$. 
 Let $P$ be the convex hull of $\{ \bar{e}_i, \bar{f}_i : 1 \le i \le 5  \}$. 
 Then $P$ is a $4$-dimensional reflexive polytope. It has a unique interior lattice point, and its other lattice points are $\{ \bar{e}_i, \bar{f}_i : 1 \le i \le 5  \}$, which are also the vertices of $P$. 
 Let $Q = \Conv{\bar{e}_1,\bar{e}_3,\bar{f}_1,\bar{f}_2}$. Then $Q$ is $2$-dimensional and isomorphic to $[0,1]^2 \subset \R^2$ as a lattice polytope. 
 The boundary of $P$ has $15$ facets in $3$ $G$-orbits. Representatives of the orbits are 
 $F_1 = \Conv{Q, \bar{f}_3}$, 
 $F_2 = \Conv{Q, \bar{f}_5}$ and
 $F_3 = \Conv{\bar{e_1}, \bar{e_3}, \bar{f_3}, \bar{f_5} }$.
 Here $F_1$ and $F_2$ are pyramids over $Q$, and $F_3$ is a simplex.
 Any $G$-invariant regular lattice polyhedral subdivision of $P$ restricts to the trivial subdivision of $Q$ and hence is not a triangulation. On the other hand, if we fix a choice of diagonal $e$ of $Q$ e.g. $e = [\bar{e}_1,\bar{f}_2]$, then there is a unique $G$-invariant lattice triangulation $\cS'$ of the boundary of $P$ such that the restriction $\cS'|_{F_i}$ is the  
 pulling refinement of $F_i$ by the vertices of $e$ for $i = 1,2$, and $\cS'|_{F_3}$ is the trivial subdivision. That is, with the choice of $e$ above, the facets of $\cS'$ have $5$ orbits with representatives $F_1' = \Conv{e,\bar{e}_3,\bar{f}_3}$,  
 $F_1'' = \Conv{e,\bar{f}_1,\bar{f}_3}$, $F_2' = \Conv{e,\bar{e}_3,\bar{f}_5}$,  
 $F_2'' = \Conv{e,\bar{f}_1,\bar{f}_5}$ and $F_3$. 
 We then obtain a $G$-invariant lattice triangulation $\cS$ of $P$ by taking the join of $\cS'$ and the origin i.e. the facets of $\cS$ are the convex hulls of the facets of $\cS'$ and the origin.
In this example, 
 $h^*(P,\rho;t) = 1 + t + t^2 + t^3 + t^4 + t(1 + 2t + t^2)\chi_{\reg}$, where 
 $\chi_{\reg}$ is the regular representation of $G$.

\end{example}

Finally, we return to  the permutahedron example.

\begin{example}\label{e:permutahedronv3}
	We continue with the setup of Example~\ref{e:permutahedronv1} and Example~\ref{e:permutahedronv2}. 
	 That is, 
	$\psi: \Z^{d + 1} \to \Z$, $\psi(u_1,\ldots,u_{d + 1}) = \sum_{i = 1}^{d + 1} u_i$, and $M' = \psi^{-1}(\binom{d + 2}{2} )$ is an affine lattice in the real affine space $W = \psi_\R^{-1}(\binom{d + 2}{2} )$.
	The symmetric group $G = \Sym_{d + 1}$ acts on  the $(d + 1)$-permutahedron $P \subset W$.

	
	Assume furthermore that $d + 1 = p$ is prime and consider the subgroup $H = \langle (1,2,\ldots,p) \rangle \cong \Z/p\Z$ of $G$. That is, $H$ acts on $P$  by cyclically shifting coordinates. When $p = 2$, $d = 1$ and $P$ is a simplex. Assume that $p$ is an odd prime.
	We claim that there exists an $H$-invariant translative regular lattice triangulation of $P$. 	Combined with Theorem~\ref{t:mainfull}, this gives a generalization of \cite[Theorem~3.52]{EKS22} that states that condition \eqref{i:mainnondegenerate} in Theorem~\ref{t:mainfull} holds in this case. 
	
	To prove the claim, recall from Example~\ref{e:permutahedronv2} that the facets of $P$ are $F_I$ where $I = (I_1,I_2)$ is an ordered partition of $\{ 1,\ldots, p \}$, and the  vertices of $F_I$ are  precisely  all $v = (v_1,\ldots,v_{d + 1})$ such that  $\{ v_i : i \in I_1 \} = \{ 1,2,\ldots, |I_1| \}$.  Then $h \cdot F_I \cap F_I = F_{h \cdot I} \cap F_I = \emptyset$ for all $h \neq \id$ in $H$, and it follows that 
	the action of $H$ on $K = \partial P$ is translative. Let $Q$ be the interior $H$-fixed lattice point $(\frac{p + 1}{2},\ldots,\frac{p + 1}{2})$.
	The claim now follows from Corollary~\ref{c:translativeK} and Corollary~\ref{c:translativeglue}. 

\end{example}

\subsection{The two dimensional case}\label{ss:dim2}

The goal of this subsection is to prove the proposition below which implies that the existence of a $G$-invariant lattice triangulation is completely determined by
whether the equivariant $h^*$-series is a polynomial in dimension two. In particular, all conditions of Theorem~\ref{t:mainfull} are equivalent when $d = 2$. 
We continue with the notation above. Note that condition \eqref{i:avoidexamplesdim2} below is independent of the choice of $P$.

\begin{proposition}\label{p:dim2mainthm}
					Let $G$ be a finite group.
	Let $M$ be a lattice of rank $d$ and let $\rho: G \to  \Aff(M)$ be an affine representation.  
	Let $P \subset M_\R$ be a $G$-invariant $d$-dimensional lattice polytope.
	Assume that $d = 2$. Then the following conditions are equivalent:
	\begin{enumerate}
		\item\label{i:maintriangulationcpy} There exists a $G$-invariant regular lattice triangulation of $P$.  
		
		\item\label{i:mainnonregulartriangulationcpy} There exists a $G$-invariant  lattice triangulation of $P$.  
		
		\item\label{i:mainpolynomialcpy} $h^*(P , \rho ; t)$ is a polynomial.
		
		\item\label{i:avoidexamplesdim2} There does not exist $g$ in $G$ and an isomorphism $M \cong \Z^2$ such that $g$ acts on $\R^2$ either by
			\begin{enumerate}[label=(\alph*)]
			\item\label{i:4}  $90$ degree rotation about $(\frac{1}{2},\frac{1}{2})$, or, 
			\item\label{i:2} reflection about the line 
			$\{ (\frac{1}{2}, y): y \in \R \}$. 
		\end{enumerate}
	\end{enumerate}
\end{proposition}
\begin{proof}
	The implication \eqref{i:maintriangulationcpy} $\Rightarrow$ \eqref{i:mainnonregulartriangulationcpy} is clear. The implication 
	\eqref{i:mainnonregulartriangulationcpy} $\Rightarrow$  \eqref{i:mainpolynomialcpy} follows from Theorem~\ref{t:mainfull}. 
	We now verify the implication \eqref{i:mainpolynomialcpy} $\Rightarrow$ 
 \eqref{i:avoidexamplesdim2}. Suppose that \eqref{i:avoidexamplesdim2} doesn't hold. We need to show that $h^*(P, \rho;t)$ is not a polynomial. 
  Firstly, consider case \ref{i:4} above.
 Then the $g$-fixed locus $P^g$ of $P$ corresponds to the point $\{ (\frac{1}{2},\frac{1}{2}) \}$, and hence $\Ehr(P,\rho;t)(g) = \frac{1}{1 - t^2}$. Since $\det(I - \tM_\C t)(g) = (1 - t)(1 + t^2)$, 
 we have  that $h^*(P,\rho;t)(g) = \frac{1 + t^2}{1 + t}$ has a pole at $t = -1$.  
 Secondly, 
 in case \ref{i:2}, \cite[Lemma~7.3]{StapledonEquivariant} implies that $h^*(P, \rho;t)$ is not a polynomial. 

It remains to show that \eqref{i:avoidexamplesdim2} $\Rightarrow$ \eqref{i:maintriangulationcpy}. Assume that \eqref{i:avoidexamplesdim2} holds.
After possibly replacing $G$ with $G/\ker \rho$, we may assume that $G$ acts faithfully. 
 Let $c$ be a $G$-fixed point in $M_\Q$. If $c \in M$, then the pulling refinement of $P$ by $c$ is a $G$-invariant regular lattice triangulation. Hence we may assume that $c \notin M$. 

The subgroups of $\GL(2,\Z)$, up to conjugacy, are classified \cite{VosOnTwoDimensional}. They are precisely the subgroups of the dihedral group $D_4$ of symmetries of $[-1,1]^2$, and the subgroups of the dihedral group $D_6$ of symmetries of the hexagon with vertices $\{\pm e_1, \pm (e_1 + e_2), \pm e_2\}$. 
Consider an element $g \neq \id$ in $G$. 
By Remark~\ref{r:cyclotomic}, the degree $2$ characteristic polynomial of the action of $g$ on $M_\C$ is a product of cyclotomic polynomials. Hence $\ord(g) \in \{ 2,3,4,6 \}$ and if $1$ is an eigenvalue of $g$ acting on $M_\C$ then $\ord(g) = 2$ (this also follows from the classification of subgroups of $\GL(2,\Z)$). 
If $1$ is not an eigenvalue for the action of $g$ on $M_\C$, then 
by Lemma~\ref{l:affinereps} and Lemma~\ref{l:bounddenomiinatorfixedgeneral} either $\ord(g) \in \{2,4\}$ and $c \in \frac{1}{2}M$ or $\ord(g) = 3$ and $c \in \frac{1}{3}M$. If $\ord(g) = 4$ and $c \in \frac{1}{2}M$, then we see from the classification that we are in case \ref{i:4}, a contradiction. In particular, we have shown that $\ord(g) \in \{ 2,3\}$.

Let $\cS$ be a  $G$-invariant regular lattice polyhedral subdivision of $P$.
We may assume that there does not exist a $G$-invariant regular lattice polyhedral subdivision $\cS' \neq \cS$ that refines $\cS$. 
Assume that $\cS$ is not a triangulation. We will derive a contradiction.

Let $J$ be the set of lattice points in $P$ that are not vertices of $\cS$. Since $\cS$ equals the pulling refinement of $\cS$ by $J$, we conclude that $J$ is the emptyset. Hence each facet $F$ of $\cS$ is an empty polytope, in the sense that $F \cap M$ is precisely the set of vertices of $F$. On the other hand, it is well-known and straightforward to check that the only empty $2$-dimensional lattice polytopes in $\R^2$, up to unimodular equivalence, are the standard simplex and $[0,1]^2$. Fix a facet $F$ of $\cS$ that is not a simplex.
Then we may assume that  $M = \Z^2$ and $F = [0,1]^2$. Label the vertices of $F$: $u = 0$, $v = e_1$, $w = e_1 + e_2$, $x = e_2$. 

Suppose that there exists a vertex $y$ of $F$ such that $(G \cdot y) \cap F$ is not either $F \cap M$ or $e \cap M$, for some edge $e$ of $F$. Then the pulling refinement of  $\cS$ by the orbit $G \cdot y$ is a proper refinement of $\cS$, a contradiction. Hence, without loss of generality, we may assume that there exists $g,h$ in $G$ such that $g \cdot u = v$ and $h \cdot x = w$. 

First assume that either $\ord(g) = 2$ or $\ord(h) = 2$.   Without loss of generality, assume that $\ord(g) = 2$. Then $\frac{u + v}{2}$ is a $g$-fixed point. 
Suppose that $1$ is an eigenvalue of $g$ acting on $M_\C$. Then $g$ fixes a line $L$ through $\frac{u + v}{2}$ such that $u,v \notin L$. Observe that  $L$ intersects the relative interior of $F$ and hence $g \cdot F = F$. Since $L$ does not contain both $w$ and $x$, we deduce that $g \cdot x = w$. Then $L$ contains $\frac{w + x}{2}$ and hence 
$L = \{ (\frac{1}{2},y) : y \in \R \}$ and we are in case \ref{i:2}, a contradiction. We deduce that  $1$ is not an eigenvalue of $g$ acting on $M_\C$. In particular, since there is a unique $g$-fixed point in $M_\Q$, we deduce that 
$c = \frac{u + v}{2} \in \frac{1}{2}M$. If $\ord(h) = 2$, then by the same argument $c = \frac{w + x}{2} \neq \frac{u + v}{2}$, a contraction. We deduce that $\ord(h) = 3$. Then $c \in \frac{1}{3}M \cap  \frac{1}{2}M = M$, a contradiction.

We conclude that  $\ord(g) = \ord(h) = 3$ and $c \in \frac{1}{3}M$. By the classification of subgroups of $\GL(2,\Z)$, $h \in \{ g, g^{-1}\}$. We may write $\rho(g)(x) = Ax + e_1$, for some $A \in \GL(2,\Z)$. Suppose that $h = g$. Then $h \cdot x = w$ is equivalent to $\rho(g)(e_2) = Ae_2 + e_1 = e_1 + e_2$. Then $Ae_2 = e_2$ and $1$ is an eigenvalue of $g$ acting on $M_\C$, a contradiction. Hence $h = g^{-1}$. Then $h \cdot x = w$ is equivalent to $\rho(g)(e_1 + e_2) = Ae_1 + Ae_2 + e_1 = e_2$. Then 
 $\rho(g)(\frac{e_1 + e_2}{2}) = \frac{Ae_1 + Ae_2}{2} + e_1 = \frac{e_2 - e_1}{2} + e_1 = \frac{e_1 + e_2}{2}$. Hence $\frac{e_1 + e_2}{2}$ is $g$-fixed and equals $c$. 
Then $c \in \frac{1}{3}M \cap  \frac{1}{2}M = M$, a contradiction.
\end{proof}

Say that a $G$-invariant lattice polyhedral subdivision $\cS$ of $P$ is \emph{minimal} if there does not exist a $G$-invariant lattice polyhedral subdivision $\cS' \neq \cS$ that refines $\cS$. 
	A corollary of the proof of Proposition~\ref{p:dim2mainthm} is that when $d = 2$, 
there exists a  $G$-invariant lattice triangulation of $P$ if and only if 
any minimal $G$-invariant lattice polyhedral subdivision of $P$ is a triangulation.
Example~\ref{e:refinementconversefalse} shows that this does not hold in higher dimension. This example will also be used later in Remark~\ref{r:converseisfalse}. 


\begin{example}\label{e:refinementconversefalse}
	Let $G = \Z/2\Z = \langle \sigma \rangle$ and let $\rho: G \to \Aff(\Z^3)$ be defined by:
	\[
	\rho(\sigma) \left( \begin{bmatrix}
		x  \\
		y \\
		z 
	\end{bmatrix}
	\right)
	=
	\begin{bmatrix}
		-1 & 0 & -1  \\
		0 & 1 & 0 \\
		0 & 0 & 1  
	\end{bmatrix} 
	\begin{bmatrix}
		x  \\
		y \\
		z 
	\end{bmatrix} + 
	\begin{bmatrix}
		1  \\
		0 \\
		0 
	\end{bmatrix}
	\]
	Then $\rho(\sigma)$ acts by reflection through the affine hyperplane $H = \{ 2x + z = 1 \}$.
	Consider the $G$-invariant square $Q =([0,1] \times [0,1]) \times \{ 0 \}$, and let 
	$P = \Conv{Q,a,b}$, where $a = (0,0,1), b = (1,1,-1) \in H$. Then the pulling refinement $\cS_1$ of the trivial subdivision of $P$ by $\{ a,b\}$ is a minimal $G$-invariant lattice triangulation of $P$. It has four maximal simplices: the convex hulls of one of the four edges of $Q$ with $[a,b]$. 
	On the other hand, the pulling refinement $\cS_2$  of the trivial subdivision of $P$ by the vertices of $Q$ is a minimal $G$-invariant lattice polyhedral subdivision of $P$ that is not a triangulation. It has two maximal facets: the convex hull of $Q$ with $a$ and $b$ respectively. 
	We have  $h^*(P,\rho;t) = 1 + \chi_{\reg} t + t^2$, where 
	$\chi_{\reg}$ is the regular representation of $G$.
	
\end{example}

\section{The principal $A$-determinant and commutative algebra}\label{s:commutativealgebra}

In this section, we develop the connection between equivariant Ehrhart theory and commutative algebra. The relevant background is introduced in  
Section~\ref{ss:commutative} and
Section~\ref{ss:principal}. In
Section~\ref{ss:applicationsEhrhart}, we give applications to equivariant Ehrhart theory including the proofs of Theorem~\ref{t:mainfull}, Corollary~\ref{c:maincorollaryshort} and Theorem~\ref{t:monotonicity}. 
We continue with the notation of the previous section. That is, let $M \cong \Z^d$ be a rank $d$ lattice and let  $\tM = M \oplus \Z$.
Let $G$ be a finite group  
and let $\rho: G \to \Aff(M)$ be an affine representation of $G$. 
Let $P \subset M_\R$ be a $G$-invariant $d$-dimensional lattice polytope. 
As in the introduction, let $C_P$ denote the cone generated by $P \times \{1 \}$ in $\tM_\R$ and consider the graded semigroup algebra $S_P = \oplus_m S_{P,m} := \C[C_P \cap \tM]$. 
Then $S_P$ is a finitely generated $\C$-algebra with Krull dimension $d + 1$.

\subsection{Commutative algebra}\label{ss:commutative}

We first introduce the relevant background from commutative algebra. We refer the reader to \cite{AtiyahMacdonald} and \cite{StanleyIntroductionCombinatorialCommutative}
for further details.
Throughout this section, all rings are commutative rings with identity.

Let $R \subset S$ be an inclusion of rings. An element $x$ of $S$ is \emph{integral} over $R$ if $x$ is a root of a monic polynomial with coefficients in $R$. 
The set of integral elements of $S$ forms a subring of $S$ containing $R$, and 
$S$ is \emph{integral} over $R$ if every element of $S$ is integral over $R$.


\begin{lemma}\cite[Corollary~5.1-5.2]{AtiyahMacdonald}\label{l:integral=finite}
	Let $R \subset S$ be an inclusion of rings. Then $S$ is a finitely generated $R$-algebra that is integral over $R$ if and only if $S$ is a finitely generated $R$-module. 
\end{lemma}

Recall that a ring $R$ is \emph{Noetherian} if any submodule of a finitely generated $R$-module is finitely generated \cite[Proposition~6.2, Proposition~6.5]{AtiyahMacdonald}. For example, any finitely generated algebra over a field is Noetherian \cite[Corollary~7.7]{AtiyahMacdonald}. All rings that appear in this paper will be Noetherian. In the lemma below, the `if' direction follows from the Noetherian hypothesis, while the `only if' direction is
\cite[Proposition~2.16]{AtiyahMacdonald}. 

\begin{lemma}\label{l:finiteness}
	Let $R \subset S \subset T$ be inclusions of  rings. 
	Assume that $R$ is a Noetherian ring and $T$ is a finitely generated $S$-module.
	Then $S$ is a finitely generated $R$-module if and only if $T$ is a finitely generated $R$-module. 
\end{lemma}

Let $k$ be a field. Let $R = \oplus_{m \ge 0} R_m$ be a positively graded $k$-algebra with $R_0 = k$. Assume that $R$ is a finitely generated $k$-algebra. 
Let $\dim R$
be the Krull dimension of $R$. Then $\dim R$ equals the maximum number of number of (homogeneous) elements of $R$ that are algebraically independent over $k$, and $\dim R = 0$ if and only if $R$ is finite dimensional as a vector space over $k$. 

\begin{definition}\label{d:hsop}
	Let $R = \oplus_{m \ge 0} R_m$ be a positively graded, finitely generated $k$-algebra with $R_0 = k$.
	A \emph{homogeneous system of parameters (h.s.o.p.)} for $R$ is a collection of homogeneous elements $x_1,\ldots,x_n$ of positive degree in $R$ such that $n = \dim R$ and $\dim R/(x_1,\ldots,x_n) = 0$.
	 We say that a h.s.o.p. has \emph{degree $r$} if each $x_i$ has degree $r$.
	 A \emph{linear system of parameters (l.s.o.p.)} is a h.s.o.p. of degree $1$. 
\end{definition}

\begin{theorem}\cite[Theorem~16-17]{StanleyIntroductionCombinatorialCommutative}\label{t:hsop}
	Let $R = \oplus_{m \ge 0} R_m$ be a positively graded, finitely generated $k$-algebra with $R_0 = k$.
Let $x_1,\ldots,x_n$ be a collection of homogeneous elements  of positive degree in $R$. Then the following are equivalent:
	\begin{enumerate}
	\item $x_1,\ldots,x_n$ is a h.s.o.p. for $R$.
	\item $x_1,\ldots,x_n$ are algebraically independent over $k$, and $\dim R/(x_1,\ldots,x_n) = 0$.
\item\label{i:algindepfinite} $x_1,\ldots,x_n$ are algebraically independent over $k$, and 
 $R$ is a finitely generated $k[x_1,\ldots,x_n]$-module. 
 	\item  $R$ is a finitely generated $k[x_1,\ldots,x_n]$-module, and $n = \dim R$. 
\end{enumerate}
There exists a h.s.o.p. for $R$ and, furthermore, if $k$ is infinite and $R$ is generated by elements of degree $1$ as a $k$-algebra, there exists a l.s.o.p. for $R$.
\end{theorem}	

The following corollary follows immediately from Lemma~\ref{l:finiteness} and \eqref{i:algindepfinite} in Theorem~\ref{t:hsop}. 

\begin{corollary}\label{c:hsopfiniteextension}
		Let $R \subset S$ be an inclusion of positively graded, finitely generated $k$-algebras with $R_0 = S_0 = k$. Assume that $S$ is a finitely generated $R$-module. Let $x_1,\ldots,x_n$ be a collection of homogeneous elements  of positive degree in $R$. 
		Then $x_1,\ldots,x_n$ is a h.s.o.p. for $R$ if and only if $x_1,\ldots,x_n$ is a h.s.o.p. for $S$. 
\end{corollary}

Recall that a sequence of elements $x_1,\ldots,x_n$ in $R$ is a \emph{regular sequence}
if $x_i$ is not a zero divisor in $R/(x_1,\ldots,x_{i - 1})$ for $1 \le i \le n$. 
We say that $R$ is \emph{Cohen-Macaulay} if there exists a h.s.o.p. that is a regular sequence.

\begin{theorem}\label{t:CMdef}\cite[Theorem~18-19]{StanleyIntroductionCombinatorialCommutative}
		Let $R = \oplus_{m \ge 0} R_m$ be a positively graded, finitely generated $k$-algebra with $R_0 = k$.
	Then the following are equivalent:

	\begin{enumerate}
		
		\item Some h.s.o.p. is a regular sequence i.e. $R$ is Cohen-Macaulay.
		
		\item Every h.s.o.p. is a regular sequence.
		
		\item For some h.s.o.p. $x_1,\ldots,x_n$, $R$ is a free 
		$k[x_1,\ldots,x_n]$-module. 
		
		\item For every h.s.o.p. $x_1,\ldots,x_n$, $R$ is a free
		$k[x_1,\ldots,x_n]$-module. 
		
		
	\end{enumerate}
	
\end{theorem} 

We will need the following example. 

\begin{theorem}\label{t:Hochster}\cite{HochsterRingsInvariants}
	Consider the positively graded semigroup algebra $S_P = \C[C_P \cap \tM]$, where $C_P$ is the cone over $P \times \{1 \}$ in $\tM_\R$. Then $S_P$ is Cohen-Macaulay.  
\end{theorem}

Let $G$ be a finite group that acts on $R$ as a positively graded $k$-algebra.
Let $k = \C$ and let $R(G)$ be the complex representation ring of $G$ (cf. Section~\ref{ss:reptheorybasics}).
Then each graded piece $R_m$ is a $\C G$-module and we may consider the corresponding \emph{equivariant Hilbert series} $\Hilb_G(R;t) := \sum_{m \ge 0} [R_m]t^m \in R(G)[[t]]$. 
Recall from \eqref{e:detdef} that if $V$ is a finite-dimensional $\C[G]$-module then 
\begin{equation*}
	\det(I - V t) := \sum_{m \ge 0} (-1)^m [\bigwedge^m V] t^m \in R(G)[t].
\end{equation*}
The lemma below essentially follows from the proof of \cite[Theorem~1.4]{StembridgeSomePermutation}. We reproduce the proof in our setup for the convenience of the reader. 

\begin{lemma}\label{l:equivHilbertofquotient}
			Let $R = \oplus_{m \ge 0} R_m$ be a positively graded, finitely generated $\C$-algebra with $R_0 = \C$.
	Assume that $R$ is Cohen-Macaulay.  
	Let $G$ be a finite group that acts on $R$ as a positively graded $\C$-algebra.
	Let $x_1,\ldots,x_n$ be a h.s.o.p. of degree $r$ and let $J = (x_1,\ldots,x_n)$. 
	Assume that $J$ is $G$-invariant.
Then 
\[
\Hilb_G(R;t) = \frac{\Hilb_G(R/J;t) }{\det(I - J_r t^r)} \in R(G)[[t]],
\]	
where $J_r = \C x_1 + \cdots + \C x_n$ is the degree $r$ component of $J$. 
\end{lemma}
\begin{proof}

	First recall that any surjection of graded $\C G$-modules splits. Indeed, if 
	$f: V \to W$ is a surjection of graded $\C G$-modules and $h': W \to V$ is a splitting of $f$ as a map of graded $\C$-vector spaces, then one may define $h: W \to V$ by  $h(w) = \frac{1}{|G|} \sum_{g \in G} g h'(g^{-1}w)$ for all $w$ in $W$.
	
	Consider the graded subring $S = \C[x_1,\ldots,x_n]$ of $R$ and let $\mathfrak{m} = (x_1,\ldots,x_n)$. Note that each $x_i$ has degree $r$. 
	By Theorem~\ref{t:hsop}, $S$ is isomorphic to a polynomial ring in $n$ variables. 
	Since $R$ is Cohen-Macaulay, $R$ is a free $S$-module by Theorem~\ref{t:CMdef}.
	
	Consider the surjection of graded $\C G$-modules $p: R \to R/J = R/\mathfrak{m} R$, and choose a splitting of graded $\C G$-modules $\sigma: R/J \to R$.
	Let $y_1,\ldots,y_s$ be a $\C$-basis of $R/J$. Then $\sigma(y_1),\ldots,\sigma(y_s)$ generate $R$ as an $S$-module by the graded version of Nakayama's lemma (see \cite[Exercise~4.6a]{EisenbudCommutativeAlgebra} and \cite[Proposition~2.8]{AtiyahMacdonald}). Here $s$ is the rank of $R$ as a free $S$-module and hence $\sigma(y_1),\ldots,\sigma(y_s)$ is a 
	a basis of $R$ as an $S$-module \cite[Exercise~3.15]{AtiyahMacdonald}. 
	

We deduce that $\sigma$ extends to an isomorphism of 
$S$-modules
$\tilde{\sigma}: R/J \otimes_\C S \to R$ 
by $\tilde{\sigma}([r] \otimes f) = f\sigma([r])$, and $\tilde{\sigma}$ is a map of graded $\C G$-modules. We claim that 
$\Hilb_G(S;t) = \frac{1}{\det(I - J_r t^r)}$.
Indeed, if $g$ in $G$ acts on $J_r$ with eigenvalues $\lambda_1,\ldots,\lambda_s$, then evaluating characters at $g$ yields the equality
$\sum_{m \ge 0} S_m(g) t^m = \frac{1}{(1 - \lambda_1 t^r)\cdots (1 - \lambda_s t^r)}$
and the claim follows from Remark~\ref{r:deteigenvalues} (see also \cite[Lemma~3.1]{StapledonEquivariant} for example). 
The result now follows by taking equivariant Hilbert series of both sides of the isomorphism $\tilde{\sigma}$.	
\end{proof}

\subsection{The principal $A$-determinant and nondegenerate polynomials}\label{ss:principal}

In this section, we recall the relevant background on the principal $A$-determinant and on nondegenerate polynomials. We refer the reader to \cite{GKZ94} for details. 
We continue with the notation of previous sections. In particular, $P \subset M_\R$ is a $G$-invariant lattice polytope, $C_P$ is the cone generated by $P \times \{1 \}$ in $\tM_\R$ and $S_P = \oplus_m S_{P,m} = \C[C_P \cap \tM]$ is the corresponding graded semigroup algebra.


Let $A \subset M$ be a finite set with convex hull $P \subset M_\R$. Let $\tM_A$ be the sublattice of $\tM$ generated by $A \times \{ 1 \}$. 
Consider an element $F = \sum_{u \in C_P \cap \tM} \lambda_u x^u \in S_P$. The \emph{support}
of $F$ is $\supp(F) := \{ u \in C_P \cap \tM : \lambda_u \neq 0 \}$. 

\begin{definition}\label{d:Aresultant}\cite[Proposition~8.2.1]{GKZ94}
	Let $A \subset M$ be a finite set with convex hull $P \subset M_\R$. With the notation above,
	consider the polynomial ring $\Z[z_{i,a}]$ for some  formal variables $\{ z_{i,a} : 1 \le i \le d + 1, a \in A \}$. 
	Then there is an irreducible polynomial (unique up to sign) $R_A(z_{i,a}) \in \Z[z_{i,a}]$ called the \emph{$A$-resultant}
	satisfying the following property:
	
	Let $F_1,\ldots,F_{d + 1}$ be elements of $S_P$ with support contained in $A \times \{1\}$. For $1 \le i \le d + 1$, write $F_i = \sum_{a \in A} \lambda_{i,a} x^{(a,1)}$ for some $\lambda_{i,a} \in \C$. Then $F_1,\ldots,F_{d + 1}$ is a l.s.o.p. for $S_P$  if and only if $R_A(\lambda_{i,a}) \neq 0$. 
\end{definition}

\begin{remark}
	In Definition~\ref{d:Aresultant}, the statement given is different but equivalent  to that in \cite[Proposition~8.2.1]{GKZ94}. Namely, \cite{GKZ94} associate to $A$ the
	semigroup algebra $S_A := \C[C_P \cap \tM_A]$ and the (not necessarily normal) toric variety $\Proj S_A$. With the notation of Definition~\ref{d:Aresultant},  the $A$-resultant in \cite{GKZ94} is determined by the condition that  $R_A(\lambda_{i,a}) \neq 0$ if and only if  the subvariety 
	$\Proj (S_A/(F_{1},\ldots,F_{d + 1}))$ of $\Proj S_A$ is empty. 
	
	We explain why this is equivalent to Definition~\ref{d:Aresultant} cf. the proof of \cite[Theorem~4.8]{BatyrevVariations} (stated in Theorem~\ref{t:nondegenerateequiv} below). First observe that $S_A \subset S_P$ and $S_P$ is integral over $S_A$.  This follows, for example, by choosing a regular triangulation of $P$ with vertices in $A$, and then noting that  every element in $C_P \cap \tM$ is a sum of elements of $A \times \{ 1 \}$ with nonnegative rational coefficients. In particular, $S_A$ has Krull dimension $d + 1$ and $S$ is a finitely generated $S_A$-module by Lemma~\ref{l:integral=finite}.
	
	Now $\Proj (S_A/(F_{1},\ldots,F_{{d + 1}}))$ is empty if and only if $\Spec (S_A/(F_{1},\ldots,F_{{d + 1}}))$ is $0$-dimensional. The latter holds if and only if $S_A/(F_{1},\ldots,F_{{d + 1}})$ 
	has Krull dimension $0$ i.e. $F_{1},\ldots,F_{{d + 1}}$ is a l.s.o.p. for $S_A$.
	By Corollary~\ref{c:hsopfiniteextension}, $F_{1},\ldots,F_{{d + 1}}$ is a l.s.o.p. for $S_A$ if and only if  $F_{1},\ldots,F_{{d + 1}}$ is a l.s.o.p. for $S_P$.

\end{remark}

In general, the $A$-resultant is difficult to compute. We have the following simple example. 

\begin{example}\cite[Example~8.1.2c]{GKZ94}\label{e:resultantsimplex}
	Suppose that $|A| = d + 1$ i.e. the elements of $A$ are affinely independent and $P$ is a simplex. 
	Then  $R_A(z_{i,a}) = \pm \det( \{ z_{i,a} \})$, where 
	we consider the  variables $\{ z_{i,a} \}$ as entries of a $(d + 1) \times (d + 1)$ matrix
	with rows indexed by $\{1,\ldots,d + 1\}$  and columns indexed by the elements of $A$ (in some order). 
\end{example}

Let $\tN_A = \Hom(\tM_A,\Z)$ and $\tN = \Hom(\tM,\Z)$ be the dual lattices to $\tM_A$ and $\tM$ respectively. Then $(\tN_A)_\C = (\tN)_\C$ and we have the natural
pairing 
$\langle \cdot , \cdot \rangle: \tM_\C \times \tN_\C \to \C$.
Given   $v$ in $\tN_\C$, we define a $\tM$-graded $\C$-linear map:
\begin{equation*}
	\partial_v : \C[\tM] \to \C[\tM],
\end{equation*}
\[
\partial_v (x^u) = \langle u , v \rangle x^u.
\]
Note that if $v \in \tN_A$, then the map $\partial_v$ preserves $\Z[\tM_A]$.  We are now ready to define the principal $A$-determinant. 

\begin{definition}\label{d:principalAdet}\cite[10.1.A]{GKZ94}
	Let $A \subset M$ be a finite set with convex hull $P \subset M_\R$. With the notation above,
	consider the polynomial ring $\Z[z_{a}]$ for some  formal variables $\{ z_{a} : a \in A \}$. 
	Fix a choice of basis $v_1,\ldots,v_{d + 1}$ of $\tN_A$.
	For $a$ in $A$ and $1 \le i \le d + 1$, 
	let 
	$w_{i,a} := \langle (a,1), v_i \rangle z_a$.  Then $E_A(z_a) := R_A(w_{i,a}) \in \Z[z_a]$.
%

	
\end{definition}

Definition~\ref{d:principalAdet} is independent of the choice of basis of  $\tN_A$ by \cite[Corollary~8.2.2]{GKZ94}.

\begin{remark}\label{r:altdefprincipalAdet}
	Definition~\ref{d:principalAdet} could alternatively have been described as follows. 	Fix a choice of basis $v_1,\ldots,v_{d + 1}$ of $\tN_A$.
		The principal $A$-determinant $E_A(z_{a}) \in \Z[z_{a}]$ is the (unique up to sign) polynomial 	satisfying the following property:
				
			Let $F$ be an element of $S_P$ with support contained in $A \times \{1\}$, and write $F = \sum_{a \in A} \lambda_{a} x^{(a,1)}$ for some $\lambda_{a} \in \C$. For all $1 \le i \le d + 1$, write 		
			$\partial_{v_i} F = \sum_{a \in A} \lambda_{i,a} x^{(a,1)}$ for some $\lambda_{i,a} \in \C$. Then $E_A(\lambda_a) = R_A(\lambda_{a,i})$.
			
			In particular, with this notation, $\partial_{v_1} F,\ldots, \partial_{v_{d + 1}} F$ is a l.s.o.p. for $S_P$  if and only if $E_A(\lambda_a) \neq 0$.
\end{remark}

In general, the principal $A$-determinant is difficult to compute. We present two simple examples below.

\begin{example}\label{e:simplexprincipalAdet}
	We continue with Example~\ref{e:resultantsimplex}. Suppose that $|A| = d + 1$. Then $E_A(z_a) = \pm \prod_{a \in A} z_a$. 
\end{example}

\begin{example}\label{e:square}\cite[Example~10.1.3b]{GKZ94}
	Let $P = [0,1]^2 \subset \R^2$, $M = \Z^2$ and $A = P \cap M$. 
	Writing $z_{i,j} := z_{(i,j)}$ for $0 \le i,j \leq 1$, we have 
	$$E_A(z_{0,0},z_{1,0},z_{0,1},z_{1,1}) = \pm z_{0,0}z_{1,0}z_{0,1}z_{1,1}(z_{0,0}z_{1,1} - z_{1,0}z_{0,1}).$$
\end{example}

Consider an element $f = \sum_{u \in M} \lambda_u x^u \in \C[M]$. The \emph{support}
of $f$ is $\supp(f) := \{ u \in M : \lambda_u \neq 0 \}$, and the \emph{Newton polytope} $\Newt(f)$ of $f$ is the convex hull of $\supp(f)$ in $M_\R$.  
If $Z \subset M_\R$ is a subset, then we may consider the restriction $f|_Z :=  \sum_{u \in Z \cap M} \lambda_u x^u$.

\begin{definition}\label{d:nondegeneratesmooth}\cite[Definition~3.3]{BatyrevVariations}
	An element $f = \sum_{u \in M} \lambda_u x^u \in \C[M]$ with $\supp(f) \subset P \cap M$ is \emph{nondegenerate} with respect to $P$ if 
	for every nonempty face $Q$ of $P$, $f|_Q \in \C[M]$ defines a (possibly empty) smooth hypersurface in 
	$\Spec \C[M]$.  
	
\end{definition}

By setting $Q$ to be a vertex of $P$ in Definition~\ref{d:nondegeneratesmooth}, 
we see that if $f$ is nondegenerate with respect to $P$ then $\Newt(f) = P$. 
Recall that $S_{P,1}$ denotes the set of homogeneous elements of $S_P$ of degree $1$. 

\begin{definition}\cite[Definition~4.7]{BatyrevVariations}\label{d:Jacobiandef}
	Consider an element $f = \sum_{u \in M} \lambda_u x^u \in \C[M]$ with
	 $\supp(f) \subset P \cap M$.
	 Let $F = \sum_{u \in P \cap M} \lambda_u x^{(u,1)}$ be the corresponding element in $S_{P,1}$. Consider the ideal $J_{f,P} \subset S_P$ generated by the $\C$-vector space  $\{ \partial_v F : v \in \tN_\C \}$. 
	 If $\Newt(f) = P$, then $J_{f,P}$ is 
	  called the \emph{Jacobian ideal} of $f$, and the quotient ring $S_P/J_{f,P}$ is called the \emph{Jacobian ring} of $f$. 
	%
\end{definition}

We recall the following equivalent characterizations of nondegeneracy.


\begin{theorem}\label{t:nondegenerateequiv}\cite[Theorem~4.8, Proposition~4.16]{BatyrevVariations}
	Consider an element $f = \sum_{u \in M} \lambda_u x^u \in \C[M]$  with $\supp(f) \subset P \cap M$. 
	Let $F = \sum_{u \in M} \lambda_u x^{(u,1)}$ be the corresponding element in $S_{P,1}$. Let $A \subset P \cap M$ be any subset containing $\supp(f)$ and consider the corresponding element $\lambda = (\lambda_a)_{a \in A} \in \C^{A}$. Then the following conditions are equivalent:
	\begin{enumerate}
		\item $f$ is nondegenerate with respect to $P$. 
		\item 
		Any ordered $\C$-basis of 
		$(J_{f,P})_1$ is a l.s.o.p. for $S_P$.
		\item $E_A(\lambda_a) \neq 0$. 
	\end{enumerate} 
\end{theorem}

We mention that the equivalence of the last two conditions in Theorem~\ref{t:nondegenerateequiv} follows from Remark~\ref{r:altdefprincipalAdet} together with the observation that, by Theorem~\ref{t:hsop},
an ordered $\C$-basis of $(J_{f,P})_1 \subset S_{P,1}$ is a l.s.o.p. for $S_P$ if and only if any ordered $\C$-basis of $(J_{f,P})_1$ is a l.s.o.p.  for $S_P$.


Consider a function $\omega: A \to \R$, and define the corresponding \emph{weight}
of a monomial $\prod_{a \in A} z_a^{\nu_a}$ in $\C[z_a]$ to be $\sum_{a \in A} \omega(a) \nu_a$. 
The \emph{initial form} $\init_\omega(g)$ of a polynomial $g(z_a) \in \C[z_a]$ is the sum of all terms corresponding to monomials with minimal weight.
Then, for an auxiliary variable $t$,  
\begin{equation}\label{e:initdegenerationgeneral}
	g(  t^{\omega(a)} z_a ) =  t^{\mu} \init_\omega g(z_a) + \textrm{ higher order terms in } t,
\end{equation}
where $\mu \in \R$ is the minimal weight of any monomial in $g(z_a)$.  
We are interested in the initial form of the following modification of the principal $A$-determinant: 
\[
E_{A,M}(z_a) := [\tM : \tM_A]^{\Vol P} E_A(z_a)^{[\tM : \tM_A]} \in \Z[z_a]. 
\]
where $\Vol(P)$ is the normalized volume of $P$ with respect to $M$ and 
$[\tM : \tM_A]$ is the index of 
the sublattice $\tM_A$ of $\tM$. Importantly, $E_{A,M}(\lambda_a) = 0$ if and only if 
$E_{A}(\lambda_a) = 0$.

Recall from Section~\ref{ss:subdivisonspolytopes} that  $\UH(\omega)$ is the convex hull of $\{ (u,\lambda) : u \in A, \lambda \ge \omega(u) \} \subset \tM_\R$, and $\cS(\omega)$ is the polyhedral subdivision of $P$ 
consisting of 
the images under  projection $\tM_\R \to M_\R$ onto all but the last coordinate
of the bounded faces of $\UH(\omega)$.
Let $A_\omega = \{ u \in A : \omega(u) = \min_{(u,\lambda) \in \UH(\omega)} \lambda \}$.  
Consider a subset $B \subset A$ such that $B \times \{1\}$ spans $\tM_\R$ as an $\R$-vector space.  Then we may consider $E_{B,M}(z_b) \in \Z[z_b] \subset \Z[z_a]$.

\begin{theorem}\label{t:initdeg}
	\cite[Theorem 10.1.12']{GKZ94} 
	Let $A \subset M$ be a finite subset with convex hull $P \subset M_\R$ and consider a function $\omega: A \to \R$. 
	Then 
	$$\init_\omega E_{A,M}
	= \prod_{F} E_{A_{\omega} \cap F, M} \in \Z[z_a],$$ 
	where $F$ varies over the facets of $\cS(\omega)$. 
	
\end{theorem}

\begin{remark}\label{r:sturmfelsgeneralization}
	See \cite[(28)]{SturmfelsNewtonPolytope} for an analogous result to Theorem~\ref{t:initdeg} for the $A$-resultant.  
\end{remark}

\begin{example}
	Let $\cS$ be a regular lattice triangulation of $P$ and let $A$ be the set of vertices of $\cS$. Choose $\omega: A \to \R$ such that $\cS(\omega) = \cS$. Then
	$A_\omega = A$ and  
	Theorem~\ref{t:initdeg} together with Example~\ref{e:simplexprincipalAdet} imply that $\init_\omega E_{A,M}(z_a)
	= \pm c \prod_{a \in A} z_a^{\nu_a}$ for some positive integers $c$ and $\{ \nu_a : a \in A\}$ (cf. \cite[Theorem~10.1.4b]{GKZ94}).
\end{example}

\begin{remark}\label{r:secondarypolytope}
	We do not need the following remark for what follows, but we mention 
	the beautiful theorem \cite[Theorem~10.1.4]{GKZ94} that
	the Newton polytope of $E_A(z_a)$ in $\R^A$  is naturally identified with the secondary polytope $\Sigma(A)$ of $A$ \cite[Definition~5.1.6]{DRSTriangulations10}.
	The faces $F(\cS,B)$ of the secondary polytope are in bijection with pairs
	 $(\cS,B)$, where  $\cS$ is a regular (lattice) polyhedral subdivision of $P$ and $B \subset A$ contains the vertices of $\cS$. The cone corresponding  to $F(\cS,B)$ in the dual fan to $\Sigma(A)$ has relative interior 
	 $\{ \omega: A \to \R : (\cS,B) = (\cS(\omega), A_\omega) \}$. 
	 
	 The Newton polytope 
	 $\Newt(E_{A,M})$ is a dilate of $\Newt(E_A)$ and hence is combinatorially equivalent. Then it follows from the definitions that  
	 $\init_\omega E_{A,M}$ is the restriction of $E_{A,M}$ to the face of $\Newt(E_{A,M})$ corresponding to $F(\cS(\omega), A_\omega)$.
	 
	 Finally, recall that we have an action $\rho: G \to \Aff(M)$  and $P$ is $G$-invariant. We will not need this in what follows, but we mention that when $A$ is $G$-invariant,
	  Reiner introduced the \emph{equivariant secondary polytope} of $A$ in  \cite{Reiner02}, with faces in bijection with all pairs $(\cS,B)$ above such that $\cS$ and $B$ are $G$-invariant.
	 
%
%
\end{remark}

We next recall a generalization of semigroup algebras associated to cones
which will be especially useful when dealing with lattice triangulations that are not regular.

\begin{definition}\label{d:toricfacering}
	Let $\Sigma$ be a rational fan in $\tM_\R$ with support $|\Sigma|$. 
	The \emph{toric face ring} $\C[\Sigma]$ associated to $\Sigma$ (alternatively known as the \emph{deformed group ring}) is the graded $\C$-algebra that equals  
	$\C[|\Sigma| \cap \tM]$ 
	as a graded 
	$\C$-vector space, with multiplication defined as follows: for any $u_1,u_2 \in |\Sigma| \cap \tM$,
	%
	\[
	x^{u_1} \cdot x^{u_2} = \begin{cases}
		x^{u_1 + u_2} &\textrm{ if there exists } C \in \Sigma \textrm{ such that } u_1,u_2 \in C, \\
		0 &\textrm{ otherwise. }
	\end{cases}
	\]
\end{definition}

We have the following generalization of Theorem~\ref{t:Hochster}.
This theorem  is also a corollary of the more general results \cite[Theorem~1.2]{BBRCohomology07} and \cite[Theorem~1.2]{IRToric07}.

\begin{theorem}\label{t:faceringisCM}\cite[Lemma~4.6]{StanleyGeneralized87}
		Let $\Sigma$ be a rational fan with convex support. 
		 Then the corresponding toric face ring $\C[\Sigma]$ is Cohen-Macaulay.
\end{theorem}

We introduce the following notation. 
Let $\cS$ be a 
rational polyhedral subdivision of $P$. 
For each face $F$ of $\cS$, let $C_F$ denote the cone generated by $F \times \{1\}$
in $\tM_\R$. The collection of all such cones forms a rational fan $\Sigma_\cS$ called the \emph{fan over the faces} of $\cS$.  
We let $S_\cS := \C[\Sigma_\cS]$ denote the corresponding toric face ring. Then $S_\cS$ has  Krull dimension $d + 1$.
For example,  $S_\cS = S_P$ when $\cS$ is the trivial subdivision of $P$.

%
%

We have the following generalization of Definition~\ref{d:nondegeneratesmooth}. 

\begin{definition}\label{d:nondegeneratesmoothgeneralized}
	Let $\cS$ be a lattice polyhedral subdivision of $P$.
	An element $f = \sum_{u \in M} \lambda_u x^u \in \C[M]$ with $\supp(f) \subset P \cap M$ is \emph{nondegenerate} with respect to $\cS$ if 
	for every nonempty face $F$ of $\cS$, $f|_F \in \C[M]$ defines a (possibly empty) smooth hypersurface in 
	$\Spec \C[M]$.  
\end{definition}

By setting $F$ to be a vertex of $\cS$ in Definition~\ref{d:nondegeneratesmooth}, 
we see that if $f$ is nondegenerate with respect to $\cS$ then $\supp(f)$ contains all vertices of $\cS$. In particular, $\Newt(f) = P$. It also follows from Definition~\ref{d:nondegeneratesmooth} and Definition~\ref{d:nondegeneratesmoothgeneralized} that $f$ is nondegenerate with respect to $\cS$ if and only if $f|_F$ is nondegenerate with respect to $F$ for all facets $F$ of $\cS$. 

Let  $J_{f,\cS} \subset S_\Sigma$ be the ideal generated by the $\C$-vector space  $\{ \partial_v F : v \in \tN_\C \}$. Recall that given a subset $B \subset A$ such that $B \times \{1\}$ spans $\tM_\R$ as an $\R$-vector space,  then we may consider $E_{B,M}(z_b) \in \Z[z_b] \subset \Z[z_a]$.
We will need the following generalization of Theorem~\ref{t:nondegenerateequiv}. 

\begin{corollary}\label{c:nondegeneratecSequiv}
		Let $\cS$ be a lattice polyhedral subdivision of $P$.
		Consider an element $f = \sum_{u \in M} \lambda_u x^u \in \C[M]$  with $\supp(f) \subset P \cap M$. 
	Let $F = \sum_{u \in P \cap M} \lambda_u x^{(u,1)}$ be the corresponding element in $S_{\cS,1}$. Let $A \subset P \cap M$ be any subset containing $\supp(f)$ and consider the corresponding element $\lambda = (\lambda_a)_{a \in A} \in \C^{A}$. Then the following conditions are equivalent:
	\begin{enumerate}
		\item $f$ is nondegenerate with respect to $\cS$. 
		\item Any ordered $\C$-basis of 
		$(J_{f,\cS})_1$ is a l.s.o.p. for $S_\cS$.
		\item $E_{A \cap F}(\lambda_a) \neq 0$ for any facet $F$ of $\cS$. 
	\end{enumerate} 
\end{corollary}
\begin{proof}
		\cite[Lemma~6.4]{BorisovMavlyutov03} states that an ordered $\C$-basis of $(J_{f,\cS})_1$ is a l.s.o.p. for $S_\cS$ if and only if its corresponding restriction to 
		$(J_{f,F})_1$ is a l.s.o.p. for $S_F$ for all facets $F$ of $\cS$. 
	The result now follows from Theorem~\ref{t:nondegenerateequiv}. 
\end{proof}


\subsection{Applications to equivariant Ehrhart theory}\label{ss:applicationsEhrhart}

We now present our connection between commutative algebra and equivariant Ehrhart theory including several applications and examples.  We continue with the notation throughout the paper. That is, let $M \cong \Z^d$ be a rank $d$ lattice and let  $\tM = M \oplus \Z$.
Let $G$ be a finite group  
and let $\rho: G \to \Aff(M)$ be an affine representation of $G$. 
Let $P \subset M_\R$ be a $G$-invariant $d$-dimensional lattice polytope. 
Let $C_P$ denote the cone generated by $P \times \{1 \}$ in $\tM_\R$ and consider the graded semigroup algebra $S_P = \oplus_m S_{P,m} := \C[C_P \cap \tM]$.

We first develop a sequence of lemmas that will be used to prove Theorem~\ref{t:mainfull}.
Let $\cS$ be a lattice polyhedral subdivision of $P$, with corresponding toric face ring $S_\cS = \oplus_m S_{\cS,m}$. 	Consider an element $f = \sum_{u \in M} \lambda_u x^u \in \C[M]$ with
$\supp(f) \subset P \cap M$. 	Let $F = \sum_{u \in P \cap M} \lambda_u x^{(u,1)}$ be the corresponding element in $S_{\cS,1}$.
Then $f$ is $G$-invariant if and only if $F \in S_\cS$ is $G$-invariant if and only if $\lambda_{g \cdot u} = \lambda_{u}$ for all $g$ in $G$ and $u$ in $M$. 
Recall that  $J_{f,\cS} \subset S_\cS$ is the ideal generated by the $\C$-vector space  $(J_{f,\cS})_1 = \{ \partial_v F : v \in \tN_\C \}$, where $\partial_v (x^u) = \langle u , v \rangle x^u$. 

\begin{lemma}\label{l:JacobianisomtotM}
	Consider an element $f = \sum_{u \in M} \lambda_u x^u \in \C[M]$. Assume that 
	$\Newt(f) = P$ and $f$ is $G$-invariant. Then the $\C$-vector space  $(J_{f,\cS})_1 \subset S_{\cS,1}$ is $G$-invariant and isomorphic to $\tM_\C$ as a $\C G$-module. 
\end{lemma}
\begin{proof}
	Recall that $\tN_C = \Hom_\C(\tM_\C,\C)$ is naturally a $\C G$-module. Explicitly, 
	the linear action $\rho: G \to \GL(\tM_\C)$ induces a linear action 
	$\rho^*: G \to \GL(\tN_\C)$ 
	determined by $\langle u , \rho^*(g)(v) \rangle = \langle \rho(g^{-1})(u) , v \rangle $ for all $u$ in $\tM_\C$ and $v$ in $\tN_\C$. Consider the map of $\C$-vector spaces: 
	\[
	\psi: \tN_\C \to (J_{f,\cS})_1,
	\]
	\[
	\psi(v) = \partial_v F. 
	\]
	We claim that $\psi$ is a map of $\C G$-modules. 
	Indeed, using the $G$-invariance of $f$, 
	we compute
	\[
	g \cdot \partial_v F = 
	\sum_{u \in M} \lambda_u \langle (u,1), v \rangle x^{(g \cdot u, 1)} = \sum_{u \in M} \lambda_u \langle \rho(g^{-1})((u,1)), v \rangle x^{(u, 1)} = \partial_{\rho^*(g)(v)} F. 
	\]
	We also claim that $\psi$ is injective. Suppose that $\partial_v F = 0$ for some $v$ in $\tN_\C$.  Then $\langle (u,1) , v \rangle = 0$ for all $u \in \supp(f)$. Since $\Newt(f) = P$, the linear span of $\supp(f) \times \{1 \}$ equals $\tM_\C$ and hence $v = 0$. We conclude that $\psi$ is an isomorphism of $\C G$-modules. 
	
	It remains to observe that $\tN_\C$ is isomorphic to $\tM_\C$ as a $\C G$-module. 
	This follows since if $\rho(g)$ has eigenvalues $\{\lambda_1,\ldots, \lambda_{d + 1}\}$, then  $\rho^*(g)$ has eigenvalues $\{\lambda_1^{-1},\ldots, \lambda_{d + 1}^{-1}\}$. On the other hand, since $\rho(g)$ is integer valued and each eigenvalue is a root of unity, $\lambda_i^{-1}$ is the complex conjugate of $\lambda_i$ and is also an eigenvalue of $\rho(g)$ for all $1 \le i \le d + 1$ (cf. the proof of Proposition~\ref{p:Nreciprocity}).  
\end{proof}

\begin{lemma}\label{l:invarianttriangulation}
		Let $\cS$ be a (not necessarily regular) $G$-invariant triangulation of $P$.
		Let $A$ be the set of vertices of $\cS$ and let $f = \sum_{a \in A} x^a \in \C[M]$. Then $f$ is $G$-invariant and nondegenerate with respect to $\cS$.
		Moreover, any ordered $\C$-basis of 
		$(J_{f,\cS})_1$ is a l.s.o.p. for $S_\cS$, and $h^*(P,\rho;t) = \Hilb_G(S_{\cS}/J_{f,\cS};t)$.
		
%
\end{lemma}
\begin{proof}
	By construction,  $f$ is $G$-invariant. 
	By Example~\ref{e:simplexprincipalAdet} and Corollary~\ref{c:nondegeneratecSequiv}, $f$ is nondegenerate with respect to $\cS$. 
By Corollary~\ref{c:nondegeneratecSequiv},  any ordered $\C$-basis of 
$(J_{f,\cS})_1$ is a l.s.o.p. for $S_\cS$. 
 By Lemma~\ref{l:JacobianisomtotM}, 
$(J_{f,\cS})_1$ is $G$-invariant and isomorphic to $\tM_\C$ as a $\C G$-module. 
By Lemma~\ref{l:equivHilbertofquotient} and Theorem~\ref{t:faceringisCM},
 $h^*(P,\rho;t) = \Hilb_G(S_{\cS}/J_{f,\cS};t)$. 
\end{proof}

When $\cS$ is the trivial subdivision, the lemma below can be compared with \cite[Theorem~3.56]{EKS22} and the proof of \cite[Proposition~5.5]{ASV20}. 

\begin{lemma}\label{l:applysquare}
	Let $\cS$ be a $G$-invariant regular lattice polyhedral subdivision of $P$. Suppose there exists a face $Q$ of $\cS$ such that $Q \cong [0,1]^2$ as a lattice polytope. Assume that there exists disjoint edges $e_1,e_2$ of $Q$ such that the elements of $e_i \cap M$ lie in the same $G$-orbit for $i = 1,2$. 
%
	 Then there does not exist  $f \in \C[M]$ that is $G$-invariant and nondegenerate with respect to $\cS$. 
\end{lemma}
\begin{proof}
	Suppose that  $f \in \C[M]$ is $G$-invariant with $\supp(f) \subset P \cap M$.
	 Let $A = Q \cap M$ and write $A = \{ u_{0,0},u_{1,0},u_{0,1},u_{1,1} \}$, where $e_1 \cap M = \{ u_{0,0},u_{1,0} \}$ and $e_2 \cap M = \{ u_{0,1},u_{1,1} \}$.
	 Then $f|_Q = \sum_{0 \le i,j \le 1} \lambda_{i,j} x^{u_{i,j}}$ for some $\lambda_{i,j} \in \C$ with $\lambda_{0,0} = \lambda_{1,0}$ and $\lambda_{0,1} = \lambda_{1,1}$. 
	By Example~\ref{e:square}, 	$$E_A(\lambda_{0,0},\lambda_{1,0},\lambda_{0,1},\lambda_{1,1}) = \pm \lambda_{0,0}\lambda_{1,0}\lambda_{0,1}\lambda_{1,1}(\lambda_{0,0}\lambda_{1,1} - \lambda_{1,0}\lambda_{0,1}) = 0.$$
	By Theorem~\ref{t:nondegenerateequiv}, $f|_Q$ is not nondegenerate with respect to $Q$. By Definition~\ref{d:nondegeneratesmoothgeneralized}, $f$ is not nondegenerate with respect to $\cS$.

\end{proof}

\begin{lemma}\label{l:existencenondegeneraterefinement}
	Let $\cS$ be a $G$-invariant regular lattice polyhedral subdivision of $P$.
	If there exists $f \in \C[M]$ that is $G$-invariant and nondegenerate with respect to $\cS$, then there exists $f \in \C[M]$ that is $G$-invariant and nondegenerate with respect to $P$.
\end{lemma}
\begin{proof}
	Let $A = P \cap M$. 
	Let $f = \sum_{a \in A} \lambda_a x^a \in \C[M]$ be $G$-invariant and nondegenerate with respect to $\cS$. 
	By Remark~\ref{r:invariantheight}, we may assume there exists a $G$-invariant function $\omega: A \to \R$ such that $\cS = \cS(\omega)$.
	By Theorem~\ref{t:initdeg} and Corollary~\ref{c:nondegeneratecSequiv}, 
	\[
	\init_\omega E_{A,M}(\lambda_a) = \prod_{F} E_{A_{\omega} \cap F, M}(\lambda_a) \neq 0. 
	\]
     Above, $F$ varies over the facets of $\cS$. 
     By
	\eqref{e:initdegenerationgeneral}, 
	\[
	\lim_{t \to 0} t^{-\mu} E_{A,M}(  t^{\omega(a)} \lambda_a ) =  \init_\omega E_{A,M}(\lambda_a) \ne 0,
	\]
	where $t \in \R$ and $\mu \in \R$ is the minimal weight (with respect to $\omega$) of any monomial in $E_{A,M}(z_a)$. In particular, for $t$ positive and sufficiently small, $E_{A,M}(  t^{\omega(a)} \lambda_a ) \ne 0$, and Theorem~\ref{t:nondegenerateequiv} implies that $\sum_{a \in A} t^{\omega(a)} \lambda_a x^a$ is $G$-invariant and nondegenerate with respect to $P$. 
\end{proof}

\begin{remark}\label{r:converseisfalse}
	The converse to Lemma~\ref{l:existencenondegeneraterefinement} is false. 
	A counterexample is given by Example~\ref{e:refinementconversefalse}. In this example, $G = \Z/2\Z$, $P$ is a bipyramid over a $G$-invariant square $Q =([0,1] \times [0,1]) \times \{ 0 \}$, and 
	$\cS_1$ is a $G$-invariant regular lattice triangulation. 
	By Lemma~\ref{l:invarianttriangulation} and Lemma~\ref{l:existencenondegeneraterefinement} it follows that there exists $f \in \C[M]$ that is $G$-invariant and nondegenerate with respect to $P$. On the other hand, $\cS_2$ is a $G$-invariant regular lattice polyhedral subdivision such that each facet contains $Q \cong [0,1]^2$.
	 The $G$-orbits of $Q \cap M$ are the sets of vertices of two disjoint edges of $Q$ and Lemma~\ref{l:applysquare} implies that  that there does not exist $f \in \C[M]$ that is $G$-invariant and nondegenerate with respect to $\cS_2$.
\end{remark}

%

We are now ready to prove Theorem~\ref{t:mainfull}. 

\begin{proof}[Proof of Theorem~\ref{t:mainfull}]
	Suppose there exists a (not necessarily regular) $G$-invariant lattice triangulation $\cS$ of $P$. Then Lemma~\ref{l:invarianttriangulation} implies that $h^*(P,\rho;t)$ is the equivariant Hilbert series of a graded $\C G$-module and hence is effective. 
%
	The implication \eqref{i:maintriangulation} $\Rightarrow$ \eqref{i:mainnondegenerate} follows from Lemma~\ref{l:invarianttriangulation} and Lemma~\ref{l:existencenondegeneraterefinement}. 
		The implication 	\eqref{i:mainnondegenerate} $\Rightarrow$
	\eqref{i:mainhilbert} follows from Theorem~\ref{t:nondegenerateequiv} and Lemma~\ref{l:JacobianisomtotM}.
		The implication \eqref{i:mainhilbert}	 $\Rightarrow$
	\eqref{i:maineffective} follows from Theorem~\ref{t:Hochster} and  Lemma~\ref{l:equivHilbertofquotient}. Explicitly, $h^*(P,\rho;t) = \Hilb_G(S/(F_1,\ldots,F_{d + 1}))$.
	Finally, the implication \eqref{i:maineffective} $\Rightarrow$
	\eqref{i:mainpolynomial} is Remark~\ref{r:effectiveimpliespolynomial}. 

\end{proof}

We have the following corollary of the proof  of Theorem~\ref{t:mainfull}, and corresponding question about a possible extension of the effectiveness conjecture (Conjecture~\ref{c:origmod}). The key difference between conditions \eqref{i:mainhilbertv2} and \eqref{i:mainhilbertv3} in Corollary~\ref{c:extendproof} below is that \eqref{i:mainhilbertv3} allows nonregular $G$-invariant lattice polyhedral subdivisions $\cS$ (cf. Remark~\ref{r:sturmfelsgeneralization} and Lemma~\ref{l:existencenondegeneraterefinement}). 

\begin{corollary}\label{c:extendproof}
			Let $G$ be a finite group.
	Let $M$ be a lattice of rank $d$ and let $\rho: G \to  \Aff(M)$ be an affine representation.  
	Let $P \subset M_\R$ be a $G$-invariant $d$-dimensional lattice polytope. 
	Let $\tM = M \oplus \Z$ and let $C_P$ be the cone over $P \times \{ 1\}$ in $\tM_\R$ with corresponding graded semigroup algebra $S_P = \C[C_P \cap \tM]$.
	Consider the following conditions:
	
		\begin{enumerate}
		
		\item\label{i:mainhilbertv2} 
		There exists a l.s.o.p. $F_1,\ldots,F_{d + 1}$ of $S_P$ such that 
		$[\C F_1 + \cdots + \C F_{d + 1}] = [\tM_\C]$ in $R(G)$.
		
		\item\label{i:mainhilbertv3} 
		There exists a $G$-invariant lattice polyhedral subdivision $\cS$ of $P$ and 
		 a l.s.o.p. $F_1,\ldots,F_{d + 1}$ of $S_\cS$ such that 
		$[\C F_1 + \cdots + \C F_{d + 1}] = [\tM_\C]$ in $R(G)$.
		
		\item\label{i:maineffectivev2} $h^*(P , \rho ; t)$ is effective.
		
		\item\label{i:mainpolynomialv2} $h^*(P , \rho ; t)$ is a polynomial.
		
	\end{enumerate}
	
	Then the following implications hold: 	\eqref{i:mainhilbertv2} $\Rightarrow$ \eqref{i:mainhilbertv3} $\Rightarrow$
	\eqref{i:maineffectivev2} $\Rightarrow$
	\eqref{i:mainpolynomialv2}.

\end{corollary}
\begin{proof}
	 The implications \eqref{i:mainhilbertv2} $\Rightarrow$ 
	\eqref{i:maineffectivev2} $\Rightarrow$
	\eqref{i:mainpolynomialv2} are Theorem~\ref{t:mainfull}. Also, 
	\eqref{i:mainhilbertv2} $\Rightarrow$ \eqref{i:mainhilbertv3}  by letting $\cS$ be the trivial subdivision, and \eqref{i:mainhilbertv3} $\Rightarrow$ 
	\eqref{i:maineffectivev2} follows from 
	Lemma~\ref{l:equivHilbertofquotient} and Theorem~\ref{t:faceringisCM}.
\end{proof}


\begin{question}\label{q:extendconjecture}
	To what extent are conditions of Corollary~\ref{c:extendproof} equivalent?
\end{question}

%

In the examples presented in this paper, we will show that the answer to Question~\ref{q:extendconjecture} is that all conditions  of Corollary~\ref{c:extendproof} are equivalent.  The only exception is Example~\ref{e:primenofixedpts}, where we leave this question open (cf. Question~\ref{q:primetriangulations}).

\begin{proof}[Proof of Corollary~\ref{c:maincorollaryshort}]
	Suppose there exists a $G$-invariant triangulation $\cS$ of $P$ with vertices in $\frac{1}{N} M$. Recall that $S_\cS = \oplus_{m \in \Z_{\ge 0}} S_{\cS,m}$ is the toric face ring associated to $\cS$. 
	Consider  the subalgebra $S' = \oplus_{m \in N\Z_{\ge 0}} S_{\cS,m} $ of $S_\cS$. 
	Then $S_\cS$ is integral over $S'$ and hence $S_\cS$ is a finite $S'$-module by Lemma~\ref{l:integral=finite}. 
	
	
	Recall that $\tM = M \oplus \Z$ and 	
	$\HT: \tM \to \Z$ is projection onto the last coordinate, with corresponding linear map 
	$\HT_\R: \tM_\R \to \R$. As in Remark~\ref{r:altformulation}, we may consider the affine lattice $M' = \HT^{-1}(N)$ in the real affine space $\HT_\R^{-1}(N)$, with corresponding affine representation $\rho': G \to \Aff(M')$ and the $G$-invariant polytope $P' := C_P \cap \HT_\R^{-1}(N)$. In particular, the equivariant Hilbert series of $S'$ is $\Ehr(P',\rho';t)$. The latter is naturally identified with 
	$\Ehr(P,\rho_N;t)$ where  $\rho_N: G \to \Aff(\frac{1}{N} M)$ is the corresponding affine representation of $G$ on $\frac{1}{N} M$ considered in the introduction and Section~\ref{ss:equivarianthNstar}.

	
	As in the proof of Theorem~\ref{t:mainfull}, by Lemma~\ref{l:JacobianisomtotM} and  Lemma~\ref{l:invarianttriangulation} there exists a h.s.o.p. $F_1,\ldots,F_{d + 1}$ for $S'$ of degree $N$ (here we view $S'$ as graded by $N\Z$ rather than $\Z$) such that $\C F_1 + \cdots + \C F_{d + 1}$ is $G$-invariant and isomorphic to $\tM_\C$ as a $\C G$-module. 
	By Corollary~\ref{c:hsopfiniteextension}, $F_1,\ldots,F_{d + 1}$ is a h.s.o.p. of degree $N$ of $S_\cS$. 
By Lemma~\ref{l:equivHilbertofquotient} and Theorem~\ref{t:faceringisCM}, 
$h^*_N(P,\rho;t^N) = \Hilb_G(S_\cS/(F_1,\ldots,F_{d + 1}))$.
 In particular, $h^*_N(P,\rho;t)$ is effective. 
		
Finally, if $h^*_N(P,\rho;t)$ is effective, then $h^*_N(P,\rho;t^N)(1) = h^*(P;t)(1 + t + \cdots + t^{N- 1})^{d + 1}$ has nonnegative integer coefficients and 
encodes the dimensions of the representations of the coefficients of  $h^*_N(P,\rho;t^N)$. In particular, $h^*_N(P,\rho;t)$ is a polynomial (cf. Remark~\ref{r:effectiveimpliespolynomial}). 
	
\end{proof}

\begin{proof}[Proof of Theorem~\ref{t:monotonicity}]
	Let $\cS$ be a $G$-invariant lattice triangulation of $P$ that restricts to a $G$-invariant lattice triangulation of $Q$.     Then restriction induces a 	$G$-invariant
	surjection  	$p: S_\cS \to S_{\cS|_Q}$   of graded $\C$-algebras between the correponding toric face rings. Explicitly, for any $u \in C_P \cap \tM$, 
	$$p(x^u) = \begin{cases}
		x^u &\textrm{ if } u \in C_Q, \\
		0  &\textrm{ otherwise.}
	\end{cases}
$$
		Let $A$ be the set of vertices of $\cS$, and let $f = \sum_{a \in A} x^a \in \C[M]$. 
		If $F = \sum_{a \in A} x^{(a,1)}$ is the element in $S_{\cS,1}$ corresponding to $f$, then $p(F)$ is the element in $(S_{\cS|_Q})_1$ corresponding to  $f|_Q$. 		Let $\tM_Q$ be the intersection of $\tM$ with the linear span of $C_Q$.
		For any $v \in \tN_\C = \Hom_\C(\tM_\C,\C)$,  we have $p(\partial_v F) = \partial_{v_Q} p(F)$, where $v_Q$ is the restriction of $v$ to $(\tM_Q)_\C$. 
		It follows that $p(J_{f,\cS}) = J_{f|_Q,\cS|_Q}$, and $p$ induces a 
		map of graded $\C G$-modules
		$ 
		\bar{p}: S_\cS/J_{f,\cS} \to S_{\cS|_Q}/J_{f|_Q,\cS|_Q}.
		$		
		By Lemma~\ref{l:invarianttriangulation}, 
		$h^*(P,\rho;t) = \Hilb_G(S_\cS/J_{f,\cS};t)$ and 
		 $h^*(P,\rho_Q;t) = \Hilb_G(S_{\cS|_Q}/J_{f|_Q,\cS|_Q};t)$.
%
		The result now follows from the fact that $\bar{p}$ is surjective.
		

\end{proof}

The proposition below gives another example where the effectiveness conjecture (Conjecture~\ref{c:origmod}) holds. Recall that $\Vol(P) \in \Z_{>0}$ denotes the normalized volume of $P$.

\begin{proposition}\label{p:d+2}
	Assume that $P$ has $d + 2$ vertices.
	 Let $A$ be the set of vertices of $P$ and assume that $A$ affinely generates $M$ i.e.
	$M = \{ \sum_{a \in A} \lambda_a a : \lambda_a \in \Z, \sum_{a} \lambda_a = 1 \}$. 
	Suppose that there does not exist a $G$-invariant 
	lattice triangulation of $P$. Then the following are equivalent: 
	\begin{enumerate}
		\item\label{i:d+2_1} $\Vol(P)$ is odd. 
		\item\label{i:d+2_2} There exists $f \in \C[M]$ that is $G$-invariant and nondegenerate with respect to $P$. 
		
		\item\label{i:d+2_3}
		There exists a l.s.o.p. $F_1,\ldots,F_{d + 1}$ of $S_P$ such that 
		$[\C F_1 + \cdots + \C F_{d + 1}] = [\tM_\C]$ in $R(G)$.
		
		\item\label{i:d+2_4} $h^*(P , \rho ; t)$ is effective.
		
		\item\label{i:d+2_5} $h^*(P , \rho ; t)$ is a polynomial.
	\end{enumerate}
\end{proposition}
\begin{proof}
	Theorem~\ref{t:mainfull} gives the implications \eqref{i:d+2_2} $\Rightarrow$ \eqref{i:d+2_3} $\Rightarrow$ \eqref{i:d+2_4} $\Rightarrow$ \eqref{i:d+2_5}.
	Let  $A = \{ x_1,\ldots,x_{d + 2}\}$. 
	Configurations of $d + 2$ points that affinely span a $d$-dimensional space are well studied, and we refer the reader to \cite[Section~2.4.1]{DRSTriangulations10} and \cite[Section~7.1.B]{GKZ94} for details. 
	There is a unique (up to scaling) affine relation between the elements of $A$. 
%
	This implies that
	there exists unique positive integers $a_1,\ldots,a_\ell$ and $a_1',\ldots,a_{\ell'}'$ such that $\ell + \ell' \le d + 2$, $\sum_{i = 1}^\ell a_i = \sum_{i = 1}^{\ell'} a_i'$, $\gcd(a_1,\ldots,a_\ell,a_1',\ldots,a_{\ell'}') = 1$ and $a_1 x_1 + \cdots + a_\ell x_\ell - (a_1' x_{\ell + 1} + \cdots + a_{\ell'}' x_{d + 2}) = 0$.  
	By \cite[Lemma~2.4.2]{DRSTriangulations10}, there exists precisely two (regular) lattice triangulations of $P$ with vertices in $A$:  $T$ with facets $\{ \Conv{A \smallsetminus \{ x_i \}} : 1 \le i \le \ell \}$, and $T'$ with facets $\{ \Conv{A \smallsetminus \{ x_i \}} : \ell + 1 \le i \le \ell + \ell' \}$.
	Since $G$ acts permuting $A$, $G$ acts on $\{ T, T' \}$. By assumption, there does not exist a $G$-invariant triangulation of $P$, and hence there exists $g$ in $G$ such that $g \cdot T = T'$. Hence $\ell = \ell'$ and, after possibly reindexing coordinates, we may assume that $g \cdot x_i = x_{\ell + i}$ and $a_i = a_{i}'$  for $1 \le i \le \ell$.
	
	By \cite[Theorem~10.1.4]{GKZ94} and \cite[Theorem~10.1.14]{GKZ94}, 
	the Newton polytope of $E_A$ (equal to the secondary polytope of $A$ by Remark~\ref{r:secondarypolytope}) is isomorphic to a unit interval, and 
	up to multiplication by a nonzero scalar and a monomial, $E_A(z_1,\ldots,z_{d + 1})$ equals $\prod_{i = 1}^\ell z_i - (-1)^{\Vol(P)} \prod_{i = 1}^\ell z_{\ell + i}$. In particular, if $\Vol(P)$ is odd, then $E_A(1,\ldots,1) \neq 0$. By Theorem~\ref{t:nondegenerateequiv}, we deduce that \eqref{i:d+2_1} $\Rightarrow$ \eqref{i:d+2_2}. 
	
	It remains to show that \eqref{i:d+2_5} $\Rightarrow$ \eqref{i:d+2_1}. 
	Note that $Q = \Conv{x_1,\ldots,x_{2\ell}}$ is the 
	unique minimal face $Q$ of $P$ that is not a simplex. Note that $Q$ is $G$-invariant.
	Since $A$ affinely generates $M$, it follows that $P$ is isomorphic to a free join of $Q$ and a standard simplex. In particular,  $\Vol(P) = \Vol(Q)$ and Lemma~\ref{l:freejoin} implies that the equivariant $h^*$-polynomials of $P$ and $Q$ agree. After possibly replacing $P$ with $Q$, we reduce to the case when $2 \ell = d + 2$. 	 Since $A$ affinely generates $M$, then  by regarding $M$ as an affine lattice we may assume that we are in the case of Example~\ref{e:d+2v2}. Consider an element $g$ in $G$ such that  $g \cdot T = T'$. By Example~\ref{e:d+2v2}, $h^*(P,\rho;t)(g)$ is a polynomial if and only if $\Vol(P)$ is odd, and we deduce that 
	\eqref{i:d+2_5} $\Rightarrow$ \eqref{i:d+2_1}.

\end{proof}

The example below shows that the converse to the implication \eqref{i:maintriangulation}  $\Rightarrow$
\eqref{i:mainnondegenerate} in Theorem~\ref{t:mainfull} is false. It also shows that 
$h^*(P,\rho;t)$ being effective  (\eqref{i:maineffective} in Theorem~\ref{t:mainfull}) does not necessarily imply that $P$ admits a $G$-invariant lattice triangulation. 
Moreover, it provides a counterexample to Conjecture~\ref{c:trivialrepalwaysappears}.

\begin{example}\label{e:d+2v4}
	With the notation of Example~\ref{e:d+2v2}, let
	 $d = 4$, $r = 3$ and $a_1 = a_2 = a_3 = 1$.  That is, let 
    Let $e_1,e_2,e_3,f_1,f_2,f_3$ be a basis of $\Z^{6}$, and let 
     $\bar{e}_i$ and $\bar{f_i}$ denote the images of $e_i$ and $f_i$ respectively in the lattice
    $L := \Z^{6}/\Z(e_1 + e_2 + e_3 - f_1 - f_2 - f_3)$. 
		Consider the group homomorphism  $\psi: L \to \Z$ satisfying $\psi(\bar{e}_i) = \psi(\bar{f}_i) = 1$ for all $1 \le i \le r$. 
		As in Example~\ref{e:constructaffine},  
	$M' = \psi^{-1}(1)$ is an affine lattice in the real affine space $W = \psi_\R^{-1}(1)$. 
	Let $P = \Conv{\bar{e}_1,\bar{e}_2,\bar{e}_3,\bar{f}_1,\bar{f}_2,\bar{f}_3} \subset W$.
	Then $P$ is a $4$-dimensional lattice polytope with $6$ vertices and
	$h^*(P;t) = 1 + t + t^2$. 
	Let $G = \Z/2\Z = \langle \sigma \rangle$ and consider the affine representation $\rho: G \to \Aff(M')$ determined by $\sigma \cdot \bar{e}_i = \bar{f}_i$ for $1 \le i \le 3$. 	Since the only lattice points of $P$ are its $d + 2$ vertices, as in the proof of Proposition~\ref{p:d+2}, $P$ has precisely two (regular) lattice triangulations $T,T'$, and $\sigma \cdot T = T'$. In particular, there does not exist a $G$-invariant lattice triangulation of $P$. Since $\Vol(P) = 3$ is odd, 
	Proposition~\ref{p:d+2} implies that there exists a $G$-invariant polynomial that is nondegenerate with respect to $P$. 
	
	Moreover, by Example~\ref{e:d+2v2}, $h^*(P,\rho;t)(\sigma) = 1 - t + t^2$, and hence  $h^*(P,\rho;t) = 1 + \chi t + t^2$, where $\chi$ is the nontrivial representation of $G$ i.e. $\chi(\sigma) = -1$. In particular, since the multiplicity of the trivial representation in the (nonzero) linear term of $h^*(P,\rho;t)$ is zero, this gives a counterexample to Conjecture~\ref{c:trivialrepalwaysappears}. 

\end{example}


\begin{question}\label{q:trivialtriangulation}
	Does Conjecture~\ref{c:trivialrepalwaysappears} hold if we further assume that $P$ admits a $G$-invariant lattice triangulation?
\end{question}


The example below shows that Proposition~\ref{p:d+2} can not be extended to the case when $P$ has $d + 3$ vertices. It 
also shows that the converse to the implication 
\eqref{i:mainnondegenerate} $\Rightarrow$
\eqref{i:mainhilbert} is false in Theorem~\ref{t:mainfull}, and, in particular,
 provides another counterexample to \cite[Conjecture~12.1]{StapledonEquivariant}. 

\begin{example}\label{e:d+3counterexample}
	Let $M = \Z^3/\Z(e_1 + e_2 + e_3) \oplus \Z$ and let $\bar{e}_i$ denote the image of $e_i$ in $M$ for $1 \le i \le 3$. Let $Q = \Conv{\bar{e}_1,\bar{e}_2,\bar{e}_3}$ and let $P = Q \times [0,e_4]$. Let $G = \Z/3\Z$ act on $M$ by cycling the first $3$ coordinates. Then $P$ is a $G$-invariant lattice polytope with $6$ vertices and 
		$h^*(P,\rho;t) = 1 + (1 + \chi_{\reg})(t + t^2)$, where $\chi_{\reg}$ is the  regular representation of $G$. Let $\cS$ be the pulling refinement of $P$ at $e_4$. Then $\cS$ has $4$ facets: the simplex $\Conv{Q,e_4}$ and $G \cdot F$, where $F = \Conv{[\bar{e}_1,\bar{e}_2] \times [0,e_4], e_4}$ is not a  simplex. Applying a pulling refinement to $F$ at any vertex of   
		$[\bar{e}_1,\bar{e}_2] \times [0,e_4]$ gives a lattice triangulation $\T$ of $F$. 
		Then there is a unique $G$-invariant (non-regular) lattice triangulation of $P$ that contains $\Conv{Q,e_4}$ and restricts to $g \cdot \T$ on $g \cdot F$ for all $g$ in $G$. On the other hand, by Lemma~\ref{l:applysquare},
		 there does not exist $f \in \C[M]$ that is $G$-invariant and nondegenerate with respect to $P$. In particular, by Theorem~\ref{t:mainfull}, 
       there does not exist a $G$-invariant regular lattice triangulation of $P$. 
	
	Let $x_1,x_2,x_3,x_4$ be the degree $1$ variables in $S_P$ corresponding to the lattice points $\bar{e}_1,\bar{e}_2,\bar{e}_3,0$ of $Q$. Let $y_1,y_2,y_3,y_4$ denote the degree $1$ variables in $S_P$ corresponding to the lattice points $\bar{e}_1 + e_4,\bar{e}_2 + e_4,\bar{e}_3 + e_4,e_4$ of $Q \times \{ e_4\}$. If $\sigma$ is a generator of $G$, let  
	$\theta_1 = x_1 + x_2 + 2x_3 + y_1 + 2y_2 + y_3$, $\theta_2 = \sigma \cdot \theta_1$, $\theta_3 = \sigma^2 \cdot \theta_1$ and $\theta_4 = x_4 - y_4$. 
	Then one may verify using Macaulay2 that $\theta_1,\theta_2,\theta_3,\theta_4$ is a l.s.o.p. for $S_P$. Also, 
	$[\C \theta_1 + \cdots + \C \theta_4] = [\tM_\C] = 1 + \chi_{\reg}$ in $R(G)$.	
\end{example}

Below we show that Example~\ref{e:p=5example} is another example that 
the converse to the implication 
\eqref{i:mainnondegenerate} $\Rightarrow$
\eqref{i:mainhilbert} is false in Theorem~\ref{t:mainfull}. In particular, it
provides another counterexample to  \cite[Conjecture~12.1]{StapledonEquivariant}. 

\begin{example}\label{e:p=5examplev2}
	We continue with Example~\ref{e:p=5example}. 
	Recall that $G = \Z/5\Z = \langle \sigma \rangle$ and  $P$ is a $4$-dimensional reflexive polytope with 
	10 vertices in two orbits $\{ \bar{e}_i : 1 \le i \le 5  \}$ and $\{ \bar{f}_i : 1 \le i \le 5  \}$, and that $P$ admits a non-regular $G$-invariant lattice triangulation.  	
	Recall that $P$ has a face 
	   $Q = \Conv{\bar{e}_1,\bar{e}_3,\bar{f}_1,\bar{f}_2} \cong [0,1]^2$. Then 
	    Lemma~\ref{l:applysquare} implies that 
	   there does not exist $f \in \C[M]$ that is $G$-invariant and nondegenerate with respect to $P$. 
    Let $x_i$ and $y_i$ denote the degree $1$ variables in $S_P$ corresponding to the lattice points $\bar{e}_i$ and $\bar{f}_i$ respectively for $1 \le i \le 5$.
    Let $\theta_i = \sigma^{i - 1} \cdot (x_1 + y_2)$ for $1 \le i \le 5$. Then    
    one may verify using Macaulay2 that $\theta_1,\ldots,\theta_5$
    is a l.s.o.p. for $S_P$.  
     Also, 
    $[\C \theta_1 + \cdots + \C \theta_5] = [\tM_\C] = \chi_{\reg}$ in $R(G)$, where 
    $\chi_{\reg}$ is the regular representation of $G$. 

\end{example}

We now consider the original counterexample to  \cite[Conjecture~12.1]{StapledonEquivariant} 
given by Santos and Stapledon and appearing in \cite[Theorem~1.2]{EKS22}. We show that it provides a counterexample to
the implication 
\eqref{i:mainnondegenerate} $\Rightarrow$
\eqref{i:mainhilbert} in Theorem~\ref{t:mainfull}, where $P$ does not admit an $G$-invariant lattice triangulation.

\begin{example}\label{e:SantosStapledon}
	Recall that a signed permutation in $n$ variables is a permutation $\pi$ of $\{ \pm 1, \ldots, \pm n \}$ such that $\pi(-i) = -\pi(i)$ for $1 \le i \le n$. 
	Let $B_3$ be the group of signed permuations in $3$ variables. 
	Let $G = \{ \id , \sigma, \tau, \sigma \tau \} \cong \Z/2\Z \times \Z/2\Z$ 
	be the subgroup of $B_3$ with elements $\sigma = (12)(-3)$, $\tau = (1 -2)(-3)$ and $\sigma \tau = \tau \sigma = (-1)(-2)$.  Here we use cycle notation so, for example, $\tau$ has orbits $\{1,-2\}$,$\{-1,2\}$ and $\{3,-3\}$.  
	Consider the affine lattice $M' = \{ (u_1,u_2,u_3) \in \Z^3 : u_1 + u_2 + u_3 = 1 \mod 2 \}$ in $\R^3$ with the corresponding action of $B_3$, and let $P = [-1,1]^3$. Then $P$ is $B_3$-invariant and is a `unit cube' with respect to the lattice $M'$. Geometrically, $\sigma$, $\tau$ and $\sigma \tau$ correspond to 
	 $180$ degree rotations through the lines $\{ x = y, z = 0 \}$,  $\{ x = -y, z = 0 \}$ and $\{ x = y = 0 \}$ respectively. See \cite[Figure~12]{EKS22} for a beautiful illustration. 
	 There are two $G$-orbits of lattice points in $P$:
	 $\{(1,1,1),(1,1,-1),(-1,-1,1),(-1,-1,-1) \}$ and $\{(1,-1,1),(1,-1,-1),(-1,1,1),(-1,1,-1)\}$. Also, $h^*(P,\rho;t) = 1 + \chi_{\reg}t + t^2$, where $\chi_{\reg}$ is the regular representation of $G$ (see the proof of \cite[Theorem~1.2]{EKS22}).
	 
	 Then \cite[Theorem~1.2]{EKS22} implies that 	   there does not exist a  $G$-invariant polynomial that is nondegenerate with respect to $P$. This also follows from Lemma~\ref{l:applysquare} applied to the face $Q$ of $P$ with vertices $(1,1,1),(1,1,-1),(1,-1,1),(1,-1,-1)$. In particular, by Theorem~\ref{t:mainfull}, $P$ does not admit a $G$-invariant regular lattice triangulation. Since all lattice triangulations of the unit cube are regular
	 \cite[Theorem~6.3.10]{DRSTriangulations10}, we deduce that $P$ does not admit a $G$-invariant lattice triangulation.

	  On the other hand, 
	 let $x_{i,j,k}$ denote the degree $1$ variable in $S_P$ corresponding to the vertex $(i,j,k)$ of $P$.
	 Fix a nonzero $\lambda \in \C$ with $\lambda^4 \neq 1$. 
	  Let $$\theta_1 = x_{-1,-1,1} - \lambda x_{1,-1,-1}, 
	 \theta_2 = \sigma \cdot \theta_1 = x_{-1,-1,-1} - \lambda x_{-1,1,1},$$$$
	 \theta_3 = \tau \cdot \theta_1 = x_{1,1,-1} - \lambda x_{1,-1,1},
	\theta_4 = (\sigma \tau) \cdot \theta_1 = x_{1,1,1} - \lambda x_{-1,1,-1}.
	 $$ 
	 If we let $a := x_{-1,-1,1}, b := x_{-1,-1,-1}, c := x_{1,1,-1}, d := x_{1,1,1}$, then one may compute that $S_P/(\theta_1,\theta_2,\theta_3,\theta_4)$ is isomorphic to the polynomial ring in the variables $a,b,c,d$, subject to the relations
	 $a^2 = b^2 = c^2 = d^2 = ad = cb = 0$ and $ab = cd = \lambda^2 ac = \lambda^2 bd$.   In particular, $\theta_1,\theta_2,\theta_3,\theta_4$ is a l.s.o.p. for $S_P$. Also,
	 $[\C \theta_1 + \cdots + \C \theta_4] = [\tM_\C'] = \chi_{\reg}$ in $R(G)$.

Finally, although $P$ does not admit a $G$-invariant lattice triangulation, one may verify that $P$ admits an $H$-invariant lattice triangulation for every cyclic subgroup $H$ of $G$. In fact, this counterexample to \cite[Conjecture~12.1]{StapledonEquivariant} was originally discovered looking for a polytope with this property. The idea being that Theorem~\ref{t:mainfull} guarantees that $h^*(P,\rho;t)$ is a polynomial, but does not guarantee that $h^*(P,\rho;t)$ is effective and hence $P$ is a candidate for a counterexample to the effectiveness conjecture (Conjecture~\ref{c:origmod}). In this particular case $h^*(P,\rho;t)$ turns out to be effective, and the effectiveness conjecture remains open.

%
%
%
%
%
%
%
	 
\end{example}

%

%

%


Our final goal in this section is to give a formula for $h^*(P,\rho;t)$ when $P$ admits a $G$-invariant lattice triangulation.
We first recall some basic properties of Stanley-Reisner rings. We refer the reader to \cite{MillerSturmfelsCombinatorial} for more details.
Let $\Sigma$ be a rational simplicial fan with respect to a lattice $L$. Assume that $\Sigma$ has convex support $|\Sigma|$, and all cones in $\Sigma$ are pointed i.e. do not contain a nonzero linear subspace. Recall from Definition~\ref{d:toricfacering} that one may consider the toric face ring $\C[\Sigma]$ associated to $\Sigma$.
Let 
$u_1,\ldots,u_t$ 
be the primitive integer vectors of the rays of $\Sigma$.
The \emph{Stanley-Reisner ring} $\SR(\Sigma)$ is the subring of $\C[\Sigma]$ 
generated by $\{ x^{u_i} : 1 \le i \le t \}$.  We consider $ x^{u_i}$ to have degree $1$ for $1 \le i \le t$,  so that $\SR(\Sigma)$ has the structure of a graded $\C$-algebra. Reisner's criterion \cite[Theorem~5.53]{MillerSturmfelsCombinatorial} implies that $\SR(\Sigma)$ is Cohen-Macaulay.
Let $F = \sum_{i = 1}^t x^{u_i} \in \SR(\Sigma)$ and consider the elements $\partial_v F = \sum_{i = 1}^t \langle u_i, v \rangle x^{u_i}$ for all $v \in L_\C^* := \Hom_\C(L_\C,\C)$. Then any $\C$-basis of $\{ \partial_v F : v \in L_\C^* \}$ is a l.s.o.p. for $\SR(\Sigma)$, and we may consider the graded finite-dimensional $\C$-algebra
$\overline{\SR}(\Sigma) := \SR(\Sigma)/( \partial_v F : v \in L_\C^* )$. 
Consider the subfan $\Star_\Sigma(C)$ consisting of all faces of all maximal cones of $\Sigma$ that contain $C$. Let $\langle C \rangle$ be the linear span of $C$ in $L_\R$ and consider the projection $p: L_\R \to L_\R/\langle C \rangle$. We also consider the rational fan $\lk_\Sigma(C) = \{ p(C') : C \subset C' \in \Sigma \}$ with respect to the lattice $p(L)$. Then $p$ induces a map of fans from $\Star_\Sigma(C)$ to $\lk_\Sigma(C)$ and a corresponding surjection from $\SR(\Star_\Sigma(C))$ to $\SR(\lk_\Sigma(C))$ that induces an isomorphism of the corresponding quotients $\overline{\SR}(\Star_\Sigma(C)) \cong \overline{\SR}(\lk_\Sigma(C))$. 
Moreover, if $G$ acts linearly on $L$ preserving the fan $\Sigma$, then $\overline{\SR}(\Sigma)$ is naturally a graded $\C G$-module and the isomorphism $\overline{\SR}(\Star_\Sigma(C)) \cong \overline{\SR}(\lk_\Sigma(C))$ is  an isomorphism of $\C G$-modules. Geometrically, if $X_{\lk_\Sigma(C)}$ is the complex toric variety corresponding to the fan $\lk_\Sigma(C)$, then there is a natural isomorphism of $\C$-graded algebras between $\overline{\SR}(\lk_\Sigma(C))$ and the cohomology ring $H^*(X_{\lk_\Sigma(C)}; \C)$, where the degree of $H^{2m}(X_{\lk_\Sigma(C)}; \C)$ is considered to be $m$. 

Let $C$ be a cone in $\Sigma$ with primitive integer vectors $u_{i_1},\ldots,u_{i_s}$. Define
\[
\BBox^\circ(C) := \{ u \in C \cap \tM : u = \sum_{j = 1}^s \lambda_j u_{i_j} \textrm{ for some } 0 < \lambda_j < 1 \}.
\]
Here $\BBox^\circ(C) = \{ 0 \}$ when $C = \{ 0 \}$. 
Then every element of $|\Sigma| \cap L$ can be uniquely written as a sum of an element in $\BBox^\circ(C)$ and a nonnegative integer combination of the primitive integer vectors of the rays of $\Star_\Sigma(C)$ for some $C$ in $\Sigma$. Hence, we have an isomorphism of $\C$-vector spaces:
\[
\C[\Sigma] \cong \bigoplus_{C \in \Sigma} \bigoplus_{u \in \BBox^\circ(C)} x^u \SR(\Star_\Sigma(C)),
\]
inducing an isomorphism of $\C$-vector spaces:
\begin{equation}\label{e:decomposetoricfacering}
\C[\Sigma]/(\partial_v F : v \in L_\C^*) \cong \bigoplus_{C \in \Sigma} \bigoplus_{u \in \BBox^\circ(C)} x^u \overline{\SR}(\lk_\Sigma(C)). 
\end{equation}
As above, if $G$ acts linearly on $L$ preserving the fan $\Sigma$, then the isomorphism in \eqref{e:decomposetoricfacering} is an isomorphism of $\C G$-modules.

Let $\cS$ be a $G$-invariant lattice triangulation of $P$. Recall that 
the fan over the faces of $\cS$ is the fan $\Sigma_\cS = \{ C_F : F \in \cS \}$ that refines $C_P$. Let $\BBox(\cS) := \bigcup_{F \in \cS} \BBox^\circ(C_F)$. 
Then $G$ acts permuting the elements of $\BBox(\cS)$. Let $\BBox(\cS)/G$ denote the set of $G$-orbits of $\BBox(\cS)$. 
For $u \in \BBox^\circ(C_F)$, write $C_u := C_F$ i.e. $C_u$ is the smallest cone in $\Sigma_\cS$ containing $u$, and let $G_u$ be the stablizer of $C_u$ in $G$. Also, let $[u] \in \BBox(\cS)/G$ denote the $G$-orbit of $u$. Recall that $\pr: \tM \to \Z$ denotes projection onto the last coordinate. 
Given a subgroup $H$ of $G$ and an element $r \in R(H)$, there is well-defined element $\Ind_H^G r \in R(G)$ such that if $r = [W]$ for a  $\C H$-module $W$, then
$\Ind_H^G r = [\Ind_H^G W]$. If $r(t) = \sum_{m \ge 0} r_m t^m \in R(H)[[t]]$ then we write $\Ind_H^G r(t) := \sum_{m \ge 0} (\Ind_H^G r_m) t^m \in R(G)[[t]]$.
We now prove  the following formula for $h^*(P,\rho;t)$ that generalizes 
known formulas when $P$ is a simplex
\cite[Proposition~6.1]{StapledonEquivariant} or when $\Sigma_\cS$ is a smooth fan i.e. $\cS$ is unimodular
\cite[Proposition~8.1]{StapledonEquivariant}. The case when $G$ is trivial is due to 
Betke and McMullen \cite[Theorem~1]{BMLatticePoints}. 
Geometrically, the proposition below states that $h^*(P,\rho;t)$  is the equivariant Hilbert polynomial of the orbifold cohomology ring of the toric variety $X_{\Sigma_\cS}$ associated to $\Sigma_{\cS}$ (see, for example, \cite[(1.1)]{KaruEhrhart}).

\begin{proposition}\label{p:formulatriangulation}
				Let $G$ be a finite group.
	Let $M$ be a lattice of rank $d$ and let $\rho: G \to  \Aff(M)$ be an affine representation.  
	Let $P \subset M_\R$ be a $G$-invariant $d$-dimensional lattice polytope. 
	Suppose there exists a $G$-invariant lattice triangulation $\cS$  of $P$. 
	 Then, with the notation above, 
	\[
	h^*(P,\rho;t) = \sum_{[u] \in \BBox(\cS)/G} t^{\pr(u)} \Ind_{G_u}^G \Hilb_{G_u}(\overline{\SR}(\lk_{\Sigma_\cS}(C_{u}));t).
	\]
\end{proposition}
\begin{proof}
	Recall that $S_\cS$ denotes the toric face ring $\C[\Sigma_\cS]$. By Lemma~\ref{l:invarianttriangulation}, if $A$ is the set of vertices of $\cS$ and $f = \sum_{a \in A} x^a \in \C[M]$, then $h^*(P,\rho;t) = \Hilb_G(S_{\cS}/J_{f,\cS};t)$.
%
%
	Applying \eqref{e:decomposetoricfacering} gives an isomorphism of graded $\C G$-modules:
	\begin{equation}\label{e:SoverJdecomp}
		S_\cS/J_{f,\cS} \to \bigoplus_{F \in \cS} \bigoplus_{u \in \BBox^\circ(C_F)} x^u \overline{\SR}(\lk_{\Sigma_\cS}(C_F))). 
	\end{equation}
	On the right hand side, for any $u \in \BBox^\circ(C_F)$, $\bigoplus_{u' \in [u]} x^{u'} \overline{\SR}(\lk_{\Sigma_\cS}(C_{u'})) 
	$ is isomorphic as a graded $\C G$-module to the induced representation $\Ind_{G_u}^G x^u \overline{\SR}(\lk_{\Sigma_\cS}(C_u))$. The result follows by taking equivariant Hilbert series of both sides of \eqref{e:SoverJdecomp}.

\end{proof}

\begin{remark}
	A corollary of the proof of Proposition~\ref{p:formulatriangulation} is that, with the notation of Proposition~\ref{p:formulatriangulation},  if  there exists a $G$-invariant lattice triangulation $\cS$  of $P$, then
	\[
	\Ehr(P,\rho;t) = \sum_{[u] \in \BBox(\cS)/G} t^{\pr(u)} \Ind_{G_u}^G \Hilb_{G_u}(\SR(\Star_\Sigma(C_u));t).
	\]
	In the case when $\cS$ is a unimodular triangulation, the right hand side is 
	$\Hilb_{G}(\SR(\Sigma_\cS);t)$ (cf. the proof of \cite[Proposition~8.1]{StapledonEquivariant} and \cite[Theorem~5.2]{DDEquivariantEhrhart}). 
\end{remark}

\bibliography{Equivariant20230202}

\providecommand{\bysame}{\leavevmode\hbox to3em{\hrulefill}\thinspace}
\providecommand{\MR}{\relax\ifhmode\unskip\space\fi MR }
\providecommand{\MRhref}[2]{%
  \href{http://www.ams.org/mathscinet-getitem?mr=#1}{#2}
}
\providecommand{\href}[2]{#2}
\begin{thebibliography}{ASVM21}

\bibitem[AM69]{AtiyahMacdonald}
M.~F. Atiyah and I.~G. Macdonald, \emph{Introduction to commutative algebra},
  Addison-Wesley Publishing Co., Reading, Mass.-London-Don Mills, Ont., 1969.

\bibitem[AN99]{ACampoGeneralizedHVectors}
A.~A'Campo-Neuen, \emph{On generalized {$h$}-vectors of rational polytopes with
  a symmetry of prime order}, Discrete Comput. Geom. \textbf{22} (1999), no.~2,
  259--268.

\bibitem[ASVM20]{ASV20}
Federico Ardila, Mariel Supina, and Andr\'{e}s~R. Vindas-Mel\'{e}ndez,
  \emph{The equivariant {E}hrhart theory of the permutahedron}, Proc. Amer.
  Math. Soc. \textbf{148} (2020), no.~12, 5091--5107.

\bibitem[ASVM21]{ASV21}
Federico Ardila, Anna Schindler, and Andr\'{e}s~R. Vindas-Mel\'{e}ndez,
  \emph{The equivariant volumes of the permutahedron}, Discrete Comput. Geom.
  \textbf{65} (2021), no.~3, 618--635.

\bibitem[Bat93]{BatyrevVariations}
Victor~V. Batyrev, \emph{Variations of the mixed {H}odge structure of affine
  hypersurfaces in algebraic tori}, Duke Math. J. \textbf{69} (1993), no.~2,
  349--409.

\bibitem[BBR07]{BBRCohomology07}
Morten Brun, Winfried Bruns, and Tim R\"{o}mer, \emph{Cohomology of partially
  ordered sets and local cohomology of section rings}, Adv. Math. \textbf{208}
  (2007), no.~1, 210--235.

\bibitem[BD24]{BDSupersolvable}
Christin Bibby and Emanuele Delucchi, \emph{Supersolvable posets and fiber-type
  abelian arrangements}, Selecta Math. (N.S.) \textbf{30} (2024), no.~5, Paper
  No. 89, 39.

\bibitem[Ber09]{BergerGeometry}
Marcel Berger, \emph{Geometry {I}}, Universitext, Springer-Verlag, Berlin,
  2009, Translated from the 1977 French original by M. Cole and S. Levy, Fourth
  printing of the 1987 English translation.

\bibitem[BHW07]{BHWRotesEhrhart}
Christian Bey, Martin Henk, and J\"{o}rg~M. Wills, \emph{Notes on the roots of
  {E}hrhart polynomials}, Discrete Comput. Geom. \textbf{38} (2007), no.~1,
  81--98.

\bibitem[BJM13]{BJMLatticePointGenerating}
Matthias Beck, Pallavi Jayawant, and Tyrrell~B. McAllister, \emph{Lattice-point
  generating functions for free sums of convex sets}, J. Combin. Theory Ser. A
  \textbf{120} (2013), no.~6, 1246--1262.

\bibitem[BKN23]{BKNThinPolytopes}
Christopher Borger, Andreas Kretschmer, and Benjamin Nill, \emph{Thin
  polytopes: Lattice polytopes with vanishing local {$h^*$}-polynomial}, Int.
  Math. Res. Not. IMRN (2023).

\bibitem[BM85]{BMLatticePoints}
U.~Betke and P.~McMullen, \emph{Lattice points in lattice polytopes}, Monatsh.
  Math. \textbf{99} (1985), no.~4, 253--265.

\bibitem[BM03]{BorisovMavlyutov03}
Lev~A. Borisov and Anvar~R. Mavlyutov, \emph{String cohomology of
  {C}alabi-{Y}au hypersurfaces via mirror symmetry}, Adv. Math. \textbf{180}
  (2003), no.~1, 355--390.

\bibitem[BR15]{BeckRobinsComputing}
Matthias Beck and Sinai Robins, \emph{Computing the continuous discretely},
  second ed., Undergraduate Texts in Mathematics, Springer, New York, 2015,
  Integer-point enumeration in polyhedra, With illustrations by David Austin.

\bibitem[CHK23]{CHK23}
Oliver Clarke, Akihiro Higashitani, and Max K\"{o}lbl, \emph{The equivariant
  {E}hrhart theory of polytopes with order-two symmetries}, Proc. Amer. Math.
  Soc. \textbf{151} (2023), no.~9, 4027--4041.

\bibitem[DD21]{DDStanleyReisner}
Alessio D'Al\`{i} and Emanuele Delucchi, \emph{Stanley-{R}eisner rings for
  symmetric simplicial complexes, {$G$}-semimatroids and {A}belian
  arrangements}, J. Comb. Algebra \textbf{5} (2021), no.~3, 185--236.

\bibitem[DD23]{DDEquivariantEhrhart}
Alessio D'al\`{i} and Emanuele Delucchi, \emph{Equivariant {H}ilbert and
  {E}hrhart series under translative group actions}, preprint
  arXiv:2312.14088v2, 2023.

\bibitem[DLRS10]{DRSTriangulations10}
J.~De~Loera, J.~Rambau, and F.~Santos, \emph{Triangulations}, Algorithms and
  Computation in Mathematics, vol.~25, Springer-Verlag, Berlin, 2010,
  Structures for algorithms and applications.

\bibitem[DR18]{DRSemimatroids}
Emanuele Delucchi and Sonja Riedel, \emph{Group actions on semimatroids}, Adv.
  in Appl. Math. \textbf{95} (2018), 199--270.

\bibitem[Dwo62]{DworkZetaI}
Bernard Dwork, \emph{On the zeta function of a hypersurface}, Inst. Hautes
  \'{E}tudes Sci. Publ. Math. (1962), no.~12, 5--68.

\bibitem[Dwo64]{DworkZetaII}
\bysame, \emph{On the zeta function of a hypersurface. {II}}, Ann. of Math. (2)
  \textbf{80} (1964), 227--299.

\bibitem[Ehr77]{EhrhartPolynomes}
E.~Ehrhart, \emph{Polyn\^{o}mes arithm\'{e}tiques et m\'{e}thode des
  poly\`edres en combinatoire}, International Series of Numerical Mathematics,
  vol. Vol. 35, Birkh\"{a}user Verlag, Basel-Stuttgart, 1977.

\bibitem[Eis95]{EisenbudCommutativeAlgebra}
David Eisenbud, \emph{Commutative algebra}, Graduate Texts in Mathematics, vol.
  150, Springer-Verlag, New York, 1995, With a view toward algebraic geometry.

\bibitem[EKS23]{EKS22}
Sophia Elia, Donghyun Kim, and Mariel Supina, \emph{Techniques in equivariant
  {E}hrhart theory}, Ann. Comb. (2023).

\bibitem[GKZ94]{GKZ94}
I.~M. Gelfand, M.~M. Kapranov, and A.~V. Zelevinsky, \emph{Discriminants,
  resultants, and multidimensional determinants}, Mathematics: Theory \&
  Applications, Birkh\"{a}user Boston, Inc., Boston, MA, 1994.

\bibitem[Hoc72]{HochsterRingsInvariants}
Melvin Hochster, \emph{Rings of invariants of tori, {C}ohen-{M}acaulay rings
  generated by monomials, and polytopes}, Ann. of Math. (2) \textbf{96} (1972),
  318--337.

\bibitem[HT09]{HTLowerBounds}
Martin Henk and Makoto Tagami, \emph{Lower bounds on the coefficients of
  {E}hrhart polynomials}, European J. Combin. \textbf{30} (2009), no.~1,
  70--83.

\bibitem[IR07]{IRToric07}
Bogdan Ichim and Tim R\"{o}mer, \emph{On toric face rings}, J. Pure Appl.
  Algebra \textbf{210} (2007), no.~1, 249--266.

\bibitem[Kar08]{KaruEhrhart}
Kalle Karu, \emph{Ehrhart analogue of the {$h$}-vector}, Integer points in
  polyhedra---geometry, number theory, representation theory, algebra,
  optimization, statistics, Contemp. Math., vol. 452, Amer. Math. Soc.,
  Providence, RI, 2008, pp.~139--146.

\bibitem[MS05]{MillerSturmfelsCombinatorial}
Ezra Miller and Bernd Sturmfels, \emph{Combinatorial commutative algebra},
  Graduate Texts in Mathematics, vol. 227, Springer-Verlag, New York, 2005.

\bibitem[Rei02]{Reiner02}
Victor Reiner, \emph{Equivariant fiber polytopes}, Doc. Math. \textbf{7}
  (2002), 113--132.

\bibitem[Ser77]{SerreLinearRepresentations}
Jean-Pierre Serre, \emph{Linear representations of finite groups}, french ed.,
  Graduate Texts in Mathematics, vol. Vol. 42, Springer-Verlag, New
  York-Heidelberg, 1977.

\bibitem[Sta76]{StanleyMagic}
Richard~P. Stanley, \emph{Magic labelings of graphs, symmetric magic squares,
  systems of parameters, and {C}ohen-{M}acaulay rings}, Duke Math. J.
  \textbf{43} (1976), no.~3, 511--531.

\bibitem[Sta80]{StanleyDecompositions}
\bysame, \emph{Decompositions of rational convex polytopes}, Ann. Discrete
  Math. \textbf{6} (1980), 333--342.

\bibitem[Sta84]{StanleyIntroductionCombinatorialCommutative}
\bysame, \emph{An introduction to combinatorial commutative algebra},
  Enumeration and design ({W}aterloo, {O}nt., 1982), Academic Press, Toronto,
  ON, 1984, pp.~3--18.

\bibitem[Sta87]{StanleyGeneralized87}
\bysame, \emph{Generalized {$H$}-vectors, intersection cohomology of toric
  varieties, and related results}, Commutative algebra and combinatorics
  ({K}yoto, 1985), Adv. Stud. Pure Math., vol.~11, North-Holland, Amsterdam,
  1987, pp.~187--213.

\bibitem[Sta93]{StanleyMonotonicity}
\bysame, \emph{A monotonicity property of {$h$}-vectors and {$h^*$}-vectors},
  European J. Combin. \textbf{14} (1993), no.~3, 251--258.

\bibitem[Sta08]{StapledonWeighted}
Alan Stapledon, \emph{Weighted {E}hrhart theory and orbifold cohomology}, Adv.
  Math. \textbf{219} (2008), no.~1, 63--88.

\bibitem[Sta11a]{StapledonEquivariant}
\bysame, \emph{Equivariant {E}hrhart theory}, Adv. Math. \textbf{226} (2011),
  no.~4, 3622--3654.

\bibitem[Sta11b]{StapledonRepresentations11}
\bysame, \emph{Representations on the cohomology of hypersurfaces and mirror
  symmetry}, Adv. Math. \textbf{226} (2011), no.~6, 5268--5297.

\bibitem[Sta12]{StanleyEnumerative}
Richard~P. Stanley, \emph{Enumerative combinatorics. {V}olume 1}, second ed.,
  Cambridge Studies in Advanced Mathematics, vol.~49, Cambridge University
  Press, Cambridge, 2012.

\bibitem[Ste94]{StembridgeSomePermutation}
John~R. Stembridge, \emph{Some permutation representations of {W}eyl groups
  associated with the cohomology of toric varieties}, Adv. Math. \textbf{106}
  (1994), no.~2, 244--301.

\bibitem[Stu94]{SturmfelsNewtonPolytope}
Bernd Sturmfels, \emph{On the {N}ewton polytope of the resultant}, J. Algebraic
  Combin. \textbf{3} (1994), no.~2, 207--236.

\bibitem[Var76]{Varchenko76}
A.~Varchenko, \emph{Zeta-function of monodromy and {N}ewton's diagram}, Invent.
  Math. \textbf{37} (1976), no.~3, 253--262.

\bibitem[Vos65]{VosOnTwoDimensional}
V.~E. Voskresenski\u{\i}, \emph{On two-dimensional algebraic tori}, Izv. Akad.
  Nauk SSSR Ser. Mat. \textbf{29} (1965), 239--244.

\bibitem[Zie95]{ZieglerLectures}
G\"{u}nter~M. Ziegler, \emph{Lectures on polytopes}, Graduate Texts in
  Mathematics, vol. 152, Springer-Verlag, New York, 1995.

\end{thebibliography}
\bibliographystyle{amsalpha}

\end{document}